\documentclass[reqno, 11pt]{amsart}

\usepackage{amsmath}
\usepackage{amssymb}
\usepackage{amsthm}
\usepackage{verbatim}
\usepackage[all]{xy}
\usepackage[usenames]{color}
\usepackage{amsfonts}
\usepackage{latexsym}
\usepackage{amscd}
\usepackage{graphicx}
\usepackage[usenames,dvipsnames]{xcolor}
\usepackage{tikz}
\usepackage{tikz-cd}
\usepackage{todonotes}
\usepackage{hyperref}
\usepackage{mathrsfs}
\usepackage{euscript}

\usepackage[capitalize]{cleveref}
\crefname{theorem}{Theorem}{Theorems}
\crefname{fact}{Fact}{Facts}
\crefname{note}{Note}{Notes}
\crefname{lemma}{Lemma}{Lemmas}
\crefname{alg}{Algorithm}{Algorithms}
\crefname{remark}{Remark}{Remarks}
\crefname{example}{Example}{Examples}
\crefname{prop}{Proposition}{Propositions}
\crefname{conj}{Conjecture}{Conjectures}
\crefname{cor}{Corollary}{Corollaries}
\crefname{defn}{Definition}{Definitions}
\crefname{equation}{\!\!}{\!\!} 

\usepackage[left=1in, top=1in, right=1in, bottom=1in]{geometry}

\usetikzlibrary{decorations.markings}
\usetikzlibrary{decorations.pathreplacing}
\usetikzlibrary{arrows,shapes,positioning}
\tikzstyle directed=[postaction={decorate,decoration={markings, mark=at position #1 with {\arrow[scale=1]{>}}}}]
\tikzstyle rdirected=[postaction={decorate,decoration={markings, mark=at position #1 with {\arrow[scale=1]{<}}}}]
\usepackage[all]{xy}
\usepackage{ytableau}
\newcommand{\clr}{rgb:black,1;blue,4;red,1}
\newcommand{\wdot}{node[circle, draw, color=\clr, fill=white, thick, inner sep=0pt, minimum width=4.5pt]{}}

\makeatletter

\newtheorem{theorem}{Theorem}[section]

\newtheorem{lemma}[theorem]{Lemma}
\newtheorem{proposition}[theorem]{Proposition}
\newtheorem{corollary}[theorem]{Corollary} 
\newtheorem{definition}[theorem]{Definition}

\newtheorem{remark}[theorem]{Remark}

\@addtoreset{equation}{section}

\newcommand{\End}{\operatorname{End}}
\newcommand{\Hom}{\operatorname{Hom}}
\newcommand{\merge}{\operatorname{merge}}
\renewcommand{\split}{\operatorname{split}}

\newcommand{\Z}{\mathbb{Z}}
\newcommand{\C}{\mathbb{C}}

\newcommand{\onto}{\twoheadrightarrow}
\newcommand{\into}{\hookrightarrow}
\newcommand{\up}{\uparrow}

\newcommand{\q}{\mathfrak{q}}

\newcommand{\frakS}{\mathfrak{S}}

\newcommand{\Sdot}{\mathscr{S}}
\newcommand{\W}{W}

\newcommand{\beq}{\begin{equation}}
\newcommand{\eeq}{\end{equation}}
\newcommand{\bea}{\begin{eqnarray*}}
\newcommand{\eea}{\end{eqnarray*}}
\newcommand{\bi}{\begin{itemize}}
\newcommand{\ei}{\end{itemize}}
\newcommand{\be}{\begin{enumerate}}
\newcommand{\ee}{\end{enumerate}}
\newcommand{\bc}{\begin{center}}
\newcommand{\ec}{\end{center}}
\newcommand{\bg}{\begin{gathered}}
\newcommand{\eg}{\end{gathered}}
\newcommand{\bt}{\begin{tikzpicture}}
\newcommand{\et}{\end{tikzpicture}}
\newcommand{\ol}{\overline}
\newcommand{\ul}{\underline}

\newcommand{\bma}{\begin{bmatrix}}
\newcommand{\ema}{\end{bmatrix}}
\newcommand{\squig}{\rightsquigarrow}
\newcommand{\arxiv}[1]{\href{http://arxiv.org/abs/#1}{\tt arXiv:\nolinkurl{#1}}}
\renewcommand{\ss}{\scriptsize}

\renewcommand{\a}{\mathsf{a}}
\renewcommand{\b}{\mathsf{b}}
\renewcommand{\c}{\mathsf{c}}
\newcommand{\A}{\mathsf{A}}
\newcommand{\fs}{\footnotesize}
\renewcommand{\H}{\mathscr{H}}
\newcommand{\nfrac}[2]{\genfrac{}{}{0pt}{}{#1}{#2}}
\renewcommand{\i}{\text{\boldmath$i$}}
\renewcommand{\th}{\text{th}}
\newcommand{\st}{\text{st}}

\newcommand{\T}{\mathscr{T}}
\newcommand{\0}{\bar{0}}
\renewcommand{\1}{\bar{1}}

\newcommand{\s}{\mathcal{S}}
\newcommand{\inftyspiderup}{\mathfrak{q}\text{-}\mathbf{Web}_{\up}}

\begin{document}

\title{Webs for permutation supermodules of type Q}

\author{Gordon C. Brown }
\address{Fort Worth, TX}
\email{gcbrown\,@\,utexas.edu}

\date{\today}

\begin{abstract}
We give a diagrammatic calculus for the intertwiners between permutation supermodules over the Sergeev superalgebra. We also give a diagrammatic basis for the space of intertwiners between two permutation supermodules.
\end{abstract}

\maketitle  



\section{Introduction}

This paper may be thought of as a companion to the author's recent work with Kujawa on webs of type Q \cite{BKu}. In that paper, the authors generalized the $\mathfrak{sl_n}$-webs of Kuperberg \cite{Ku} to give a combinatorial model of a category of modules over $\q_n$, the Lie superalgebra of type Q. The starting point was a pair of mutually centralizing actions (in that case between $\q_m$ and $\q_n$), which, since Cautis-Kamnitzer-Morrison's groundbreaking work in 2014 \cite{CKM}, are known to lead to diagrammatic presentations of certain module categories. See \cite{RT,QS,ST,TVW} for other recent examples of webs in representation theory.

In the present work, we apply similar techniques to the mutually centralizing actions arising in the Schur-Weyl-Sergeev duality  \cite{Se} between $\q_n$ and the Sergeev superalgebra $\H_c(r)$ (known elsewhere as the Hecke-Clifford superalgebra). This time, the module category being described consists of the permutation supermodules $M^\lambda$ of $\H_c(r)$, for $\lambda$ a composition of $r$. Notably, the combinatorial model used here is a ``submodel" of that appearing in \cite{BKu}; we will make this precise shortly.

\subsection{Webs of type Q}

Let us briefly recall some of the combinatorics developed in \cite{BKu}. The reader is referred there for further details.

 Let $\a=(\a_1,\dots,\a_m)$ and $\b=(\b_1,\dots,\b_n)$ be finite sequences of positive integers such that $|\a|:=\sum_{i=1}^m \a_i$ and $|\b|:=\sum_{i=1}^n \b_i$ are equal. A(n unoriented) \emph{web of type Q from $\a\to\b$} is a kind of decorated, trivalent graph, with lower boundary labeled by $\a$ and upper boundary labeled by $\b$. For example, the following is a web from $(4,9,6,7)\to(6,5,1,4,8,2)$:
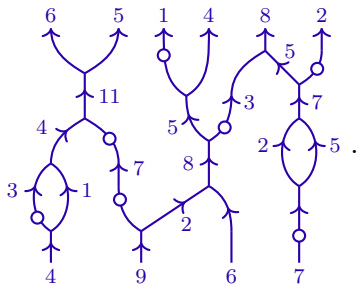
\[
\xy
(0,0)*{
\bt[color=\clr, scale=1.2]
	\node at (2,0) {\scriptsize $4$};
	\node at (3,0) {\scriptsize $9$};
	\node at (4,0) {\scriptsize $6$};
	\node at (4.75,0) {\scriptsize $7$};
	\node at (2,2.9) {\scriptsize $6$};
	\node at (2.75,2.9) {\scriptsize $5$};
	\node at (3.25,2.9) {\scriptsize $1$};
	\node at (3.75,2.9) {\scriptsize $4$};
	\node at (4.375,2.9) {\scriptsize $8$};
	\node at (5,2.9) {\scriptsize $2$};
	\draw [thick, directed=0.65] (2,0.15) to (2,0.5);
	\draw [thick, directed=0.65] (3,0.15) to (3,0.5);
	\draw [thick, ] (4,0.15) to (4,0.5);
	\draw [thick, directed=0.75] (4.75,0.15) to (4.75,1);
	\draw [thick, directed=0.65] (2,0.5) [out=30, in=330] to (2,1.25);
	\draw [thick, directed=0.65] (2,0.5) [out=150, in=210] to (2,1.25);
	\draw [thick, directed=0.65] (2,1.25) [out=90, in=210] to (2.375,1.75);
	\draw [thick, ] (3,0.5) [out=150, in=270] to (2.75,1);
	\draw [thick, directed=0.65] (3,0.5) to (3.75,1);
	\draw [thick, ] (2.75,1) to (2.75,1.25);
	\draw [thick, directed=1] (3.5,2) [out=30, in=270] to (3.75,2.75);
	\draw [thick, directed=1] (3.5,2) [out=150, in=270] to (3.25,2.75);
	\draw [thick, directed=0.01] (2.75,1.25) [out=90, in=330] to (2.375,1.75);
	\draw [thick, directed=0.65] (2.375,1.75) to (2.375,2.25);
	\draw [thick, directed=1] (2.375,2.25) [out=30, in=270] to (2.75,2.75);
	\draw [thick, directed=1] (2.375,2.25) [out=150, in=270] to (2,2.75);
	\draw [thick, directed=0.45] (4,0.5) [out=90, in=330] to (3.75,1);
	\draw [thick, directed=0.65] (3.75,1) to (3.75,1.5);
	\draw [thick, directed=0.65] (3.75,1.5) [out=150, in=270] to (3.5,2);
	\draw [thick, ] (3.75,1.5) [out=30, in=270] to (4,2);
	\draw [thick, directed=0.01] (4,2) [out=90, in=210] to (4.375,2.5);
	\draw [thick, directed=0.65] (4.75,1) [out=30, in=330] to (4.75,1.75);
	\draw [thick, directed=0.65] (4.75,1) [out=150, in=210] to (4.75,1.75);
	\draw [thick, directed=0.65] (4.75,1.75) to (4.75,2.125);
	\draw [thick, directed=0.65] (4.75,2.125) to (4.375,2.5);
	\draw [thick, ] (4.75,2.125) [out=30, in=270] to (5,2.5);
	\draw [thick, directed=1] (5,2.5) to (5,2.75);
	\draw [thick, directed=1] (4.375,2.5) to (4.375,2.75);
	\node at (1.575,0.95) {\scriptsize $3$}; 
	\node at (2.4,0.95) {\scriptsize $1$}; 
	\node at (2.975,1.2) {\scriptsize $7$}; 
	\node at (2.65,2) {\scriptsize $11$}; 
	\node at (1.9,1.65) {\scriptsize $4$}; 
	\node at (3.5,0.575) {\scriptsize $2$}; 
	\node at (3.525,1.25) {\scriptsize $8$}; 
	\node at (3.35,1.65) {\scriptsize $5$}; 
	\node at (4.2,1.925) {\scriptsize $3$}; 
	\node at (4.95,1.925) {\scriptsize $7$}; 
	\node at (5.15,1.45) {\scriptsize $5$}; 
	\node at (4.35,1.45) {\scriptsize $2$}; 
	\node at (4.65,2.5) {\scriptsize $5$}; 
	\draw (1.85,0.65) \wdot;
	\draw (2.765,0.85) \wdot;
	\draw (2.65,1.525) \wdot;
	\draw (4.75,0.45) \wdot;
	\draw (4.95,2.3) \wdot;
	\draw (3.925,1.65) \wdot;
	\draw (3.25,2.45) \wdot;
\et
};
\endxy.
\]
In \cite{BKu}, the arrowheads were necessary to distinguish between a module and its dual. While there is no such need in this paper (i.e. all arrowheads will be upward-oriented), we leave them in place for consistency, as well as to remind the reader of our chosen convention to read webs from bottom to top. Under this convention, each vertex of the above web consists of either a \emph{merge} of two edges into one, or a \emph{split} of one edge into two. Thinking of the numerical labels as indicating the ``thickness" of an edge, one sees that merges and splits preserve thickness. Finally, on some edges lie one or more circular \emph{dots}, which leave thicknesses unchanged.\footnote{As arrowheads and thicknesses are simply labels, there is no significance in their exact locations relative to each other or to merges, splits, and dots.} 

We are justified in referring to $|\a|=|\b|$ as the \emph{total thickness} of a web from $\a\to\b$, e.g. the above web has total thickness $26.$ We put a $\Z_2$-grading on webs by declaring the parity of a web to be the number of dots modulo 2.

Let $\C$ be the field of complex numbers. The category $\inftyspiderup$ has as objects all finite sequences of positive integers. The morphism space $\Hom_{\inftyspiderup}(\a,\b)$ is the $\Z_2$-graded $\C$-vector space spanned by equivalence classes of webs of type Q from $\a\to\b$ modulo certain relations, where we declare $\Hom_{\inftyspiderup}(\a,\b)=0$ if $|\a|\neq|\b|$. Composition is given by vertical concatenation of webs, i.e. for webs $w_1,w_2$ we have
\[
w_1\circ w_2 \ := \ 
\xy
(0,0)*{
\begin{tikzpicture}[scale=.35, color=\clr]
	\draw [ thick, ] (0,-5) to (0,-4);
	\draw [ thick, ] (2,-5) to (2,-4);
	\draw [ thick,  directed=1] (0,-2) to (0,-1);
	\draw [ thick,  directed=1] (2,-2) to (2,-1);
	\draw [ thick] (-0.3,-4) rectangle (2.3,-2);
	\node at (1,-3) { $w_1$};
	\draw [ thick, ] (0,-8) to (0,-7);
	\draw [ thick, ] (2,-8) to (2,-7);
	\draw [ thick] (-0.3,-7) rectangle (2.3,-5);
	\node at (1,-1.65) { \,$\cdots$};
	\node at (1,-4.6) { \,$\cdots$};
	\node at (1,-7.6) { \,$\cdots$};
	\node at (1,-6) { $w_2$};
\end{tikzpicture}
};
\endxy \ .
\]
We declare $w_1\circ w_2=0$ if the lower boundary of $w_1$ and the upper boundary of $w_2$ are different. There is also a tensor product structure on $\inftyspiderup$, given, roughly speaking, by horizontal concatenation. 
 Ultimately, $\inftyspiderup$ forms a \emph{monoidal supercategory} in the sense of Brundan-Ellis \cite{BE}. 

\subsection{From categories to algebras}

One of the main results of \cite{BKu} (Theorem 7.1) is that a certain quotient of $\inftyspiderup$ is monoidally superequivalent to the full subcategory of $\q_n$-supermodules tensor-generated by the supersymmetric powers $\s^k(\C^{n|n})$ of the natural $\q_n$-supermodule $\C^{n|n},$ $k\geq 0$.\footnote{The categories also admit symmetric braidings, and the equivalence extends to them.} In particular, the functor maps an object $\a=(\a_1,\dots,\a_m)$ to the tensor product
\[
\s^{\,\a_1}(\C^{n|n})\otimes\cdots\otimes\s^{\,\a_m}(\C^{n|n}).
\]

In this paper, we fix $r>0$ and consider the full subcategory of $\inftyspiderup$ whose objects are all sequences $\a$ with $|\a|=r$ (whence ``submodel"). The tensor product operation of $\inftyspiderup$ does not restrict to this subcategory, and as such we dispense with the categorical setting of categories and functors in favor of the algebraic setting of algebras and homomorphisms. Namely, we view the aforementioned subcategory as the superalgebra
\[
\W=\W(r):=\bigoplus_{\substack{\a,\b\\ |\a|=|\b|=r}}\Hom_{\inftyspiderup}(\a,\b),
\]
with multiplication given by composition of morphisms. In Section \ref{web-sec} we give a presentation of $\W$ by generators and relations. 

\subsection{Sergeev algebra and permutation supermodules}

The \emph{Sergeev superalgebra} $\H=\H_c(r)$ is isomorphic as a vector space to $C_r\otimes\C\frakS_r$, where $C_r$ is the Clifford superalgebra of rank $r$ and $\C\frakS_r$ is the symmetric group algebra on $r$ letters. The relations are such that both $C_r$ and $\C\frakS_r$ are subalgebras. See Definition \ref{sergeev-def} for the $\Z_2$-grading on $\H_c(r)$ and other details.

The discovery of $\H$ is due to Sergeev \cite{Se}, who established the existence of mutually centralizing actions between $\H$ and the Lie superalgebra $\q_n$ on the $r$-fold tensor product $(\C^{n|n})^{\otimes r}.$ This result is now known as \emph{Schur-Weyl-Sergeev duality}. On the other hand, it is now known (see e.g. \cite[Corollary 3.5]{BKl}) that $\H$ is Morita superequivalent to the \emph{twisted group algebra} $\EuScript{T}_r$ of $\frakS_r.$ The latter arises in the study of \emph{projective representations} of $\frakS_r$, i.e. homomorphisms from $\frakS_r$ into a projective general linear group. The Morita superequivalence implies that studying ordinary representations of $\H$ is equivalent to studying projective representations of $\frakS_r$, giving us extra incentive to understand the representation theory of $\H$.

Let $\lambda$ be a \emph{composition of $r$}, i.e. a finite sequence $\lambda=(\lambda_1,\dots,\lambda_n)$ of \emph{nonnegative} integers with $|\lambda|=r$. Associated to $\lambda$ is the \emph{Young subgroup} $\frakS_\lambda:=\frakS_{\lambda_1}\times\cdots\times\frakS_{\lambda_n}\subseteq\frakS_r.$ The \emph{permutation supermodule of shape $\lambda$} is the left $\H$-supermodule
\[
M^\lambda:=\H\otimes_{\C\frakS_\lambda}\text{Triv}_\lambda
\]
where Triv$_\lambda$ is the trivial $\frakS_\lambda$-module. To the author's knowledge, these modules first appeared in \cite{Se2}; they have since appeared in, for example, \cite{DW2,WW}. In Section \ref{perm-mods} we give an alternative, but isomorphic, definition of $M^\lambda.$

The permutation modules of $\C\frakS_r$, i.e. the left $\C\frakS_r$-modules of the form $\C\frakS_r\otimes_{\C\frakS_\lambda}\text{Triv}_\lambda$, are ubiquitous in the literature, as they were used by Specht to construct the irreducible representations of $\frakS_r$ in characteristic zero \cite{Sp}. As is well known, they admit a combinatorial interpretation as the span of of all \emph{tabloids of shape $\lambda$}. We extend this idea to $M^\lambda$ by defining \emph{super}tabloids of shape $\lambda$ (Definition \ref{supertab-def}), and showing that their span has the structure of an $\H$-supermodule isomorphic to $M^\lambda$. The idea of ``supertableaux" with entries coming from a signed alphabet is not new, however; see for example \cite{BR,CPT,H,LNS,Mu,St}.

\subsection{Main results}

This paper has two main results. The first is Theorem \ref{psi-map}, and consists of an isomorphism of superalgebras
\[
\psi\colon\W\xrightarrow{\sim}\End_{\H}(\textstyle\bigoplus_{\lambda\in\Lambda'(r)}M^\lambda).
\]
Here $\Lambda'(r)$ is a certain proper subset of the set of all compositions of $r$, although every isomorphism class of permutation supermodule has at least one representative in $\Lambda'(r)$. Moreover, $\psi$ induces superspace isomorphisms
\[
\psi_{\a,\b}\colon\Hom_{\inftyspiderup}(\a,\b)\xrightarrow{\sim}\Hom_{\H}(M^{\a},M^{\b})
\]
with the property that $\psi_{\a,\c}(g\circ f)=\psi_{\b,\c}(g)\circ\psi_{\a,\b}(f)$ for all $f\in\Hom_{\inftyspiderup}(\a,\b)$ and $g\in\Hom_{\inftyspiderup}(\b,\c).$ Now the entries of $\a$ and $\b$, i.e. the lower and upper boundaries of webs from $\a\to\b$, correspond to the rows of supertabloids. This affords a combinatorial description of the $\H$-morphisms\footnote{ For a superalgebra $A$, we shorten ``homomorphism of $A$-supermodules" to ``$A$-morphism".} between permutation supermodules.

At its core, Theorem \ref{psi-map} is a synthesis of two basic observations. The first is that the superalgebra $\Sdot_c(n,r):=\End_{\H}((\C^{n|n})^{\otimes r})$, the \emph{Schur superalgebra of type Q}, is isomorphic to both a quotient of $U(\q_n)$, the universal enveloping algebra of $\q_n$ (by Schur-Weyl-Sergeev duality), and to an endomorphism algebra of a sum of permutation supermodules much like the target of $\psi$ (by an easy computation). This object was studied by Brundan-Kleshchev \cite{BKl}, and has been presented by generators and relations as a quotient of $U(\q_n)$ by Du-Wan \cite{DW1} using ideas that go back to Doty-Giaquinto in type A \cite{DG}. The second observation is that an idempotented presentation of $\Sdot_c(n,r)$ \cite[Theorem 4.10]{DW1}, viewing it as a quotient of the idempotented version $\dot{U}(\q_n)$ of $U(\q_n)$, can be translated into webs of type Q by drawing the Chevalley generators as certain webs called \emph{ladders}. This is one of the key ideas in \cite{CKM}. The definition of the category $\inftyspiderup$ comes from a translation of $\dot{U}(\q_n)$ (for all $n$) into ladders, which is precisely the reason the web calculus here is the same as the one in \cite{BKu}.

The second main result is Corollary \ref{basis-corollary}, which identifies bases for the spaces $\Hom_{\H}(M^\lambda,M^\mu)$ in terms of webs. The strategy is to first produce bases of the morphism spaces $\Hom_{\inftyspiderup}(\a,\b)$ and then apply $\psi_{\a,\b}$. The basis of $\Hom_{\H}(M^\lambda,M^\mu)$ thus obtained is the type Q analog of that in type A indexed by $(\frakS_\mu,\frakS_\lambda)$-double cosets of $\frakS_r$ (e.g. \cite[Theorem 13.19]{Ja} and \cite[Theorem 4.7]{Ma}). 

We note that Du-Wan \cite[Proposition 5.2]{DW2} have given bases for the spaces $\Hom_{\H_{c,q}(r)}(M_q^\lambda,M_q^\mu)$, where $\H_{c,q}(r)$ is the quantum deformation of $\H_c(r)$ defined by Olshanski \cite{O} and $M_q^\lambda,M_q^\mu$ are its permutation supermodules. Since $\H_{c,1}(r)\simeq\H_c(r)$ and $M^\lambda_1=M^\lambda$, $M^\mu_1=M^\mu$, their basis specializes when $q=1$ to give a basis of $\Hom_{\H_c(r)}(M^\lambda,M^\mu)$. We emphasize therefore that the novelty in Corollary \ref{basis-corollary} is that the basis is given diagrammatically.

\subsection{Organization}

The paper is organized as follows. In Section 2 we provide a brief review of some background information on representations of superalgebras. Section 3 introduces the algebraic (as opposed to diagrammatic) players, namely $\Sdot_c(n,r)$, $\H_c(r)$, and the permutation supermodules $M^\lambda$, including their definition in terms of $\lambda$-supertabloids. In Section 4 we define the web superalgebra $W$ in terms of generators and relations, and derive a number of consequences needed for the main theorems involving special webs called \emph{clasp idempotents} and \emph{Sergeev diagrams}. Finally in Section 5 we state and prove the main theorems.

\subsection{Future work}

In type A, there is a basis of the Iwahori-Hecke algebra, the \emph{Murphy basis}, which can be ``lifted" to obtain a basis of the $q$-Schur algebra. This was first observed by Du-Scott \cite{DS}, and was done for the cyclotomic $q$-Schur algebra by Dipper-James-Mathas \cite{DJM} (see also \cite[\S3]{Ma}). Notably, all of these bases are cellular in the sense of Graham-Lehrer \cite{GL}, and those of the Schur algebras are indexed by pairs of certain tableaux (e.g. semistandard). It would be interesting to investigate whether a $\Z_2$-graded analog of this phenomenon exists for $\H_c(r)$ and $\Sdot_c(n,r)$, and if so, whether webs of type Q prove helpful toward its discovery. 

Finally, one could work out $q$-analogs of the results in this paper for $\H_{c,q}(r)$ and its permutation supermodules. Davidson, Kujawa, and the author are currently investigating $q$-analogs of webs of type Q \cite{BDK}, which would lead to such results in the same way that \cite{BKu} lead to the present paper.

\subsection{Conventions}

All vector spaces are over the field $\C$ of complex numbers.


\subsection{Acknowledgments}


I'm indebted to Jonathan Brundan, Sabin Cautis, Nicholas Davidson, and Jieru Zhu for helpful comments. I'm especially grateful to my PhD advisor, Jonathan Kujawa, for his patient guidance during every stage of this project.


\section{Preliminaries}


In this section, we briefly touch on some topics needed for the main body of the paper.


\subsection{Background on superalgebra}


This subsection includes a review of basic superalgebra. The reader is referred to \cite{CW,Kl} for more information.

Let $\Z_2=\{\bar{0},\bar{1}\}$ be the field with two elements. A \emph{superspace} is a $\Z_2$-graded vector space $V=V_{\bar{0}}\oplus V_{\bar{1}}$. Elements of $V_{\0}$ (resp.\ $V_{\1}$) are said to have parity $\0$ or, equivalently, to be \emph{even} (resp.\ parity $\1$ or \emph{odd}). Given a homogeneous element $v \in V$, we write $|v| \in \Z_{2}$ for the parity of $v$.

The space $\Hom(V,W)$ of linear maps between superspaces $V,W$ is a superspace via the $\Z_2$-grading given by $|f|=i$ if $f(V_j)\subseteq W_{i+j}$ for $j\in\Z_2.$ For this reason, elements of $\Hom(V,W)_{\0}$ (resp. $\Hom(V,W)_{\bar{1}}$) are called \emph{parity-preserving} (resp. \emph{parity-reversing}). The tensor product $V\otimes W$ is also a superspace, via the $\Z_2$-grading given by $\ol{v\otimes w}=\ol{v}+\ol{w}$ for homogeneous $v\in V, w\in W.$

A \emph{superalgebra} is a $\Z_2$-graded associative algebra $A=A_{\0}\oplus A_{\1}$, and a \emph{homomorphism} of superalgebras is a parity-preserving algebra homomorphism. A \emph{supermodule} over a superalgebra $A$ consists of a superspace $V$ and a parity-preserving algebra homomorphism $A\to\End(V)$.


\subsection{Locally unital superalgebras}


Many of the superalgebras $A$ in this paper are \emph{locally unital} in the sense of \cite{BD}. This means $A$ has a system $(1_x)_{x\in X}$ of pairwise orthogonal idempotents such that
\[
A=\bigoplus_{x,y\in X}1_y A1_x.
\]
Further, we say a superalgebra homomorphism $\phi\colon A\to B$ is \emph{locally unital} if $A$ and $B$ are locally unital superalgebras and $\phi$ takes distinguished idempotents to distinguished idempotents. In that case one has linear maps
\[
\phi_{x,y}\colon1_y A1_x\to\phi(1_y)B\phi(1_x).
\]


\subsection{Compositions of $r$}


The permutation supermodules $M^\lambda$ of $\H$ are indexed by \emph{compositions of r}: tuples $\lambda=(\lambda_1,\dots,\lambda_n)$ of nonnegative integers summing to $r$. We denote by $\Lambda(n,r)$ the set of compositions of $r$ with $n$ entries. For $\lambda\in\Lambda(n,r)$, let $l(\lambda)$ be the position of the last nonzero entry of $\lambda$, i.e. the smallest positive integer such that $\lambda_i=0$ if $i> l(\lambda)$. For example, $\lambda=(7,0,1,6,4,0,0)$ is an element of $\Lambda(7,18)$ with $l(\lambda)=5$. Throughout the paper, we denote by $\omega=\omega_r$ the composition $(1,\dots,1)\in\Lambda(r,r).$ 


\section{Sergeev superalgebra and permutation supermodules}


In this section, we recall the Sergeev superalgebra and its permutation supermodules. We also review the Schur superalgebra of type Q, which we will use to define the permutation supermodules and, later, to connect morphisms to webs.


\subsection{Schur superalgebra of type Q}


We now recall the Schur superalgebra of type Q, $\Sdot_c(n,r)$. For $1\leq i\leq n-1$, let $\alpha_i:=(0,\dots,0,1,-1,0,\dots,0)\in\Z^n$ where the $1$ is in the $i^{\th}$ spot. Expressions of the form $\lambda\pm\alpha_i$ for $\lambda\in\Lambda(n,r)$ refer to component-wise addition. We give an idempotented presentation which is slightly different, but immediately deducible, from \cite[Theorem 4.10]{DW1}.

\begin{definition}
The \emph{Schur superalgebra of type Q} is the associative unital superalgebra $\Sdot_c(n,r)$ generated by the even elements $1_\lambda,e_i,f_i$ and odd elements $e_{\bar{i}},f_{\bar{i}},h_{\bar{j}}$ for $\lambda\in\Lambda(n,r)$, $1\leq i\leq n-1$, and $1\leq j\leq n$, subject to the relations
\beq\bg\label{S1}
1_\lambda1_\mu=\delta_{\lambda,\mu}1_\lambda,\quad\textstyle\sum_{\lambda\in\Lambda(n,r)}1_\lambda=1,\quad h_{\bar{i}}1_\lambda=1_\lambda h_{\bar{i}},\\[1ex]
(h_{\bar{i}}h_{\bar{j}}+h_{\bar{j}}h_{\bar{i}})1_\lambda=\delta_{i,j}2\lambda_i1_\lambda,\quad h_{\bar{i}}1_\lambda=0\text{ if }\lambda_i=0,\tag{S1}
\eg\eeq
\beq\bg\label{S2}
e_i1_\lambda=\begin{cases} 1_{\lambda+\alpha_i}e_i & \text{if }\lambda+\alpha_i\in\Lambda(n,r)\\ 0 & \text{otherwise}\end{cases},\quad
f_i1_\lambda=\begin{cases} 1_{\lambda-\alpha_i}f_i & \text{if }\lambda-\alpha_i\in\Lambda(n,r)\\ 0 & \text{otherwise}\end{cases},\\[1ex]
e_{\bar{i}}1_\lambda=\begin{cases} 1_{\lambda+\alpha_i}e_{\bar{i}} & \text{if }\lambda+\alpha_i\in\Lambda(n,r)\\ 0 & \text{otherwise}\end{cases},\quad
f_{\bar{i}}1_\lambda=\begin{cases} 1_{\lambda-\alpha_i}f_{\bar{i}} & \text{if }\lambda-\alpha_i\in\Lambda(n,r)\\ 0 & \text{otherwise}\end{cases},\\[1ex]
1_\lambda e_i=\begin{cases} e_i1_{\lambda-\alpha_i} & \text{if }\lambda-\alpha_i\in\Lambda(n,r)\\ 0 & \text{otherwise}\end{cases},\quad
1_\lambda f_i=\begin{cases} f_i1_{\lambda+\alpha_i} & \text{if }\lambda+\alpha_i\in\Lambda(n,r)\\ 0 & \text{otherwise}\end{cases},\\[1ex]
1_\lambda e_{\bar{i}}=\begin{cases} e_{\bar{i}}1_{\lambda-\alpha_i} & \text{if }\lambda-\alpha_i\in\Lambda(n,r)\\ 0 & \text{otherwise}\end{cases},\quad
1_\lambda f_{\bar{i}}=\begin{cases} f_{\bar{i}}1_{\lambda+\alpha_i} & \text{if }\lambda+\alpha_i\in\Lambda(n,r)\\ 0 & \text{otherwise}\end{cases},\tag{S2}\\[3ex]
\eg\eeq
\beq\bg\label{S3}
(h_{\bar{i}}e_i-e_ih_{\bar{i}})1_\lambda=e_{\bar{i}}1_\lambda,\quad (h_{\bar{i}}e_{i-1}-e_{i-1}h_{\bar{i}})1_\lambda=-e_{\bar{i-1}}1_\lambda,\\[1ex]
(h_{\bar{i}}f_i-f_ih_{\bar{i}})1_\lambda=-f_{\bar{i}}1_\lambda,\quad (h_{\bar{i}}f_{i-1}-f_{i-1}h_{\bar{i}})1_\lambda= f_{\bar{i-1}}1_\lambda,\\[1ex]
(h_{\bar{i}}e_{\bar{i}}+e_{\bar{i}}h_{\bar{i}})1_\lambda=e_i1_\lambda,\quad (h_{\bar{i}}e_{\bar{i-1}}+e_{\bar{i-1}}h_{\bar{i}})1_\lambda= e_{i-1}1_\lambda,\\[1ex]
(h_{\bar{i}}f_{\bar{i}}+f_{\bar{i}}h_{\bar{i}})1_\lambda=f_i1_\lambda,\quad (h_{\bar{i}}f_{\bar{i-1}}+f_{\bar{i-1}}h_{\bar{i}})1_\lambda= f_{i-1}1_\lambda,\\[1ex]
(h_{\bar{i}}e_j-e_jh_{\bar{i}})1_\lambda=(h_{\bar{i}}f_j-f_jh_{\bar{i}})1_\lambda=0\text{ for }i\neq j,j+1\\[1ex]
(h_{\bar{i}}e_{\bar{j}}+e_{\bar{j}}h_{\bar{i}})1_\lambda=(h_{\bar{i}}f_{\bar{j}}+f_{\bar{j}}h_{\bar{i}})1_\lambda=0\text{ for }i\neq j,j+1,\tag{S3}\\[2ex]
\eg\eeq
\beq\bg\label{S4}
(e_if_j-f_je_i)1_\lambda=\delta_{i,j}(\lambda_i-\lambda_{i+1})1_\lambda,\\[1ex]
(e_{\bar{i}}f_{\bar{j}}+f_{\bar{j}}e_{\bar{i}})1_\lambda=\delta_{i,j}(\lambda_i+\lambda_{i+1})1_\lambda,\\[1ex]
(e_if_{\bar{j}}-f_{\bar{j}}e_i)1_\lambda=\delta_{i,j}(h_{\bar{i}}-h_{\bar{i+1}})1_\lambda,\\[1ex]
(e_{\bar{i}}f_j-f_je_{\bar{i}})1_\lambda=\delta_{i,j}(h_{\bar{i}}-h_{\bar{i+1}})1_\lambda,\tag{S4}\\[2ex]
\eg\eeq
\beq\bg\label{S5}
e_{\bar{i}}^2 1_\lambda= f_{\bar{i}}^2 1_\lambda=0,\\[1ex]
(e_ie_{\bar{j}}-e_{\bar{j}}e_i)1_\lambda=(f_if_{\bar{j}}-f_{\bar{j}}f_i)1_\lambda=0\text{ for }i\neq j\pm 1,\\[1ex]
(e_ie_j-e_je_i)1_\lambda=(f_if_j-f_jf_i)1_\lambda=0\text{ for }|i-j|>1,\\[1ex]
(e_{\bar{i}}e_{\bar{j}}+e_{\bar{j}}e_{\bar{i}})1_\lambda=(f_{\bar{i}}f_{\bar{j}}+f_{\bar{j}}f_{\bar{i}})1_\lambda=0\text{ for }|i-j|>1,\\[1ex]
(e_ie_{i+1}-e_{i+1}e_i)1_\lambda=(e_{\bar{i}}e_{\bar{i+1}}+e_{\bar{i+1}}e_{\bar{i}})1_\lambda,\\[1ex]
(e_ie_{\bar{i+1}}-e_{\bar{i+1}}e_i)1_\lambda=(e_{\bar{i}}e_{i+1}-e_{i+1}e_{\bar{i}})1_\lambda,\\[1ex]
(f_{i+1}f_i-f_if_{i+1})1_\lambda=(f_{\bar{i+1}}f_{\bar{i}}+f_{\bar{i}}f_{\bar{i+1}})1_\lambda,\\[1ex]
(f_{\bar{i+1}}f_i-f_if_{\bar{i+1}})1_\lambda=(f_{i+1}f_{\bar{i}}-f_{\bar{i}}f_{i+1})1_\lambda,\tag{S5}\\[2ex]
\eg\eeq
\beq\bg\label{S6}
(e_i^{(2)}e_j-e_ie_je_i+e_je_i^{(2)})1_\lambda=(f_i^{(2)}f_j-f_if_jf_i+f_jf_i^{(2)})1_\lambda=0\text{ for }i=j\pm 1,\\[1ex]
(e_i^{(2)}e_{\bar{j}}-e_ie_{\bar{j}}e_i+e_{\bar{j}}e_i^{(2)})1_\lambda=(f_i^{(2)}f_{\bar{j}}-f_if_{\bar{j}}f_i+f_{\bar{j}}f_i^{(2)})1_\lambda=0\text{ for }i=j\pm 1\tag{S6}\\[1ex]
\eg\eeq
where we denote by $e_i^{(\,j)},f_i^{(\,j)}$ the \emph{divided powers}
\[
e_i^{(\,j)}:=\frac{e_i^{\,j}}{j!}\,,\quad f_i^{(\,j)}:=\frac{f_i^{\,j}}{j!}\,.
\]
\end{definition}



It is clear from the definition that $\Sdot_c(n,r)$ admits the superspace decomposition
\beq\label{S-decomp}
\Sdot_c(n,r)=\bigoplus_{\lambda,\mu\in\Lambda(n,r)}1_\mu\Sdot_c(n,r)1_\lambda
\eeq
and is locally unital with distinguished idempotents $(1_\lambda)_{\lambda\in\Lambda(n,r)}$. We will be especially interested in the case of $n=r$, for which we define the shorthand
\[
\Sdot_c(r):=\Sdot_c(r,r),\quad\Lambda(r):=\Lambda(r,r).
\] 

Let $\Lambda'(n,r)$ be the subset of $\Lambda(n,r)$ consisting of compositions $\lambda$ whose only zeros are trailing, i.e. for $1\leq i\leq n-1$, $\lambda_i=0$ implies $\lambda_{i+1}=\lambda_{i+2}=\cdots=\lambda_n=0$. Let also
\[
\Lambda'(r):=\Lambda'(r,r).
\]
For the sake of streamlining the combinatorics of webs, it will be convenient to truncate $\Sdot_c(r)$ along $\Lambda'(r)$, hence the following definition.

\begin{definition}\label{S-plus}
We define the superalgebra $\Sdot'_c(r)$ to be
\[
\Sdot'_c(r):=\bigoplus_{\lambda,\mu\in\Lambda'(r)}1_\mu\Sdot_c(r)1_\lambda.
\]
It is locally unital with distinguished idempotents $(1_\lambda)_{\lambda\in\Lambda'(r)}.$
\end{definition}

Note that while we have a canonical surjection $\Sdot_c(r)\onto\Sdot'_c(r)$ and injection $\Sdot'_c(r)\into\Sdot_c(r)$, the latter does not preserve the identity.

%
%



\subsection{Sergeev superalgebra}


\begin{definition}\label{sergeev-def}
The \emph{Sergeev superalgebra of rank $r$} is the associative unital superalgebra $\H=\H_c(r)$ generated by the even elements $s_1,\dots,s_{r-1}$ and odd elements $c_1,\dots,c_r$, subject to the relations
\beq\label{hecke-relations}\bg
c_i^2=1,\quad c_ic_j=-c_jc_i\text{ if }i\neq j,\\
s_i^2=1,\quad s_is_{i+1}s_i=s_{i+1}s_is_{i+1},\quad s_is_j=s_js_i\text{ if }i\neq j\pm1,\\
c_is_i=s_ic_{i+1},\quad c_{i+1}s_i=s_ic_i
\eg\eeq
for admissible $i,j.$
\end{definition}

There are two important subsuperalgebras of $\H$, namely the \emph{Clifford superalgebra} $C_r$, generated by $c_1,\dots,c_r$, and the \emph{symmetric group algebra} $\C\frakS_r$, generated by $s_1,\dots,s_{r-1}$. It is well known (e.g. \cite[\S13]{Kl}) that $\H$ admits the basis
\[
\{c_1^{a_1}\cdots c_r^{a_r}\sigma\mid a_1,\dots,a_r\in\{0,1\},\sigma\in\frakS_r\},
\]
which we will refer to as the \emph{standard basis}. 
 We have the \emph{canonical embeddings} $\H_c(k)\into\H_c(r)$ for $k<r$ given by sending the generators of $\H_c(k)$ to those of $\H_c(r)$ with the same names.


\subsection{Permutation supermodules}\label{perm-mods}


We now use supercommuting actions of $\Sdot_c(n,r)$ and $\H$ on the $r$-fold tensor product $(\C^{n|n})^{\otimes r}$ to define the permutation supermodules of $\H$.

Define the sets
\[
I(n|0)=\{1,\dots,n\},\quad I(0|n)=\{\bar{1},\dots,\bar{n}\},\quad I(n|n)=I(n|0)\sqcup I(0|n).
\]
There is an involution on $I(n|n)$ interchanging $I(n|0)$ and $I(0|n)$ given by barring every element, where we declare that bars cancel each other in pairs, e.g. $\bar{\bar{5}}=5$. For $i\in I(n|n)$ we define $\ul{i}\in I(n|0)$ to be the unbarred version of $i$, e.g. $\ul{5}=\ul{\bar{5}}=5$, and $|i|\in\Z_2$ to be $\0$ if and only if $i\in I(n|0)$. 

Let $V=\C^{n|n}$, i.e. the superspace $\C^n\oplus\C^n$. We fix a homogeneous basis $\{v_1,\dots,v_n,v_{\bar{1}},\dots, v_{\bar{n}}\}$ of $V$ where $|v_i|=|i|$. 
 In turn, we have a monomial basis
\[
\{v_\i=v_{i_1}\otimes\cdots\otimes v_{i_r}\mid \i=(i_1,\dots,i_r)\in I(n|n)^r\}
\]
 of $V^{\otimes r}$, where $|v_\i|=|v_{i_1}|+\cdots+|v_{i_r}|$. 

For $\lambda\in\Lambda(n,r)$, we define $V^{\otimes r}_\lambda=(V^{\otimes r})_\lambda$ to be the subsuperspace of $V^{\otimes r}$ spanned by the monomials $v_\i$ with the property that for $1\leq k\leq n$, the number of tensorands $v_{i_j}$ with $\ul{i_j}=k$ is $\lambda_k$. Clearly, $V^{\otimes r}=\bigoplus_{\lambda\in\Lambda(n,r)}V^{\otimes r}_\lambda$.

The action of $\Sdot_c(n,r)$ on $V^{\otimes r}$ is given as follows. For $\lambda\in\Lambda(n,r)$, $1_\lambda$ projects $V^{\otimes r}$ onto $V^{\otimes r}_\lambda$, and for $x\in\{e_i,f_i,e_{\bar{i}},f_{\bar{i}},h_{\bar{j}}\mid 1\leq i\leq n-1,1\leq j\leq n\}$ the action of $x$ on a monomial $v_{i_1}\otimes\cdots\otimes v_{i_r}$ is given by
\bea
& & x.v_{i_1}\otimes v_{i_2}\otimes v_{i_3}\otimes\cdots\otimes v_{i_r}\\[1ex]
&+& (-1)^{|x|\cdot |i_1|}v_{i_1}\otimes x.v_{i_2}\otimes v_{i_3}\otimes\cdots\otimes v_{i_r}\\[1ex]
&+& (-1)^{|x|\cdot(|i_1|+|i_2|)}v_{i_1}\otimes v_{i_2}\otimes x.v_{i_3}\otimes\cdots\otimes v_{i_r}\\
&\vdots& \\
&+& (-1)^{|x|\cdot(|i_1|+\cdots+|i_{r-1}|)}v_{i_1}\otimes v_{i_2}\otimes v_{i_3}\otimes\cdots\otimes x.v_{i_r}
\eea
where
\[
e_i.v_k=\delta_{i,\,\ul{k}-1}v_{k-1},\quad f_i.v_k=\delta_{i,\,\ul{k}}v_{k+1},\quad e_{\bar{i}}.v_k=\delta_{i,\,\ul{k}-1}v_{\ol{k-1}},\quad f_{\bar{i}}.v_k=\delta_{i,\,\ul{k}}v_{\ol{k+1}},\quad h_{\bar{j}}.v_k=\delta_{j,\,\ul{k}}v_{\ol{k}}
\]
for $k\in I(n|n)$.

The action of $\H$ on $V^{\otimes r}$ is given as follows. Let $P$ be the odd automorphism of $V$ given by $P(v_i)=(-1)^{|i|}\sqrt{-1}v_{\bar{i}}$ for $i\in I(n|n)$. Then the action of $\H$ is given by
\[\bg
s_i.(u_1\otimes\cdots\otimes u_r)=(-1)^{|u_i|\cdot|u_{i+1}|}u_1\otimes\cdots\otimes u_{i-1}\otimes u_{i+1}\otimes u_i\otimes u_{i+2}\otimes\cdots\otimes u_r,\\
c_j.(u_1\otimes\cdots\otimes u_r)=(-1)^{|u_1|+\cdots+|u_{j-1}|}u_1\otimes\cdots\otimes u_{j-1}\otimes P(u_j)\otimes u_{j+1}\otimes\cdots\otimes u_r
\eg\]
for $1\leq i\leq r-1$, $1\leq j\leq r$, and $u_1,\dots,u_r\in V.$

It is easily checked that the actions of $\Sdot_c(n,r)$ and $\H$ on $V^{\otimes r}$ supercommute. There are several consequences of this. First, it implies that $V^{\otimes r}_\lambda$ is an $\H$-supermodule for $\lambda\in\Lambda(n,r)$, which we call the \emph{permutation supermodule of shape $\lambda$} and denote
\[
M^\lambda:=V^{\otimes r}_\lambda.
\]
It also means we have a homomorphism
\[
\phi^n\colon\Sdot_c(n,r)\to\End_{\H}(V^{\otimes r}),
\]
which, by combining \cite[Theorem 3]{Se} and \cite[Theorem 4.10]{DW1}, is known to be an isomorphism. Further, using the decomposition $V^{\otimes r}=\bigoplus_{\lambda\in\Lambda(n,r)}M^\lambda$ we have
\[
\End_{\H}(V^{\otimes r})=\bigoplus_{\lambda,\mu\in\Lambda(n,r)}\Hom_{\H}(M^\lambda,M^\mu).
\]
The above constitutes a decomposition of $\End_{\H}(V^{\otimes r})$ as a locally unital superalgebra with distinguished idempotents $(\phi^n(1_\lambda))_{\lambda\in\Lambda(n,r)}$. Summarizing:

\begin{proposition}\label{phi-map}
There exists a locally unital isomorphism
\[
\phi^n\colon\Sdot_c(n,r) \xrightarrow{\sim} \End_{\H}(\textstyle\bigoplus_{\lambda\in\Lambda(n,r)}M^\lambda).
\]
In particular, we have superspace isomorphisms
\[
\phi^n_{\lambda,\mu}\colon 1_\mu\Sdot_c(n,r)1_\lambda\xrightarrow{\sim}\Hom_{\H}(M^\lambda,M^\mu)
\]
for $\lambda,\mu\in\Lambda(n,r)$.
\end{proposition} 


\subsection{Supertabloids}\label{supertabloids-section}


There is also a combinatorial interpretation of $M^\lambda$, which we now construct. Define the ordered alphabet $D=\{1'<1<2'<2<3'<3\cdots\}$. We say the elements $1',2',3',\dots$ are \emph{primed}, and equip $D$ with the $\Z_2$-grading given by $|d|=\ol{1}$ if and only if $d$ is primed. There is a parity-reversing set involution on $D$ given by priming every element, where we declare that $d''=d$ for an unprimed $d\in D$. For $k>0$ let $D_k=\{1',1,2',2,\dots,k',k\}\subset D.$

\begin{definition}\label{supertab-def}
For $\lambda\in\Lambda(n,r)$ and $\mu\in\Lambda(l,r)$, a \emph{$\lambda$-supertabloid of type $\mu$} is an arrangement $T$ of $r$ elements of $D_r$ into $n$ rows, such that
\be
\item for $1\leq i\leq n$, the $i^{\th}$ row of $T$ has $\lambda_i$ entries, and
\item for $1\leq j\leq l$, exactly $\mu_j$ entries of $T$ are elements of $\{j,j'\}$.
\ee
We declare two $\lambda$-supertabloids of type $\mu$ to be equal if they are identical up to permutations of the entries which stabilize the rows. 
\end{definition}

Below are four $(2,1,3)$-supertabloids of type $\omega,$ $(4,2)$, $(1,1,1,2,1)$, and (1,3,2), respectively.
\[
\xy
(0,0)*{
\bt[scale=0.5]
	\draw (0,0) to (2,0);
	\draw (0,-1) to (2,-1);
	\draw (0,-2) to (3,-2);
	\draw (0,-3) to (3,-3);
	\node at (0.5,-0.5) { $3'$};
	\node at (1.5,-0.5) { $6'$};
	\node at (0.5,-1.5) { $2$};
	\node at (0.5,-2.5) { $1$};
	\node at (1.5,-2.5) { $4'$};
	\node at (2.5,-2.5) { $5$};
\et
};
\endxy\quad\quad
\xy
(0,0)*{
\bt[scale=0.5]
	\draw (0,0) to (2,0);
	\draw (0,-1) to (2,-1);
	\draw (0,-2) to (3,-2);
	\draw (0,-3) to (3,-3);
	\node at (0.5,-0.5) { $1'$};
	\node at (1.5,-0.5) { $2'$};
	\node at (0.5,-1.5) { $1$};
	\node at (0.5,-2.5) { $1$};
	\node at (1.5,-2.5) { $1$};
	\node at (2.5,-2.5) { $2'$};
\et
};
\endxy\quad\quad
\xy
(0,0)*{
\bt[scale=0.5]
	\draw (0,0) to (2,0);
	\draw (0,-1) to (2,-1);
	\draw (0,-2) to (3,-2);
	\draw (0,-3) to (3,-3);
	\node at (0.5,-0.5) { $4'$};
	\node at (1.5,-0.5) { $5'$};
	\node at (0.5,-1.5) { $4$};
	\node at (0.5,-2.5) { $1$};
	\node at (1.5,-2.5) { $2'$};
	\node at (2.5,-2.5) { $3$};
\et
};
\endxy\quad\quad
\xy
(0,0)*{
\bt[scale=0.5]
	\draw (0,0) to (2,0);
	\draw (0,-1) to (2,-1);
	\draw (0,-2) to (3,-2);
	\draw (0,-3) to (3,-3);
	\node at (0.5,-0.5) { $2$};
	\node at (1.5,-0.5) { $3'$};
	\node at (0.5,-1.5) { $3$};
	\node at (0.5,-2.5) { $1'$};
	\node at (1.5,-2.5) { $2'$};
	\node at (2.5,-2.5) { $2$};
\et
};
\endxy
\]

Let $\T(\lambda,\mu)$ be the set of all $\lambda$-supertabloids of type $\mu$. We put a $\Z_2$-grading on $\T(\lambda,\mu)$ by defining $|T|=\sum |d|$, summing over all entries $d$ of $T$. As we will be especially interested in supertabloids of type $\omega$, we define the shorthand $\T(\lambda):=\T(\lambda,\omega)$. Further, for $T\in\T(\lambda)$ and $1\leq i\leq r$, we call the unique entry $d$ of $T$ with $d\in\{i,i'\}$ the \emph{$i^{\th}$ entry of $T$} and denote it $T_i$.

There is a parity-preserving bijection between the monomial basis of $M^\lambda$ and $\T(\lambda)$. Indeed, we construct the supertabloid $T\in\T(\lambda)$ associated to $v_\i\in M^\lambda$ by placing a $j$ (resp. a $j'$) in row $\ul{i_j}$ if $i_j\in I(n|0)$ (resp. if $i_j\in I(0|n)$) for $1\leq j\leq r$. For example, the leftmost $(2,1,3)$-supertabloid above corresponds to $v_\i=v_3v_2v_{\ol{1}}v_{\ol{3}}v_3v_{\ol{1}}$ (omitting the $\otimes$ symbols).

We impose an $\H$-supermodule structure on the superspace spanned by $\T(\lambda)$ by transporting the structure from $M^\lambda$ along the above bijection. First we need some notation. For $\lambda\in\Lambda(n,r)$ and $T\in\T(\lambda)$, define $T_{i\leftrightarrow i+1}$ and $T_{j\to j'}$ to be the $\lambda$-supertabloids obtained from $T$ by interchanging $T_i$ and $T_{i+1}$ and by priming $T_j$, respectively, for $1\leq i\leq r-1$ and $1\leq j\leq r.$ Then the following may be verified by direct calculation.

\begin{proposition}
For $\lambda\in\Lambda(n,r)$, $M^\lambda$ is isomorphic as an $\H$-supermodule to the superspace spanned by $\T(\lambda)$ under the action
\[
s_i.T = (-1)^{|T_i|\cdot|T_{i+1}|}T_{i\leftrightarrow i+1},\quad c_j.T = (-1)^{|T_1|+\cdots+|T_{j}|}\sqrt{-1}\,T_{j\to j'}
\]
for $1\leq i\leq r-1$, $1\leq j\leq r$, and $T\in\T(\lambda)$.
\end{proposition}

For example,
\[
s_5.\left(\,
\xy
(0,0)*{
\bt[scale=0.5]
	\draw (0,0) to (2,0);
	\draw (0,-1) to (2,-1);
	\draw (0,-2) to (3,-2);
	\draw (0,-3) to (3,-3);
	\node at (0.5,-0.5) { $3'$};
	\node at (1.5,-0.5) { $6'$};
	\node at (0.5,-1.5) { $2$};
	\node at (0.5,-2.5) { $1$};
	\node at (1.5,-2.5) { $4'$};
	\node at (2.5,-2.5) { $5$};
\et
};
\endxy\,\right) \ = \ 
\xy
(0,0)*{
\bt[scale=0.5]
	\draw (0,0) to (2,0);
	\draw (0,-1) to (2,-1);
	\draw (0,-2) to (3,-2);
	\draw (0,-3) to (3,-3);
	\node at (0.5,-0.5) { $3'$};
	\node at (1.5,-0.5) { $5$};
	\node at (0.5,-1.5) { $2$};
	\node at (0.5,-2.5) { $1$};
	\node at (1.5,-2.5) { $4'$};
	\node at (2.5,-2.5) { $6'$};
\et
};
\endxy \ ,\quad
c_3.\left(\,
\xy
(0,0)*{
\bt[scale=0.5]
	\draw (0,0) to (2,0);
	\draw (0,-1) to (2,-1);
	\draw (0,-2) to (3,-2);
	\draw (0,-3) to (3,-3);
	\node at (0.5,-0.5) { $3'$};
	\node at (1.5,-0.5) { $6'$};
	\node at (0.5,-1.5) { $2$};
	\node at (0.5,-2.5) { $1$};
	\node at (1.5,-2.5) { $4'$};
	\node at (2.5,-2.5) { $5$};
\et
};
\endxy\,\right) \ = \ -\sqrt{-1} \ 
\xy
(0,0)*{
\bt[scale=0.5]
	\draw (0,0) to (2,0);
	\draw (0,-1) to (2,-1);
	\draw (0,-2) to (3,-2);
	\draw (0,-3) to (3,-3);
	\node at (0.5,-0.5) { $3$};
	\node at (1.5,-0.5) { $6'$};
	\node at (0.5,-1.5) { $2$};
	\node at (0.5,-2.5) { $1$};
	\node at (1.5,-2.5) { $4'$};
	\node at (2.5,-2.5) { $5$};
\et
};
\endxy \ .
\]

We conclude this section by highlighting two key examples of permutation supermodules, namely $M^{\omega}$ and $M^{(r)}$. The former is isomorphic to $\H$ itself. Indeed, there is an $\H$-isomorphism $\H\xrightarrow{\sim}M^{\omega}$ given by mapping a standard basis element $x=c_1^{a_1}\cdots c_r^{a_r}\sigma$ to $x.T^0$ where $T^0\in\T(\omega)$ is the supertabloid with entries $1,2,\dots,r$ from top to bottom. Meanwhile $M^{(r)}$, sometimes referred to as the \emph{basic spin supermodule}, is isomorphic to $C_r$ when the latter is thought of as an $\H$-supermodule via
\[
c_i.c_{i_1}c_{i_2}\cdots=c_ic_{i_1}c_{i_2}\cdots,\quad\quad\sigma.c_{i_1}c_{i_2}\cdots=c_{\sigma(i_1)}c_{\sigma(i_2)}\cdots
\]
for $1\leq i,i_1,i_2,\dots\leq r$ and $\sigma\in\frakS_r$. Indeed, an $\H$-isomorphism $C_r\xrightarrow{\sim}M^{(r)}$ is given by sending $y=c_1^{a_1}\cdots c_r^{a_r}$ to $y.T^1$ where $T^1\in\T((r))$ is the supertabloid with entries $1,2,\dots,r$. It is straightforward to see that both maps are isomorphisms of superspaces which preserve the actions of $\H$. 


\section{Web superalgebra}\label{web-sec}


In this section, we define a superalgebra $\W=\W(r)$ whose elements are linear combinations of webs of type Q modulo certain relations. It will be shown later that $\W$ is isomorphic to $\Sdot'_c(r)$ (Theorem \ref{psi-map}). 


\subsection{Definition of $\W$}


Recall from Section 1 that multiplication of webs is given by vertical concatenation, i.e. for webs $w_1,w_2$ we have
\[
w_1w_2 \ := \ 
\xy
(0,0)*{
\begin{tikzpicture}[scale=.35, color=\clr]
	\draw [ thick, ] (0,-5) to (0,-4);
	\draw [ thick, ] (2,-5) to (2,-4);
	\draw [ thick,  directed=1] (0,-2) to (0,-1);
	\draw [ thick,  directed=1] (2,-2) to (2,-1);
	\draw [ thick] (-0.3,-4) rectangle (2.3,-2);
	\node at (1,-3) { $w_1$};
	\draw [ thick, ] (0,-8) to (0,-7);
	\draw [ thick, ] (2,-8) to (2,-7);
	\draw [ thick] (-0.3,-7) rectangle (2.3,-5);
	\node at (1,-1.65) { \,$\cdots$};
	\node at (1,-4.6) { \,$\cdots$};
	\node at (1,-7.6) { \,$\cdots$};
	\node at (1,-6) { $w_2$};
\end{tikzpicture}
};
\endxy \ .
\]
We set $w_1w_2=0$ if the edge labels along the bottom of $w_1$ do not match those along the top of $w_2$, and refer to this statement as the \emph{compatibility condition}. Each web will be homogeneous, and we have the \emph{superinterchange law}
\[
\xy
(0,0)*{
\begin{tikzpicture}[scale=.35, color=\clr]
	\draw [ thick, ] (3.2,-8) to (3.2,-4);
	\draw [ thick, ] (5.2,-8) to (5.2,-4);
	\draw [ thick,  directed=1] (3.2,-2) to (3.2,-1);
	\draw [ thick,  directed=1] (5.2,-2) to (5.2,-1);
	\draw [ thick] (2.9,-4) rectangle (5.5,-2);
	\node at (4.2,-3) { $w_2$};
	\draw [ thick, ] (0,-8) to (0,-7);
	\draw [ thick, ] (2,-8) to (2,-7);
	\draw [ thick] (-0.3,-7) rectangle (2.3,-5);
	\draw [ thick, directed=1] (0,-5) to (0,-1);
	\draw [ thick, directed=1] (2,-5) to (2,-1);
	\node at (4.2,-6) { \ $\cdots$};
	\node at (1,-7.6) { \ $\cdots$};
	\node at (4.2,-1.65) { \ $\cdots$};
	\node at (1,-3) { \ $\cdots$};
	\node at (1,-6) { $w_1$};
\end{tikzpicture}
};
\endxy \ =(-1)^{|w_1|\cdot|w_2|} \ 
\xy
(0,0)*{
\begin{tikzpicture}[scale=.35, color=\clr]
	\draw [ thick, ] (0,-8) to (0,-4);
	\draw [ thick, ] (2,-8) to (2,-4);
	\draw [ thick,  directed=1] (0,-2) to (0,-1);
	\draw [ thick,  directed=1] (2,-2) to (2,-1);
	\draw [ thick] (-0.3,-4) rectangle (2.3,-2);
	\node at (1,-3) { $w_1$};
	\draw [ thick, ] (3.2,-8) to (3.2,-7);
	\draw [ thick, ] (5.2,-8) to (5.2,-7);
	\draw [ thick] (2.9,-7) rectangle (5.5,-5);
	\draw [ thick, directed=1] (3.2,-5) to (3.2,-1);
	\draw [ thick, directed=1] (5.2,-5) to (5.2,-1);
	\node at (1,-6) { \ $\cdots$};
	\node at (4.2,-7.6) { \ $\cdots$};
	\node at (1,-1.65) { \ $\cdots$};
	\node at (4.2,-3) { \ $\cdots$};
	\node at (4.2,-6) { $w_2$};
\end{tikzpicture}
};
\endxy \ .
\]
We will use the words ``edge" and ``strand" synonymously, and will refer to an edge labeled $k$ as a ``$k$-strand".

Let $\A(r)$ be the set of tuples $\a=(\a_1,\dots,\a_m)$ of \emph{positive} integers with $\a_1+\cdots+\a_m=r$, $m\geq 1.$ Then $\W$ is generated by the even elements 
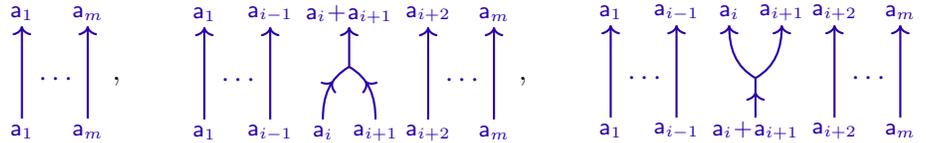
\[
\xy
(0,0)*{
\begin{tikzpicture}[color=\clr, scale=.35]
	\draw [thick, directed=1] (-4,-1) to (-4,2.5);
	\draw [thick, directed=1] (-1.5,-1) to (-1.5,2.5);
	\node at (-1.5,3) {\fs $\a_m$};
	\node at (-1.5,-1.5) {\fs $\a_m$};
	\node at (-4,3) {\fs $\a_1$};
	\node at (-4,-1.5) {\fs $\a_1$};
	\node at (-2.8,0.5) { \ $\cdots$};
\end{tikzpicture}
};
\endxy,\quad\quad
\xy
(0,0)*{
\begin{tikzpicture}[color=\clr, scale=.35]
	\draw [thick, directed=1] (-3,-1) to (-3,2.5);
	\draw [thick, directed=1] (3,-1) to (3,2.5);
	\draw [thick, directed=1] (-5.5,-1) to (-5.5,2.5);
	\draw [thick, directed=1] (5.5,-1) to (5.5,2.5);
	\draw [ thick, directed=1] (0,1) to (0,2.5);
	\draw [ thick, directed=.65] (1,-1) to [out=90,in=330] (0,1);
	\draw [ thick, directed=.65] (-1,-1) to [out=90,in=210] (0,1);
	\node at (0,3) {\fs $\a_i\!+\!\a_{i+1}$};
	\node at (-1,-1.5) {\fs $\a_i$};
	\node at (1,-1.5) {\fs $\a_{i+1}$};
	\node at (-3,-1.5) {\fs $\a_{i-1}$};
	\node at (-3,3) {\fs $\a_{i-1}$};
	\node at (3,-1.5) {\fs $\a_{i+2}$};
	\node at (3,3) {\fs $\a_{i+2}$};
	\node at (-5.5,-1.5) {\fs $\a_1$};
	\node at (-5.5,3) {\fs $\a_1$};
	\node at (5.5,-1.5) {\fs $\a_m$};
	\node at (5.5,3) {\fs $\a_m$};
	\node at (-4.3,0.5) { \ $\cdots$};
	\node at (4.2,0.5) { \ $\cdots$};
\end{tikzpicture}
};
\endxy,\quad\quad
\xy
(0,0)*{
\begin{tikzpicture}[color=\clr, scale=.35]
	\draw [thick, directed=1] (-3,-1) to (-3,2.5);
	\draw [thick, directed=1] (3,-1) to (3,2.5);
	\draw [thick, directed=1] (-5.5,-1) to (-5.5,2.5);
	\draw [thick, directed=1] (5.5,-1) to (5.5,2.5);
	\draw [ thick, directed=0.65] (0,-1) to (0,0.5);
	\draw [ thick, directed=1] (0,0.5) to [out=30,in=270] (1,2.5);
	\draw [ thick, directed=1] (0,0.5) to [out=150,in=270] (-1,2.5); 
	\node at (0,-1.5) {\fs $\a_i\!+\!\a_{i+1}$};
	\node at (-1,3) {\fs $\a_i$};
	\node at (1,3) {\fs $\a_{i+1}$};
	\node at (-3,-1.5) {\fs $\a_{i-1}$};
	\node at (-3,3) {\fs $\a_{i-1}$};
	\node at (3,-1.5) {\fs $\a_{i+2}$};
	\node at (3,3) {\fs $\a_{i+2}$};
	\node at (-5.5,-1.5) {\fs $\a_1$};
	\node at (-5.5,3) {\fs $\a_1$};
	\node at (5.5,-1.5) {\fs $\a_m$};
	\node at (5.5,3) {\fs $\a_m$};
	\node at (-4.3,0.5) { \ $\cdots$};
	\node at (4.2,0.5) { \ $\cdots$};
\end{tikzpicture}
};
\endxy,
\]
called \emph{identities}, \emph{merges}, and \emph{splits}, respectively, and the odd elements
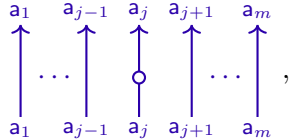
\[
\xy
(0,0)*{
\begin{tikzpicture}[color=\clr, scale=.35]
	\draw [thick, directed=1] (-4,-1) to (-4,2.5);
	\draw [thick, directed=1] (-1.5,-1) to (-1.5,2.5);
	\draw [thick, directed=1] (0.5,-1) to (0.5,2.5);
	\draw [thick, directed=1] (2.5,-1) to (2.5,2.5);
	\draw [thick, directed=1] (5,-1) to (5,2.5);
	\node at (0.5,3) {\fs $\a_j$};
	\node at (0.5,-1.5) {\fs $\a_j$};
	\node at (-1.5,3) {\fs $\a_{j-1}$};
	\node at (-1.5,-1.5) {\fs $\a_{j-1}$};
	\node at (-4,3) {\fs $\a_1$};
	\node at (-4,-1.5) {\fs $\a_1$};
	\node at (2.5,3) {\fs $\a_{j+1}$};
	\node at (2.5,-1.5) {\fs $\a_{j+1}$};
	\node at (5,3) {\fs $\a_m$};
	\node at (5,-1.5) {\fs $\a_m$};
	\node at (-2.8,0.5) { \ $\cdots$};
	\node at (3.7,0.5) { \ $\cdots$};
	\draw (0.5,0.5) \wdot;
\end{tikzpicture}
};
\endxy,
\]
called \emph{dots}, for $\a\in\A(r)$, $m\geq 1$, $1\leq i\leq m-1$, and $1\leq j\leq m$. We will more often use the word ``merge" to refer only to the strands labeled $\a_i$, $\a_{i+1}$, and $\a_i+\a_{i+1}$ in the picture above, and similarly for ``split" and ``dot". Abusing notation, we will refer to the identity corresponding to $\a\in\A(r)$ simply as $\a$. By the compatibility condition, $(\a)_{\a\in\A(r)}$ is a family of pairwise orthogonal idempotents.

\begin{remark}\label{conventions}
As a matter of convenience, some expressions later on may involve webs containing $k$-strands with $k\leq 0$. For this we adopt the following conventions:
\be
\item we erase any $0$-strands which have no dots,
\item we declare that any web containing a dot which lies on a $0$-strand is zero, and
\item we declare that any web containing a $k$-strand with $k<0$ is zero.
\ee
\end{remark}

As dots are the only generators of odd parity, the parity of a web is the number of dots modulo 2. We also have the following relation, implied by the superinterchange law: any pair of disjoint generators, lying at adjacent heights in a single web, may exchange heights with no penalty except when both are dots, in which case we have
\beq\label{superinterchange}
\xy
(0,0)*{
\begin{tikzpicture}[color=\clr, scale=.35]
	\draw [thick, directed=1] (-4,-1) to (-4,2.5);
	\draw [thick, directed=1] (-1.5,-1) to (-1.5,2.5);
	\node at (-1.5,3.1) {\fs $l$};
	\node at (-1.5,-1.6) {\fs $l$};
	\node at (-4,3.1) {\fs $k$};
	\node at (-4,-1.6) {\fs $k$};
	\node at (-2.8,-0.25) { \ $\cdots$};
	\draw (-4,0.5) \wdot;
	\draw (-1.5,1.25) \wdot;
\end{tikzpicture}
};
\endxy=-
\xy
(0,0)*{
\begin{tikzpicture}[color=\clr, scale=.35]
	\draw [thick, directed=1] (-4,-1) to (-4,2.5);
	\draw [thick, directed=1] (-1.5,-1) to (-1.5,2.5);
	\node at (-1.5,3.1) {\fs $l$};
	\node at (-1.5,-1.6) {\fs $l$};
	\node at (-4,3.1) {\fs $k$};
	\node at (-4,-1.6) {\fs $k$};
	\node at (-2.8,-0.25) { \ $\cdots$};
	\draw (-4,1.25) \wdot;
	\draw (-1.5,0.5) \wdot;
\end{tikzpicture}
};
\endxy
\eeq
for $k,l>0$. It will be convenient to draw horizontally adjacent dots as lying at the same height; we resolve the resulting ambiguity by declaring
\[
\xy
(0,0)*{
\begin{tikzpicture}[color=\clr, scale=.35]
	\draw [thick, directed=1] (-4,-1) to (-4,2.5);
	\draw [thick, directed=1] (-1.5,-1) to (-1.5,2.5);
	\node at (-1.5,3.1) {\fs $l$};
	\node at (-1.5,-1.6) {\fs $l$};
	\node at (-4,3.1) {\fs $k$};
	\node at (-4,-1.6) {\fs $k$};
	\draw (-4,0.5) \wdot;
	\draw (-1.5,0.5) \wdot;
\end{tikzpicture}
};
\endxy:=
\xy
(0,0)*{
\begin{tikzpicture}[color=\clr, scale=.35]
	\draw [thick, directed=1] (-4,-1) to (-4,2.5);
	\draw [thick, directed=1] (-1.5,-1) to (-1.5,2.5);
	\node at (-1.5,3.1) {\fs $l$};
	\node at (-1.5,-1.6) {\fs $l$};
	\node at (-4,3.1) {\fs $k$};
	\node at (-4,-1.6) {\fs $k$};
	\draw (-4,1) \wdot;
	\draw (-1.5,0) \wdot;
\end{tikzpicture}
};
\endxy \ .
\]

To state the other relations of $\W$, we introduce some shorthand. We define a \emph{ladder} (with a single rung) to be a web which, ignoring its dots, is a product of a merge on top of a split, such that the left-hand edge of the split connects to right-hand edge of the merge or vise versa:
\[
\xy
(0,0)*{
\bt[color=\clr, scale=.4]
	\draw [ thick, color=\clr, directed=0.3, directed=1] (-2,-1.5) to (-2,2);
	\draw [ thick, color=\clr, directed=1, directed=0.3] (0,-1.5) to (0,2);
	\draw [ thick, color=\clr, directed=0.65] (0,0.25) to (-2,0.25);
	\node at (-2,-2) {\fs $k$};
	\node at (0,-2) {\fs $l$};
	\node at (-2,2.5) {\fs $k\! +\! j$};
	\node at (0,2.5) {\fs $l\! -\! j$};
	\node at (-1,1) {\fs $j$ \ };
\et
};
\endxy
:= 
\xy
(0,0)*{
\bt[color=\clr, scale=.4]
\reflectbox{\draw [ thick, color=\clr,directed=1] (1.5,.75) to (1.5,2);
	\draw [ thick, color=\clr,directed=0.75] (2.5,-1.5) to [out=90,in=320] (1.5,.75);
	\draw [ thick, color=\clr,directed=.65] (0,-1.5) to (0,-.25);
	\draw [ thick, color=\clr] (0,-0.25) to [out=150,in=270] (-1,1.5);
	\draw [ thick, color=\clr,directed=1] (-1,1.5) to (-1,2);
	\draw [ thick, color=\clr,directed=.55] (0,-.25) to (1.5,.75);}
	\node at (-2.5,-2) {\fs $k$};
	\node at (0,-2) {\fs $l$};
	\node at (-1.5,2.5) {\fs $k\! +\! j$};
	\node at (1,2.5) {\fs $l\! -\! j$};
	\node at (-0.25,1) {\fs $j$ \ };
\et
};
\endxy,\quad\quad
\xy
(0,0)*{
\bt[color=\clr, scale=.4]
	\draw [ thick, color=\clr, directed=0.3, directed=1] (-2,-1.5) to (-2,2);
	\draw [ thick, color=\clr, directed=1, directed=0.3] (0,-1.5) to (0,2);
	\draw [ thick, color=\clr, directed=0.65] (-2,0.25) to (0,0.25);
	\node at (-2,-2) {\fs $k$};
	\node at (0,-2) {\fs $l$};
	\node at (-2,2.5) {\fs $k\! -\! j$};
	\node at (0,2.5) {\fs $l\! +\! j$};
	\node at (-1,1) {\fs $j$ \ };
\et
};
\endxy
:= 
\xy
(0,0)*{
\bt[color=\clr, scale=.4]
	\draw [ thick, color=\clr,directed=1] (1.5,.75) to (1.5,2);
	\draw [ thick, color=\clr,directed=0.75] (2.5,-1.5) to [out=90,in=320] (1.5,.75);
	\draw [ thick, color=\clr,directed=.65] (0,-1.5) to (0,-.25);
	\draw [ thick, color=\clr] (0,-0.25) to [out=150,in=270] (-1,1.5);
	\draw [ thick, color=\clr,directed=1] (-1,1.5) to (-1,2);
	\draw [ thick, color=\clr,directed=.55] (0,-.25) to (1.5,.75);
	\reflectbox{\node at (-2.5,-2) {\reflectbox{\fs $l$}};
	\node at (0,-2) {\reflectbox{\fs $k$}};
	\node at (-1.5,2.5) {\reflectbox{\fs $l\! +\! j$}};
	\node at (1,2.5) {\reflectbox{\fs $k\! - \! j$}};
	\node at (-0.75,1) {\reflectbox{\fs $j$ \ }};}
\et
};
\endxy
\]
for $j,k,l>0$. For convenience we did not draw the appropriate identity strands to left and right of either picture. More generally, we define a ladder to be any product of webs of this form, and refer to any horizontal edge of a ladder as a \emph{rung}. 

In addition to to the compatibility condition and superinterchange law, the generators of $\W$ are subject to relations \eqref{associativity}-\eqref{double-rungs-2} below for $h,k,l>0$, along with the relations obtained by reflecting the webs in \eqref{dots-past-merges} about a vertical axis and reversing the orientations of the rungs in \eqref{double-rungs-1} and \eqref{double-rungs-2}. Each is a \emph{local relation}; that is, each should be interpreted as having the appropriate identity strands to its left and right. For example, \eqref{dot-collision} is really the statement
\[
\xy
(0,0)*{
\begin{tikzpicture}[color=\clr, scale=.3] 
	\draw [ thick, directed=1 ] (-4,-2) to (-4,2);
	\draw [ thick, directed=1 ] (-1,-2) to (-1,2);
	\draw [ thick, directed=1 ] (3,-2) to (3,2);
	\draw [ thick, directed=1 ] (6,-2) to (6,2);
	\draw [ color=\clr, thick, directed=1] (1,-2) to (1,2);
	\node at (-4,2.6) {\fs $\a_1$};
	\node at (-1,2.6) {\fs $\a_{i-1}$};
	\node at (1,2.6) {\fs $\a_i$};
	\node at (3,2.6) {\fs $\a_{i+1}$};
	\node at (6,2.6) {\fs $\a_m$};
	\node at (-4,-2.5) {\fs $\a_1$};
	\node at (-1,-2.5) {\fs $\a_{i-1}$};
	\node at (1,-2.5) {\fs $\a_i$};
	\node at (3,-2.5) {\fs $\a_{i+1}$};
	\node at (6,-2.5) {\fs $\a_m$};
	\draw (1,0.5) \wdot; 
	\draw (1,-0.5) \wdot;
	\node at (-2.5,0) {$\cdots$};
	\node at (4.5,0) {$\cdots$};
\end{tikzpicture}
};
\endxy=(\a_i)
\xy
(0,0)*{
\begin{tikzpicture}[color=\clr, scale=.3] 
	\draw [ thick, directed=1 ] (-4,-2) to (-4,2);
	\draw [ thick, directed=1 ] (-1,-2) to (-1,2);
	\draw [ thick, directed=1 ] (3,-2) to (3,2);
	\draw [ thick, directed=1 ] (6,-2) to (6,2);
	\draw [ color=\clr, thick, directed=1] (1,-2) to (1,2);
	\node at (-4,2.6) {\fs $\a_1$};
	\node at (-1,2.6) {\fs $\a_{i-1}$};
	\node at (1,2.6) {\fs $\a_i$};
	\node at (3,2.6) {\fs $\a_{i+1}$};
	\node at (6,2.6) {\fs $\a_m$};
	\node at (-4,-2.5) {\fs $\a_1$};
	\node at (-1,-2.5) {\fs $\a_{i-1}$};
	\node at (1,-2.5) {\fs $\a_i$};
	\node at (3,-2.5) {\fs $\a_{i+1}$};
	\node at (6,-2.5) {\fs $\a_m$};
	\node at (-2.5,0) {$\cdots$};
	\node at (4.5,0) {$\cdots$};
\end{tikzpicture}
};
\endxy
\]
for $\a\in\A(r)$, $m\geq 1$, and $1\leq i\leq m$, and similarly for the others. The parentheses here and in \eqref{dumbbell-relation}, while not strictly necessary, indicate scalars as opposed to edge labels, and as usual the scalar in \eqref{digon-removal} is the binomial $\binom{k+l}{l}=\frac{(k+l)!}{k!\,l!}$:
\beq\label{associativity}
\xy
(0,0)*{
\begin{tikzpicture}[color=\clr, scale=.3]
	\draw [ color=\clr, thick, directed=.55] (0,.75) to [out=90,in=220] (1,2.5);
	\draw [ color=\clr, thick, directed=.55] (1,-1) to [out=90,in=330] (0,.75);
	\draw [ color=\clr, thick, directed=.55] (-1,-1) to [out=90,in=210] (0,.75);
	\draw [ color=\clr, thick, directed=.55] (3,-1) to [out=90,in=330] (1,2.5);
	\draw [ color=\clr, thick, directed=1] (1,2.5) to (1,3.75);
	\node at (-1,-1.65) {\fs $h$};
	\node at (1,-1.65) {\fs $k$};
	\node at (-1.15,1.75) {\fs $h\! +\! k$ \ \ };
	\node at (3,-1.65) {\fs $l$};
	\node at (1,4.4) {\fs $h\! +\! k\! +\! l$};
\end{tikzpicture}
};
\endxy=
\xy,
(0,0)*{\reflectbox{
\begin{tikzpicture}[color=\clr, scale=.3]
	\draw [ color=\clr, thick, directed=.55] (0,.75) to [out=90,in=220] (1,2.5);
	\draw [ color=\clr, thick, directed=.55] (1,-1) to [out=90,in=330] (0,.75);
	\draw [ color=\clr, thick, directed=.55] (-1,-1) to [out=90,in=210] (0,.75);
	\draw [ color=\clr, thick, directed=.55] (3,-1) to [out=90,in=330] (1,2.5);
	\draw [ color=\clr, thick, directed=1] (1,2.5) to (1,3.75);
	\node at (-1,-1.65) {\fs \reflectbox{$l$}};
	\node at (1,-1.65) {\fs \reflectbox{$k$}};
	\node at (-1,1.75) {\fs \reflectbox{ \ $k\! +\! l$}};
	\node at (3,-1.65) {\fs \reflectbox{$h$}};
	\node at (1,4.4) {\fs \reflectbox{$h\! +\! k\! +\! l$}};
\end{tikzpicture}
}};
\endxy\quad,\quad
\xy
(0,0)*{\rotatebox{180}{
\begin{tikzpicture}[color=\clr, scale=.3]
	\draw [ color=\clr, thick, rdirected=.55] (0,.75) to [out=90,in=220] (1,2.5);
	\draw [ color=\clr, thick, rdirected=.1] (1,-1) to [out=90,in=330] (0,.75);
	\draw [ color=\clr, thick, rdirected=.1] (-1,-1) to [out=90,in=210] (0,.75);
	\draw [ color=\clr, thick, rdirected=.05] (3,-1) to [out=90,in=330] (1,2.5);
	\draw [ color=\clr, thick, rdirected=.55] (1,2.5) to (1,3.75);
	\node at (-1,-1.75) {\rotatebox{180}{\fs $l$}};
	\node at (1,-1.75) {\rotatebox{180}{\fs $k$}};
	\node at (-1,1.75) {\rotatebox{180}{\fs \ $k\! +\! l$}};
	\node at (3,-1.75) {\rotatebox{180}{\fs $h$}};
	\node at (1,4.4) {\rotatebox{180}{\fs $h\! +\! k\! +\! l$}};
\end{tikzpicture}
}};
\endxy=
\xy
(0,0)*{\reflectbox{\rotatebox{180}{
\begin{tikzpicture}[color=\clr, scale=.3]
	\draw [ color=\clr, thick, rdirected=.55] (0,.75) to [out=90,in=220] (1,2.5);
	\draw [ color=\clr, thick, rdirected=.1] (1,-1) to [out=90,in=330] (0,.75);
	\draw [ color=\clr, thick, rdirected=.1] (-1,-1) to [out=90,in=210] (0,.75);
	\draw [ color=\clr, thick, rdirected=.05] (3,-1) to [out=90,in=330] (1,2.5);
	\draw [ color=\clr, thick, rdirected=.55] (1,2.5) to (1,3.75);
	\node at (-1,-1.75) {\reflectbox{\rotatebox{180}{\fs $h$}}};
	\node at (1,-1.75) {\reflectbox{\rotatebox{180}{\fs $k$}}};
	\node at (-1,1.75) {\reflectbox{\rotatebox{180}{\fs $h\! +\! k$ \ }}};
	\node at (3,-1.75) {\reflectbox{\rotatebox{180}{\fs $l$}}};
	\node at (1,4.4) {\reflectbox{\rotatebox{180}{\fs $h\! +\! k\! +\! l$}}};
\end{tikzpicture}
}}};
\endxy,
\eeq
\beq\label{digon-removal}
\xy
(0,0)*{
\begin{tikzpicture}[color=\clr, scale=.3]
	\draw [ color=\clr, thick, directed=1] (0,.75) to (0,2);
	\draw [ color=\clr, thick, directed=.55] (0,-2.75) to [out=30,in=330] (0,.75);
	\draw [ color=\clr, thick, directed=.55] (0,-2.75) to [out=150,in=210] (0,.75);
	\draw [ color=\clr, thick, directed=.65] (0,-4) to (0,-2.75);
	\node at (0,-4.5) {\fs $k\! +\! l$};
	\node at (0,2.6) {\fs $k\! +\! l$};
	\node at (-1.5,-1) {\fs $k$ \ \ };
	\node at (1.5,-1) {\fs  \ \ $l$};
\end{tikzpicture}
};
\endxy=\left(\nfrac{k+l}{l}\right)
\xy
(0,0)*{
\begin{tikzpicture}[color=\clr, scale=.3]
	\draw [ color=\clr, thick, directed=1] (0,-4) to (0,2);
	\node at (0,-4.5) {\fs $k\! +\! l$};
	\node at (0,2.6) {\fs $k\! +\! l$};
\end{tikzpicture}
};
\endxy,
\eeq
\beq\label{dot-collision}
\xy
(0,0)*{
\begin{tikzpicture}[color=\clr, scale=.3] 
	\draw [ color=\clr, thick, directed=1] (1,-2) to (1,2);
	\node at (1,2.6) {\fs $k$};
	\node at (1,-2.5) {\fs $k$};
	\draw (1,0.5) \wdot; 
	\draw (1,-0.5) \wdot;
\end{tikzpicture}
};
\endxy=(k)
\xy
(0,0)*{
\begin{tikzpicture}[color=\clr, scale=.3] 
	\draw [ color=\clr, thick, directed=1] (1,-2) to (1,2);
	\node at (1,2.6) {\fs $k$};
	\node at (1,-2.5) {\fs $k$};
\end{tikzpicture}
};
\endxy,
\eeq
\beq\label{dots-past-merges}
\xy
(0,0)*{
\bt[color=\clr, scale=.35]
	\draw [ color=\clr, thick, directed=1] (0, .75) to (0,2.5);
	\draw [ color=\clr, thick, directed=.6] (1,-1) to [out=90,in=330] (0,.75);
	\draw [ color=\clr, thick, directed=.6] (-1,-1) to [out=90,in=210] (0,.75);
	 \draw (0,1.5) \wdot;
	\node at (0, 3) {\fs $1\! +\! k$};
	\node at (-1,-1.5) {\ss $1$};
	\node at (1,-1.5) {\fs $k$};
\et
};
\endxy=
\xy
(0,0)*{
\bt[color=\clr, scale=.35]
	\draw [ color=\clr, thick, directed=1] (0, .75) to (0,2);
	\draw [ color=\clr, thick, directed=.3] (1,-1) to [out=90,in=330] (0,.75);
	\draw [ color=\clr, thick, directed=.3] (-1,-1) to [out=90,in=210] (0,.75);
	 \draw (-0.75,0) \wdot;
	\node at (0, 2.5) {\fs $1\! +\! k$};
	\node at (-1,-1.5) {\ss $1$};
	\node at (1,-1.5) {\fs $k$};
\et
};
\endxy+
\xy
(0,0)*{
\bt[color=\clr, scale=.35]
	\draw [ color=\clr, thick, directed=1] (0, .75) to (0,2);
	\draw [ color=\clr, thick, directed=.3] (1,-1) to [out=90,in=330] (0,.75);
	\draw [ color=\clr, thick, directed=.3] (-1,-1) to [out=90,in=210] (0,.75);
	 \draw (0.75,0) \wdot;
	\node at (0, 2.5) {\fs $1\! +\! k$};
	\node at (-1,-1.5) {\ss $1$};
	\node at (1,-1.5) {\fs $k$};
\et
};
\endxy,\quad\quad
\xy
(0,0)*{
\bt[color=\clr, scale=.35]
	\draw [ color=\clr, thick, directed=.3] (0,-1) to (0,.75);
	\draw [ color=\clr, thick, directed=1] (0,.75) to [out=30,in=270] (1,2.5);
	\draw [ color=\clr, thick, directed=1] (0,.75) to [out=150,in=270] (-1,2.5);
	 \draw (0,0) \wdot;
	\node at (0, -1.5) {\fs $1\! +\! k$};
	\node at (-1,3) {\ss $1$};
	\node at (1,3) {\fs $k$};
\et
};
\endxy=
\xy
(0,0)*{
\bt[color=\clr, scale=.35]
	\draw [ color=\clr, thick, directed=.65] (0,-0.5) to (0,.75);
	\draw [ color=\clr, thick, directed=1] (0,.75) to [out=30,in=270] (1,2.5);
	\draw [ color=\clr, thick, directed=1] (0,.75) to [out=150,in=270] (-1,2.5);
	 \draw (-0.75,1.5) \wdot;
	\node at (0, -1) {\fs $1\! +\! k$};
	\node at (-1,3) {\ss $1$};
	\node at (1,3) {\fs $k$};
\et
};
\endxy+
\xy
(0,0)*{
\bt[color=\clr, scale=.35]
	\draw [ color=\clr, thick, directed=.65] (0,-0.5) to (0,.75);
	\draw [ color=\clr, thick, directed=1] (0,.75) to [out=30,in=270] (1,2.5);
	\draw [ color=\clr, thick, directed=1] (0,.75) to [out=150,in=270] (-1,2.5);
	 \draw (0.75,1.5) \wdot;
	\node at (0, -1) {\fs $1\! +\! k$};
	\node at (-1,3) {\ss $1$};
	\node at (1,3) {\fs $k$};
\et
};
\endxy,
\eeq
\beq\label{dumbbell-relation}
\xy
(0,0)*{
\begin{tikzpicture}[color=\clr, scale=.3]
	\draw [ color=\clr, thick, directed=.55] (0,-1) to (0,.75);
	\draw [ color=\clr, thick, directed=1] (0,.75) to [out=30,in=270] (1,2.5);
	\draw [ color=\clr, thick, directed=1] (0,.75) to [out=150,in=270] (-1,2.5); 
	\draw [ color=\clr, thick, directed=.65] (1,-2.75) to [out=90,in=330] (0,-1);
	\draw [ color=\clr, thick, directed=.65] (-1,-2.75) to [out=90,in=210] (0,-1);
	\node at (-1,3) {\ss $1$};
	\node at (1,3) {\ss $1$};
	\node at (-1,-3.35) {\ss $1$};
	\node at (1,-3.35) {\ss $1$};
	\node at (0.75,-0.25) {\fs $2$};
\end{tikzpicture}
};
\endxy
-
\xy
(0,0)*{
\begin{tikzpicture}[color=\clr, scale=.3]
	\draw [ color=\clr, thick, directed=.55] (0,-1) to (0,.75);
	\draw [ color=\clr, thick, directed=1] (0,.75) to [out=30,in=270] (1,2.5);
	\draw [ color=\clr, thick, directed=1] (0,.75) to [out=150,in=270] (-1,2.5); 
	\draw [ color=\clr, thick, directed=.75] (1,-2.75) to [out=90,in=330] (0,-1);
	\draw [ color=\clr, thick, directed=.75] (-1,-2.75) to [out=90,in=210] (0,-1);
	\node at (-1,3) {\ss $1$};
	\node at (1,3) {\ss $1$};
	\node at (-1,-3.35) {\ss $1$};
	\node at (1,-3.35) {\ss $1$};
	\node at (0.75,-0.25) {\fs $2$};
	\draw (-0.9,-2.15) \wdot; 
	\draw (0.9,-2.15) \wdot; 
	\draw (0.85,1.65) \wdot; 
	\draw (-0.85,1.65) \wdot;
\end{tikzpicture}
};
\endxy=(2)
\xy
(0,0)*{
\begin{tikzpicture}[color=\clr, scale=.3] 
	\draw [ color=\clr, thick, directed=1] (1,-2.75) to (1,2.5);
	\draw [ color=\clr, thick, directed=1] (-1,-2.75) to (-1,2.5);
	\node at (-1,3) {\ss $1$};
	\node at (1,3) {\ss $1$};
	\node at (-1,-3.35) {\ss $1$};
	\node at (1,-3.35) {\ss $1$};
\end{tikzpicture}
};
\endxy \ ,
\eeq
\beq\label{square-switch}
\xy
(0,0)*{
\bt[color=\clr]
	\draw[ color=\clr, thick, directed=.15, directed=1, directed=.55] (0,0) to (0,1.75);
	\node at (0,-0.2) {\fs $k$};
	\node at (0,2) {\fs $k$};
	\draw[ color=\clr, thick, directed=.15, directed=1, directed=.55] (1,0) to (1,1.75);
	\node at (1,-0.2) {\fs $l$};
	\node at (1,2) {\fs $l$};
	\draw[ color=\clr, thick, directed=.55] (0,0.5) to (1,0.5);
	\node at (0.5,0.25) {\ss$1$};
	\node at (-0.5,0.875) {\fs $k\!-\!1$};
	\draw[ color=\clr, thick, directed=.55] (1,1.25) to (0,1.25);
	\node at (0.5,1.5) {\ss$1$};
	\node at (1.5,0.875) {\fs $l\!+\!1$};
\et
};
\endxy-
\xy
(0,0)*{
\bt[color=\clr]
	\draw[ color=\clr, thick, directed=.15, directed=1, directed=.55] (0,0) to (0,1.75);
	\node at (0,-0.2) {\fs $k$};
	\node at (0,2) {\fs $k$};
	\draw[ color=\clr, thick, directed=.15, directed=1, directed=.55] (1,0) to (1,1.75);
	\node at (1,-0.2) {\fs $l$};
	\node at (1,2) {\fs $l$};
	\draw[ color=\clr, thick, directed=.55] (1,0.5) to (0,0.5);
	\node at (0.5,0.25) {\ss$1$};
	\node at (-0.5,0.875) {\fs $k\!+\!1$};
	\draw[ color=\clr, thick, directed=.55] (0,1.25) to (1,1.25);
	\node at (0.5,1.5) {\ss$1$};
	\node at (1.5,0.875) {\fs $l\!-\!1$};
\et
};
\endxy=(k-l)
\xy
(0,0)*{
\begin{tikzpicture}[color=\clr, scale=.3] 
	\draw [ color=\clr, thick, directed=1] (1,-2.75) to (1,2.5);
	\draw [ color=\clr, thick, directed=1] (-1,-2.75) to (-1,2.5);
	\node at (-1,3.1) {\fs $k$};
	\node at (1,3.1) {\fs $l$};
	\node at (-1,-3.35) {\fs $k$};
	\node at (1,-3.35) {\fs $l$};
\end{tikzpicture}
};
\endxy \ ,
\eeq
\begin{eqnarray}\label{square-switch-dots}
\xy
(0,0)*{
\bt[color=\clr]
	\draw[ color=\clr, thick, directed=.15, directed=1, directed=.55] (0,0) to (0,1.75);
	\node at (0,-0.2) {\fs $k$};
	\node at (0,2) {\fs $k$};
	\draw[ color=\clr, thick, directed=.15, directed=1, directed=.55] (1,0) to (1,1.75);
	\node at (1,-0.2) {\fs $l$};
	\node at (1,2) {\fs $l$};
	\draw[ color=\clr, thick, directed=.55] (0,0.5) to (1,0.5);
	\node at (0.5,0.25) {\ss$1$};
	\node at (-0.5,0.875) {\fs $k\!-\!1$};
	\draw[ color=\clr, thick, directed=.55] (1,1.25) to (0,1.25);
	\node at (0.5,1.5) {\ss$1$};
	\node at (1.5,0.875) {\fs $l\!+\!1$};
	 \draw (0.25,1.25) \wdot;
\et
};
\endxy-
\xy
(0,0)*{
\bt[color=\clr]
	\draw[ color=\clr, thick, directed=.15, directed=1, directed=.55] (0,0) to (0,1.75);
	\node at (0,-0.2) {\fs $k$};
	\node at (0,2) {\fs $k$};
	\draw[ color=\clr, thick, directed=.15, directed=1, directed=.55] (1,0) to (1,1.75);
	\node at (1,-0.2) {\fs $l$};
	\node at (1,2) {\fs $l$};
	\draw[ color=\clr, thick, directed=.55] (1,0.5) to (0,0.5);
	\node at (0.5,0.25) {\ss$1$};
	\node at (-0.5,0.875) {\fs $k\!+\!1$};
	\draw[ color=\clr, thick, directed=.55] (0,1.25) to (1,1.25);
	\node at (0.5,1.5) {\ss$1$};
	\node at (1.5,0.875) {\fs $l\!-\!1$};
	 \draw (0.75,0.5) \wdot; 
\et
};
\endxy &=&
\xy
(0,0)*{
\begin{tikzpicture}[color=\clr, scale=.3] 
	\draw [ color=\clr, thick, directed=1] (1,-2.75) to (1,2.5);
	\draw [ color=\clr, thick, directed=1] (-1,-2.75) to (-1,2.5);
	\node at (-1,3.1) {\fs $k$};
	\node at (1,3.1) {\fs $l$};
	\node at (-1,-3.35) {\fs $k$};
	\node at (1,-3.35) {\fs $l$};
	 \draw (-1,-0.125) \wdot;
\end{tikzpicture}
};
\endxy-
\xy
(0,-1)*{
\begin{tikzpicture}[color=\clr, scale=.3] 
	\draw [ color=\clr, thick, directed=1] (1,-2.75) to (1,2.5);
	\draw [ color=\clr, thick, directed=1] (-1,-2.75) to (-1,2.5);
	\node at (-1,3.1) {\fs $k$};
	\node at (1,3.1) {\fs $l$};
	\node at (-1,-3.35) {\fs $k$};
	\node at (1,-3.35) {\fs $l$};
	 \draw (1,-0.125) \wdot;
	\end{tikzpicture}
};
\endxy\\
&=&
\xy
(0,0)*{
\bt[color=\clr]
	\draw[ color=\clr, thick, directed=.15, directed=1, directed=.55] (0,0) to (0,1.75);
	\node at (0,-0.2) {\fs $k$};
	\node at (0,2) {\fs $k$};
	\draw[ color=\clr, thick, directed=.15, directed=1, directed=.55] (1,0) to (1,1.75);
	\node at (1,-0.2) {\fs $l$};
	\node at (1,2) {\fs $l$};
	\draw[ color=\clr, thick, directed=.55] (0,0.5) to (1,0.5);
	\node at (0.5,0.25) {\ss$1$};
	\node at (-0.5,0.875) {\fs $k\!-\!1$};
	\draw[ color=\clr, thick, directed=.55] (1,1.25) to (0,1.25);
	\node at (0.5,1.5) {\ss$1$};
	\node at (1.5,0.875) {\fs $l\!+\!1$};
	 \draw (0.25,0.5) \wdot; 
\et
};
\endxy-
\xy
(0,0)*{
\bt[color=\clr]
	\draw[ color=\clr, thick, directed=.15, directed=1, directed=.55] (0,0) to (0,1.75);
	\node at (0,-0.2) {\fs $k$};
	\node at (0,2) {\fs $k$};
	\draw[ color=\clr, thick, directed=.15, directed=1, directed=.55] (1,0) to (1,1.75);
	\node at (1,-0.2) {\fs $l$};
	\node at (1,2) {\fs $l$};
	\draw[ color=\clr, thick, directed=.55] (1,0.5) to (0,0.5);
	\node at (0.5,0.25) {\ss$1$};
	\node at (-0.5,0.875) {\fs $k\!+\!1$};
	\draw[ color=\clr, thick, directed=.55] (0,1.25) to (1,1.25);
	\node at (0.5,1.5) {\ss$1$};
	\node at (1.5,0.875) {\fs $l\!-\!1$};
	 \draw (0.75,1.25) \wdot; 
\et
};
\endxy \ ,\nonumber
\end{eqnarray}
\beq\label{double-rungs-1}
\xy
(0,0)*{
\bt[color=\clr]
	\draw[ color=\clr, thick, directed=.15, directed=1] (-1,0) to (-1,1.75);
	\node at (-1,-0.2) {\fs $h$};
	\node at (-1,2) {\fs $h\!+\!1$};
	\draw[ color=\clr, thick, directed=.15, directed=1, directed=.55] (0,0) to (0,1.75);
	\node at (0,-0.2) {\fs $k$};
	\node at (0,2) {\fs $k$};
	\draw[ color=\clr, thick, directed=.15, directed=1] (1,0) to (1,1.75);
	\node at (1,-0.2) {\fs $l$};
	\node at (1,2) {\fs $l\!-\!1$};
	\draw[ color=\clr, thick, directed=.55] (1,0.5) to (0,0.5);
	\node at (0.5,0.25) {\ss$1$};
	\draw[ color=\clr, thick, directed=.55] (0,1.25) to (-1,1.25);
	\node at (-0.5,1.5) {\ss$1$};
\et
};
\endxy-
\xy
(0,0)*{
\bt[color=\clr]
	\draw[ color=\clr, thick, directed=.15, directed=1] (-1,0) to (-1,1.75);
	\node at (-1,-0.2) {\fs $h$};
	\node at (-1,2) {\fs $h\!+\!1$};
	\draw[ color=\clr, thick, directed=.15, directed=1, directed=.55] (0,0) to (0,1.75);
	\node at (0,-0.2) {\fs $k$};
	\node at (0,2) {\fs $k$};
	\draw[ color=\clr, thick, directed=.15, directed=1] (1,0) to (1,1.75);
	\node at (1,-0.2) {\fs $l$};
	\node at (1,2) {\fs $l\!-\!1$};
	\draw[ color=\clr, thick, directed=.55] (0,0.5) to (-1,0.5);
	\node at (-0.5,0.25) {\ss$1$};
	\draw[ color=\clr, thick, directed=.55] (1,1.25) to (0,1.25);
	\node at (0.5,1.5) {\ss$1$};
\et
};
\endxy=
\xy
(0,0)*{
\bt[color=\clr]
	\draw[ color=\clr, thick, directed=.15, directed=1] (-1,0) to (-1,1.75);
	\node at (-1,-0.2) {\fs $h$};
	\node at (-1,2) {\fs $h\!+\!1$};
	\draw[ color=\clr, thick, directed=.15, directed=1, directed=.55] (0,0) to (0,1.75);
	\node at (0,-0.2) {\fs $k$};
	\node at (0,2) {\fs $k$};
	\draw[ color=\clr, thick, directed=.15, directed=1] (1,0) to (1,1.75);
	\node at (1,-0.2) {\fs $l$};
	\node at (1,2) {\fs $l\!-\!1$};
	\draw[ color=\clr, thick, directed=.55] (1,0.5) to (0,0.5);
	\node at (0.5,0.25) {\ss$1$};
	\draw[ color=\clr, thick, directed=.55] (0,1.25) to (-1,1.25);
	\node at (-0.5,1.5) {\ss$1$};
	\draw (0.75,0.5) \wdot;
	\draw (-0.75,1.25) \wdot; 
\et
};
\endxy+
\xy
(0,0)*{
\bt[color=\clr]
	\draw[ color=\clr, thick, directed=.15, directed=1] (-1,0) to (-1,1.75);
	\node at (-1,-0.2) {\fs $h$};
	\node at (-1,2) {\fs $h\!+\!1$};
	\draw[ color=\clr, thick, directed=.15, directed=1, directed=.55] (0,0) to (0,1.75);
	\node at (0,-0.2) {\fs $k$};
	\node at (0,2) {\fs $k$};
	\draw[ color=\clr, thick, directed=.15, directed=1] (1,0) to (1,1.75);
	\node at (1,-0.2) {\fs $l$};
	\node at (1,2) {\fs $l\!-\!1$};
	\draw[ color=\clr, thick, directed=.55] (0,0.5) to (-1,0.5);
	\node at (-0.5,0.25) {\ss$1$};
	\draw[ color=\clr, thick, directed=.55] (1,1.25) to (0,1.25);
	\node at (0.5,1.5) {\ss$1$};
	\draw (0.75,1.25) \wdot;
	\draw (-0.75,0.5) \wdot; 
\et
};
\endxy,
\eeq
\beq\label{double-rungs-2}
\xy
(0,0)*{
\bt[color=\clr]
	\draw[ color=\clr, thick, directed=.15, directed=1] (-1,0) to (-1,1.75);
	\node at (-1,-0.2) {\fs $h$};
	\node at (-1,2) {\fs $h\!+\!1$};
	\draw[ color=\clr, thick, directed=.15, directed=1, directed=.55] (0,0) to (0,1.75);
	\node at (0,-0.2) {\fs $k$};
	\node at (0,2) {\fs $k$};
	\draw[ color=\clr, thick, directed=.15, directed=1] (1,0) to (1,1.75);
	\node at (1,-0.2) {\fs $l$};
	\node at (1,2) {\fs $l\!-\!1$};
	\draw[ color=\clr, thick, directed=.55] (1,0.5) to (0,0.5);
	\node at (0.5,0.25) {\ss$1$};
	\draw[ color=\clr, thick, directed=.55] (0,1.25) to (-1,1.25);
	\node at (-0.5,1.5) {\ss$1$};
	 \draw (0.75,0.5) \wdot; 
\et
};
\endxy-
\xy
(0,0)*{
\bt[color=\clr]
	\draw[ color=\clr, thick, directed=.15, directed=1] (-1,0) to (-1,1.75);
	\node at (-1,-0.2) {\fs $h$};
	\node at (-1,2) {\fs $h\!+\!1$};
	\draw[ color=\clr, thick, directed=.15, directed=1, directed=.55] (0,0) to (0,1.75);
	\node at (0,-0.2) {\fs $k$};
	\node at (0,2) {\fs $k$};
	\draw[ color=\clr, thick, directed=.15, directed=1] (1,0) to (1,1.75);
	\node at (1,-0.2) {\fs $l$};
	\node at (1,2) {\fs $l\!-\!1$};
	\draw[ color=\clr, thick, directed=.55] (0,0.5) to (-1,0.5);
	\node at (-0.5,0.25) {\ss$1$};
	\draw[ color=\clr, thick, directed=.55] (1,1.25) to (0,1.25);
	\node at (0.5,1.5) {\ss$1$};
	 \draw (0.75,1.25) \wdot; 
\et
};
\endxy=
\xy
(0,0)*{
\bt[color=\clr]
	\draw[ color=\clr, thick, directed=.15, directed=1] (-1,0) to (-1,1.75);
	\node at (-1,-0.2) {\fs $h$};
	\node at (-1,2) {\fs $h\!+\!1$};
	\draw[ color=\clr, thick, directed=.15, directed=1, directed=.55] (0,0) to (0,1.75);
	\node at (0,-0.2) {\fs $k$};
	\node at (0,2) {\fs $k$};
	\draw[ color=\clr, thick, directed=.15, directed=1] (1,0) to (1,1.75);
	\node at (1,-0.2) {\fs $l$};
	\node at (1,2) {\fs $l\!-\!1$};
	\draw[ color=\clr, thick, directed=.55] (1,0.5) to (0,0.5);
	\node at (0.5,0.25) {\ss$1$};
	\draw[ color=\clr, thick, directed=.55] (0,1.25) to (-1,1.25);
	\node at (-0.5,1.5) {\ss$1$};
	 \draw (-0.75,1.25) \wdot; 
\et
};
\endxy-
\xy
(0,0)*{
\bt[color=\clr]
	\draw[ color=\clr, thick, directed=.15, directed=1] (-1,0) to (-1,1.75);
	\node at (-1,-0.2) {\fs $h$};
	\node at (-1,2) {\fs $h\!+\!1$};
	\draw[ color=\clr, thick, directed=.15, directed=1, directed=.55] (0,0) to (0,1.75);
	\node at (0,-0.2) {\fs $k$};
	\node at (0,2) {\fs $k$};
	\draw[ color=\clr, thick, directed=.15, directed=1] (1,0) to (1,1.75);
	\node at (1,-0.2) {\fs $l$};
	\node at (1,2) {\fs $l\!-\!1$};
	\draw[ color=\clr, thick, directed=.55] (0,0.5) to (-1,0.5);
	\node at (-0.5,0.25) {\ss$1$};
	\draw[ color=\clr, thick, directed=.55] (1,1.25) to (0,1.25);
	\node at (0.5,1.5) {\ss$1$};
	 \draw (-0.75,0.5) \wdot; 
\et
};
\endxy.
\eeq

We refer to \eqref{associativity} as \emph{associativity} and \eqref{digon-removal} as \emph{explosion}. It is clear from the definition that $\W$ admits the superspace decomposition
\[
\W=\bigoplus_{\a,\b\in\A(r)}\b\W\a
\]
and is locally unital with distinguished idempotents $(\a)_{\a\in\A(r)}$.


\subsection{Additional relations}


We now deduce some local relations in $W$ which will be needed later. The first and perhaps most important is
\beq\label{complete-explosion}
\xy
(0,0)*{
\begin{tikzpicture}[scale=.3, color=\clr]
	\draw [ thick, directed=1] (0,-4) to (0,2);
	\node at (0,-4.5) {\scriptsize $k$};
	\node at (0,2.5) {\scriptsize $k$};
\end{tikzpicture}
};
\endxy=\frac{1}{k!}
\xy
(0,0)*{
\begin{tikzpicture}[scale=.3, color=\clr]
	\draw [ thick, directed=1] (0,.75) to (0,2);
	\draw [ thick, directed=.85] (0,-2.75) to [out=30,in=330] (0,.75);
	\draw [ thick, directed=.85] (0,-2.75) to [out=150,in=210] (0,.75);
	\draw [ thick, directed=.65] (0,-4) to (0,-2.75);
	\node at (0,-4.5) {\scriptsize $k$};
	\node at (0,2.5) {\scriptsize $k$};
	\node at (-1.5,-1) {\scriptsize $1$};
	\node at (1.5,-1) {\scriptsize $1$};
	\node at (0,-1) { \,$\cdots$};
\end{tikzpicture}
};
\endxy
\eeq
for $k>0$, where the dots indicate that the $k$-strand has been completely exploded into $k$ separate $1$-strands by repeatedly applying \eqref{digon-removal}. By associativity, there is no ambiguity in the web on the right and we may draw it as it appears.

Let $\a_\omega:=(1,\dots,1)\in\A(r)$.

\begin{lemma}\label{beta-theta-maps}
For $\b,\c\in\A(r)$ we have parity-preserving linear maps
\[
\beta=\beta_\b^\c\colon\c\W\b\to\a_\omega\W\a_\omega,\quad\theta=\theta_\b^\c\colon\a_\omega\W\a_\omega\to\c\W\b
\]
given by
\[
\beta(w)=
\xy
(0,0)*{
\begin{tikzpicture}[scale=1, color=\clr]
	\draw[ thick] (-1.5,1) rectangle (0.5,.5);
	\draw[ thick, directed=0.55] (-1.25,0) to (-1.25,.5);
	\draw[ thick, directed=0.65] (-1.25,1) to (-1.25,1.5);
	\draw[ thick, directed=0.55] (.25,0) to (.25,.5);
	\draw[ thick, directed=0.65] (.25,1) to (.25,1.5);
	\node at (-0.5,1.3) {$\cdots$};
	\node at (-0.5,.2) {$\cdots$};
	\node at (-0.5,0.75) {$w$};
	\node at (-1.55,1.3) {\fs $\c_1$};
	\node at (-1.55,0.2) {\fs $\b_1$};
	\node at (0.65,1.3) {\fs $ \ \c_{m'}$};
	\node at (0.65,0.2) {\fs $\b_m$};
	\draw [ thick, rdirected=.55] (-1.25,0) to [out=210,in=90] (-1.6,-.5);
	\draw [ thick, rdirected=.55] (-1.25,0) to [out=330,in=90] (-0.9,-.5);
	\draw [ thick, rdirected=.55] (0.25,0) to [out=330,in=90] (0.6,-.5);
	\draw [ thick, rdirected=.55] (0.25,0) to [out=210,in=90] (-0.1,-.5);
	\draw [ thick, directed=1] (-1.25,1.5) to [out=30,in=270] (-0.9,2);
	\draw [ thick, directed=1] (-1.25,1.5) to [out=150,in=270] (-1.6,2);
	\draw [ thick, directed=1] (0.25,1.5) to [out=30,in=270] (0.6,2);
	\draw [ thick, directed=1] (0.25,1.5) to [out=150,in=270] (-0.1,2);
	\node at (-1.2,1.8) {\small $\cdots$};
	\node at (-1.2,-.375) {\small $\cdots$};
	\node at (0.3,1.8) {\small $\cdots$};
	\node at (0.3,-.375) {\small $\cdots$};
	\node at (-0.9,2.2) {\ss $1$};
	\node at (-1.6,2.2) {\ss $1$};
	\node at (-0.9,-.7) {\ss $1$};
	\node at (-1.6,-.7) {\ss $1$};
	\node at (-0.1,2.2) {\ss $1$};
	\node at (0.6,2.2) {\ss $1$};
	\node at (-0.1,-.7) {\ss $1$};
	\node at (0.6,-.7) {\ss $1$};
\end{tikzpicture}
};
\endxy
 \ ,\quad\quad
\theta(u)=
\xy
(0,0)*{\rotatebox{180}{
\bt[scale=.35, color=\clr]
	\draw [ thick, rdirected=0.15] (0,1.5) to (0,2.5);
	\draw [ thick, rdirected=0.55] (0,2.5) to [out=30,in=270] (1,4.25);
	\draw [ thick, rdirected=0.55] (0,2.5) to [out=150,in=270] (-1,4.25); 
	\draw [ thick, rdirected=0.15] (5,1.5) to (5,2.5);
	\draw [ thick, rdirected=0.55] (5,2.5) to [out=30,in=270] (6,4.25);
	\draw [ thick, rdirected=0.55] (5,2.5) to [out=150,in=270] (4,4.25); 
	\node at (0,9) {\rotatebox{180}{\scriptsize $\b_m$}};
	\node at (5,9) {\rotatebox{180}{\scriptsize $\b_1$}};
	\node at (0,0.75) {\rotatebox{180}{\scriptsize $\c_{m'}$}};
	\node at (5,0.75) {\rotatebox{180}{\scriptsize $\c_1$}};
	\node at (3.6, 3.75) {\rotatebox{180}{\scriptsize $1$}};
	\node at (6.4, 3.75) {\rotatebox{180}{\scriptsize $1$}};
	\node at (1.4, 3.75) {\rotatebox{180}{\scriptsize $1$}};
	\node at (-1.4, 3.75) {\rotatebox{180}{\scriptsize $1$}};
	\node at (3.6, 6.25) {\rotatebox{180}{\scriptsize $1$}};
	\node at (6.4, 6.25) {\rotatebox{180}{\scriptsize $1$}};
	\node at (1.4, 6.25) {\rotatebox{180}{\scriptsize $1$}};
	\node at (-1.4,6.25) {\rotatebox{180}{\scriptsize $1$}};
	\draw [ thick, rdirected=0.65] (0, 7.5) to (0,8.5);
	\draw [ thick, rdirected=0.55] (1,5.75) to [out=90,in=330] (0,7.5);
	\draw [ thick, rdirected=0.55] (-1,5.75) to [out=90,in=210] (0,7.5);
	\draw [ thick, rdirected=0.65] (5,7.5) to (5,8.5);
	\draw [ thick, rdirected=0.55] (6,5.75) to [out=90,in=330] (5,7.5);
	\draw [ thick, rdirected=0.55] (4,5.75) to [out=90,in=210] (5,7.5);
	\draw [ thick ] (-1.5,5.75) rectangle (6.5,4.25);
	\node at (0.1, 6.25) { $\cdots$};
	\node at (5.1, 6.25) { $\cdots$};
	\node at (0.1, 3.75) { $\cdots$};
	\node at (5.1, 3.75) { $\cdots$};
	\node at (2.6, 7) { $\cdots$};
	\node at (2.6, 2.75) { $\cdots$};
	\node at (2.6, 5) {\rotatebox{180}{ $u$}};
\et
}};
\endxy
\]
for $w\in\c\W\b$ and $u\in\a_\omega\W\a_\omega$. Moreover, $\beta$ is injective and $\theta$ is surjective.
\end{lemma}

\begin{proof}
By \eqref{complete-explosion}, $\theta\circ\beta$ is the identity on $\c\W\b$ up to a nonzero scalar, so $\beta$ is injective and $\theta$ is surjective. 
\end{proof}

\begin{lemma}\label{additional-relations}
The following local relations hold in $\W$ for $k\geq1$:
\beq\label{2-dots-zero}
\xy
(0,0)*{
\begin{tikzpicture}[color=\clr, scale=.3]
	\draw [ color=\clr, thick, directed=1] (0,.75) to (0,2);
	\draw [ color=\clr, thick, directed=.85] (0,-2.75) to [out=30,in=330] (0,.75);
	\draw [ color=\clr, thick, directed=.85] (0,-2.75) to [out=150,in=210] (0,.75);
	\draw [ color=\clr, thick, directed=.65] (0,-4) to (0,-2.75);
	\node at (0,-4.5) {\scriptsize $2$};
	\node at (0,2.5) {\scriptsize $2$};
	\node at (-1.5,-1) {\scriptsize $1$};
	\node at (1.5,-1) {\scriptsize $1$};
	\draw (-0.85,-1) \wdot;
	\draw (0.85,-1) \wdot;
\end{tikzpicture}
};
\endxy=0,
\eeq
\beq\label{dot-on-k-strand}
\xy
(0,0)*{
\begin{tikzpicture}[color=\clr, scale=.3]
	\draw [ color=\clr, thick, directed=1] (0,.75) to (0,2);
	\draw [ color=\clr, thick, directed=.85] (0,-2.75) to [out=30,in=330] (0,.75);
	\draw [ color=\clr, thick, directed=.85] (0,-2.75) to [out=150,in=210] (0,.75);
	\draw [ color=\clr, thick, directed=.65] (0,-4) to (0,-2.75);
	\node at (0,-4.5) {\scriptsize $k$};
	\node at (0,2.5) {\scriptsize $k$};
	\node at (-1.5,-1) {\scriptsize $1$};
	\node at (2,-1) {\scriptsize $k\!-\!1$};
	\draw (-0.875,-1) \wdot;
\end{tikzpicture}
};
\endxy=
\xy
(0,0)*{
\begin{tikzpicture}[color=\clr, scale=.3]
	\draw [ color=\clr, thick, directed=1] (0,-4.) to (0,2);
	\node at (0,-4.5) {\scriptsize $k$};
	\node at (0,2.5) {\scriptsize $k$};
	\draw (0,-1) \wdot;
\end{tikzpicture}
};
\endxy=
\xy
(0,0)*{
\begin{tikzpicture}[color=\clr, scale=.3]
	\draw [ color=\clr, thick, directed=1] (0,.75) to (0,2);
	\draw [ color=\clr, thick, directed=.85] (0,-2.75) to [out=30,in=330] (0,.75);
	\draw [ color=\clr, thick, directed=.85] (0,-2.75) to [out=150,in=210] (0,.75);
	\draw [ color=\clr, thick, directed=.65] (0,-4) to (0,-2.75);
	\node at (0,-4.5) {\scriptsize $k$};
	\node at (0,2.5) {\scriptsize $k$};
	\node at (-2,-1) {\scriptsize $k\!-\!1$};
	\node at (1.5,-1) {\scriptsize $1$};
	\draw (0.875,-1) \wdot;
\end{tikzpicture}
};
\endxy.
\eeq
\end{lemma}

\begin{proof}
The proof of \cref{2-dots-zero} is by direct calculation (cf. \cite[Lemma 2.9]{TVW}):
\bea
\xy
(0,1)*{
\begin{tikzpicture}[color=\clr, scale=.3]
	\draw [ color=\clr, thick, directed=1] (0,.75) to (0,7.5);
	\draw [ color=\clr, thick, directed=.85] (0,-2.75) to [out=30,in=330] (0,.75);
	\draw [ color=\clr, thick, directed=.85] (0,-2.75) to [out=150,in=210] (0,.75);
	\draw [ color=\clr, thick, directed=.65] (0,-4) to (0,-2.75);
	\node at (0,-4.5) {\scriptsize $2$};
	\node at (0,8) {\scriptsize $2$};
	\node at (-1.5,-1) {\scriptsize $1$};
	\node at (1.5,-1) {\scriptsize $1$};
	\draw (-0.85,-1) \wdot;
	\draw (0.85,-1) \wdot;
\end{tikzpicture}
};
\endxy
 & \stackrel{\eqref{digon-removal}}{=} & 
\frac{1}{2}
\xy
(0,0)*{
\begin{tikzpicture}[color=\clr, scale=.3]
	\draw [ color=\clr, thick, directed=.65] (0,.75) to (0,2.5);
	\draw [ color=\clr, thick, directed=.85] (0,-2.75) to [out=30,in=330] (0,.75);
	\draw [ color=\clr, thick, directed=.85] (0,-2.75) to [out=150,in=210] (0,.75);
	\draw [ color=\clr, thick, directed=.65] (0,-4) to (0,-2.75);
	\draw [ color=\clr, thick, directed=1] (0,6) to (0,7.5);
	\draw [ color=\clr, thick, directed=.55] (0,2.5) to [out=30,in=330] (0,6);
	\draw [ color=\clr, thick, directed=.55] (0,2.5) to [out=150,in=210] (0,6);
	\node at (0,-4.5) {\scriptsize $2$};
	\node at (0,8) {\scriptsize $2$};
	\node at (-1.5,-1) {\scriptsize $1$};
	\node at (1.5,-1) {\scriptsize $1$};
	\node at (-1.75,4.25) {\scriptsize $1$};
	\node at (1.75,4.25) {\scriptsize $1$};
	\node at (0.75,1.75) {\scriptsize $2$};
	\draw (-0.85,-1) \wdot;
	\draw (0.85,-1) \wdot;
\end{tikzpicture}
};
\endxy
 \ \stackrel{\eqref{dumbbell-relation}}{=} \ 
\frac{1}{2}
\xy
(0,0)*{
\begin{tikzpicture}[color=\clr, scale=.3]
	\draw [ color=\clr, thick, directed=.65] (0,.75) to (0,2.5);
	\draw [ color=\clr, thick, directed=.85] (0,-2.75) to [out=30,in=330] (0,.75);
	\draw [ color=\clr, thick, directed=.85] (0,-2.75) to [out=150,in=210] (0,.75);
	\draw [ color=\clr, thick, directed=.65] (0,-4) to (0,-2.75);
	\draw [ color=\clr, thick, directed=1] (0,6) to (0,7.5);
	\draw [ color=\clr, thick, directed=.85] (0,2.5) to [out=30,in=330] (0,6);
	\draw [ color=\clr, thick, directed=.85] (0,2.5) to [out=150,in=210] (0,6);
	\node at (0,-4.5) {\scriptsize $2$};
	\node at (0,8) {\scriptsize $2$};
	\node at (-1.5,-1) {\scriptsize $1$};
	\node at (1.5,-1) {\scriptsize $1$};
	\node at (-1.5,4.25) {\scriptsize $1$};
	\node at (1.5,4.25) {\scriptsize $1$};
	\node at (0.75,1.75) {\scriptsize $2$};
	\draw (-0.85,-1) \wdot;
	\draw (0.85,-1) \wdot;
	\draw (-0.7,-2) \wdot;
	\draw (0.7,-2) \wdot;
	\draw (-0.85,4.25) \wdot;
	\draw (0.85,4.25) \wdot;
\end{tikzpicture}
};
\endxy
+
\xy
(0,0)*{
\begin{tikzpicture}[color=\clr, scale=.3]
	\draw [ color=\clr, thick, directed=.55] (.9,-1) to (.9,4.25);
	\draw [ color=\clr, thick, directed=.55] (-.9,-1) to (-.9,4.25);
	\draw [ thick] (0,-2.75) to [out=30,in=270] (.9,-1);
	\draw [ thick] (0,-2.75) to [out=150,in=270] (-.9,-1);
	\draw [ color=\clr, thick, directed=.65] (0,-4) to (0,-2.75);
	\draw [ color=\clr, thick, directed=1] (0,6) to (0,7.5);
	\draw [ thick] (.9,4.25) to [out=90,in=330] (0,6);
	\draw [ thick] (-.9,4.25) to [out=90,in=210] (0,6);
	\node at (0,-4.5) {\scriptsize $2$};
	\node at (0,8) {\scriptsize $2$};
	\node at (-1.75,1.6) {\scriptsize $1$};
	\node at (1.75,1.6) {\scriptsize $1$};
	\draw (-0.85,-1) \wdot;
	\draw (0.85,-1) \wdot;
\end{tikzpicture}
};
\endxy
\stackrel{\eqref{superinterchange}}{=}
-\frac{1}{2}
\xy
(0,0)*{
\begin{tikzpicture}[color=\clr, scale=.3]
	\draw [ color=\clr, thick, directed=.65] (0,.75) to (0,2.5);
	\draw [ color=\clr, thick, directed=.85] (0,-2.75) to [out=30,in=330] (0,.75);
	\draw [ color=\clr, thick, directed=.85] (0,-2.75) to [out=150,in=210] (0,.75);
	\draw [ color=\clr, thick, directed=.65] (0,-4) to (0,-2.75);
	\draw [ color=\clr, thick, directed=1] (0,6) to (0,7.5);
	\draw [ color=\clr, thick, directed=.85] (0,2.5) to [out=30,in=330] (0,6);
	\draw [ color=\clr, thick, directed=.85] (0,2.5) to [out=150,in=210] (0,6);
	\node at (0,-4.5) {\scriptsize $2$};
	\node at (0,8) {\scriptsize $2$};
	\node at (-1.5,-1) {\scriptsize $1$};
	\node at (1.5,-1) {\scriptsize $1$};
	\node at (-1.5,4.25) {\scriptsize $1$};
	\node at (1.5,4.25) {\scriptsize $1$};
	\node at (0.75,1.75) {\scriptsize $2$};
	\draw (-0.85,-0.5) \wdot;
	\draw (0.85,-1.25) \wdot;
	\draw (-0.85,-1.25) \wdot;
	\draw (0.7,-2) \wdot;
	\draw (-0.85,4.25) \wdot;
	\draw (0.85,4.25) \wdot;
\end{tikzpicture}
};
\endxy
+
\xy
(0,0)*{
\begin{tikzpicture}[color=\clr, scale=.3]
	\draw [ color=\clr, thick, directed=1] (0,.75) to (0,2);
	\draw [ color=\clr, thick, directed=.85] (0,-2.75) to [out=30,in=330] (0,.75);
	\draw [ color=\clr, thick, directed=.85] (0,-2.75) to [out=150,in=210] (0,.75);
	\draw [ color=\clr, thick, directed=.65] (0,-4) to (0,-2.75);
	\node at (0,-4.5) {\scriptsize $2$};
	\node at (0,2.5) {\scriptsize $2$};
	\node at (-1.5,-1) {\scriptsize $1$};
	\node at (1.5,-1) {\scriptsize $1$};
	\draw (-0.85,-1) \wdot;
	\draw (0.85,-1) \wdot;
\end{tikzpicture}
};
\endxy\\
 & \stackrel{\eqref{dot-collision}}{=} &
-\frac{1}{2}
\xy
(0,0)*{
\begin{tikzpicture}[color=\clr, scale=.3]
	\draw [ color=\clr, thick, directed=.65] (0,.75) to (0,2.5);
	\draw [ color=\clr, thick, directed=.55] (0,-2.75) to [out=30,in=330] (0,.75);
	\draw [ color=\clr, thick, directed=.55] (0,-2.75) to [out=150,in=210] (0,.75);
	\draw [ color=\clr, thick, directed=.65] (0,-4) to (0,-2.75);
	\draw [ color=\clr, thick, directed=1] (0,6) to (0,7.5);
	\draw [ color=\clr, thick, directed=.85] (0,2.5) to [out=30,in=330] (0,6);
	\draw [ color=\clr, thick, directed=.85] (0,2.5) to [out=150,in=210] (0,6);
	\node at (0,-4.5) {\scriptsize $2$};
	\node at (0,8) {\scriptsize $2$};
	\node at (-1.75,-1) {\scriptsize $1$};
	\node at (1.75,-1) {\scriptsize $1$};
	\node at (-1.5,4.25) {\scriptsize $1$};
	\node at (1.5,4.25) {\scriptsize $1$};
	\node at (0.75,1.75) {\scriptsize $2$};
	\draw (-0.85,4.25) \wdot;
	\draw (0.85,4.25) \wdot;
\end{tikzpicture}
};
\endxy
+
\xy
(0,0)*{
\begin{tikzpicture}[color=\clr, scale=.3]
	\draw [ color=\clr, thick, directed=1] (0,.75) to (0,2);
	\draw [ color=\clr, thick, directed=.85] (0,-2.75) to [out=30,in=330] (0,.75);
	\draw [ color=\clr, thick, directed=.85] (0,-2.75) to [out=150,in=210] (0,.75);
	\draw [ color=\clr, thick, directed=.65] (0,-4) to (0,-2.75);
	\node at (0,-4.5) {\scriptsize $2$};
	\node at (0,2.5) {\scriptsize $2$};
	\node at (-1.5,-1) {\scriptsize $1$};
	\node at (1.5,-1) {\scriptsize $1$};
	\draw (-0.85,-1) \wdot;
	\draw (0.85,-1) \wdot;
\end{tikzpicture}
};
\endxy
\stackrel{\eqref{digon-removal}}{=}
-
\xy
(0,0)*{
\begin{tikzpicture}[color=\clr, scale=.3]
	\draw [ color=\clr, thick, directed=1] (0,.75) to (0,2);
	\draw [ color=\clr, thick, directed=.85] (0,-2.75) to [out=30,in=330] (0,.75);
	\draw [ color=\clr, thick, directed=.85] (0,-2.75) to [out=150,in=210] (0,.75);
	\draw [ color=\clr, thick, directed=.65] (0,-4) to (0,-2.75);
	\node at (0,-4.5) {\scriptsize $2$};
	\node at (0,2.5) {\scriptsize $2$};
	\node at (-1.5,-1) {\scriptsize $1$};
	\node at (1.5,-1) {\scriptsize $1$};
	\draw (-0.85,-1) \wdot;
	\draw (0.85,-1) \wdot;
\end{tikzpicture}
};
\endxy+
\xy
(0,0)*{
\begin{tikzpicture}[color=\clr, scale=.3]
	\draw [ color=\clr, thick, directed=1] (0,.75) to (0,2);
	\draw [ color=\clr, thick, directed=.85] (0,-2.75) to [out=30,in=330] (0,.75);
	\draw [ color=\clr, thick, directed=.85] (0,-2.75) to [out=150,in=210] (0,.75);
	\draw [ color=\clr, thick, directed=.65] (0,-4) to (0,-2.75);
	\node at (0,-4.5) {\scriptsize $2$};
	\node at (0,2.5) {\scriptsize $2$};
	\node at (-1.5,-1) {\scriptsize $1$};
	\node at (1.5,-1) {\scriptsize $1$};
	\draw (-0.85,-1) \wdot;
	\draw (0.85,-1) \wdot;
\end{tikzpicture}
};
\endxy \ = \ 0.
\eea
For \cref{dot-on-k-strand}, we first prove the case of $k=2$. We start by computing that

\beq\label{dot-on-2-strand}
\xy
(0,0)*{
\begin{tikzpicture}[color=\clr, scale=.3]
	\draw [ color=\clr, thick, directed=1] (0,-4) to (0,2);
	\node at (0,-4.5) {\scriptsize $2$};
	\node at (0,2.5) {\scriptsize $2$};
	\draw (0,1) \wdot;
\end{tikzpicture}
};
\endxy
 \ \stackrel{\eqref{digon-removal}}{=} \ 
\frac{1}{2}
\xy
(0,0)*{
\begin{tikzpicture}[color=\clr, scale=.3]
	\draw [ color=\clr, thick, directed=1] (0,.75) to (0,2.5);
	\draw [ color=\clr, thick, directed=.55] (0,-2.75) to [out=30,in=330] (0,.75);
	\draw [ color=\clr, thick, directed=.55] (0,-2.75) to [out=150,in=210] (0,.75);
	\draw [ color=\clr, thick, directed=.65] (0,-4) to (0,-2.75);
	\node at (0,-4.5) {\scriptsize $2$};
	\node at (0,3) {\scriptsize $2$};
	\node at (-1.75,-1) {\scriptsize $1$};
	\node at (1.75,-1) {\scriptsize $1$};
	\draw (0,1.5) \wdot;
\end{tikzpicture}
};
\endxy
 \ \stackrel{\eqref{dots-past-merges}}{=} \ 
\frac{1}{2}\left(
\xy
(0,0)*{
\begin{tikzpicture}[color=\clr, scale=.3]
	\draw [ color=\clr, thick, directed=1] (0,.75) to (0,2);
	\draw [ color=\clr, thick, directed=.85] (0,-2.75) to [out=30,in=330] (0,.75);
	\draw [ color=\clr, thick, directed=.85] (0,-2.75) to [out=150,in=210] (0,.75);
	\draw [ color=\clr, thick, directed=.65] (0,-4) to (0,-2.75);
	\node at (0,-4.5) {\scriptsize $2$};
	\node at (0,2.5) {\scriptsize $2$};
	\node at (-1.5,-1) {\scriptsize $1$};
	\node at (1.5,-1) {\scriptsize $1$};
	\draw (-0.875,-1) \wdot;
\end{tikzpicture}
};
\endxy+
\xy
(0,0)*{
\begin{tikzpicture}[color=\clr, scale=.3]
	\draw [ color=\clr, thick, directed=1] (0,.75) to (0,2);
	\draw [ color=\clr, thick, directed=.85] (0,-2.75) to [out=30,in=330] (0,.75);
	\draw [ color=\clr, thick, directed=.85] (0,-2.75) to [out=150,in=210] (0,.75);
	\draw [ color=\clr, thick, directed=.65] (0,-4) to (0,-2.75);
	\node at (0,-4.5) {\scriptsize $2$};
	\node at (0,2.5) {\scriptsize $2$};
	\node at (-1.5,-1) {\scriptsize $1$};
	\node at (1.5,-1) {\scriptsize $1$};
	\draw (0.875,-1) \wdot;
\end{tikzpicture}
};
\endxy\right).
\eeq
Next, we compose \cref{dots-past-merges} on bottom with 
$\xy
(0,0)*{
\begin{tikzpicture}[color=\clr, scale=.1] 
	\draw [ thick, directed=1] (2,-2) to (2,2.5);
	\draw [ thick, directed=1] (-1,-2) to (-1,2.5);
	\draw (-1,0) \wdot;
	\draw (2,0) \wdot;
\end{tikzpicture}
};
\endxy$
 followed by a split to get

\[
\xy
(0,0)*{
\begin{tikzpicture}[color=\clr, scale=.3]
	\draw [ color=\clr, thick, directed=1] (0,.75) to (0,2.5);
	\draw [ color=\clr, thick, directed=.35] (0,-2.75) to [out=30,in=330] (0,.75);
	\draw [ color=\clr, thick, directed=.35] (0,-2.75) to [out=150,in=210] (0,.75);
	\draw [ color=\clr, thick, directed=.65] (0,-4) to (0,-2.75);
	\node at (0,-4.5) {\scriptsize $2$};
	\node at (0,3) {\scriptsize $2$};
	\node at (-1.5,-1) {\scriptsize $1$};
	\node at (1.5,-1) {\scriptsize $1$};
	\draw (-0.85,-1) \wdot;
	\draw (0.85,-1) \wdot;
	\draw (0,1.5) \wdot;
\end{tikzpicture}
};
\endxy \ = \ 
\xy
(0,0)*{
\begin{tikzpicture}[color=\clr, scale=.3]
	\draw [ color=\clr, thick, directed=1] (0,.75) to (0,2);
	\draw [ color=\clr, thick, directed=.35] (0,-2.75) to [out=30,in=330] (0,.75);
	\draw [ color=\clr, thick, directed=.35] (0,-2.75) to [out=150,in=210] (0,.75);
	\draw [ color=\clr, thick, directed=.65] (0,-4) to (0,-2.75);
	\node at (0,-4.5) {\scriptsize $2$};
	\node at (0,2.5) {\scriptsize $2$};
	\node at (-1.5,-1) {\scriptsize $1$};
	\node at (1.5,-1) {\scriptsize $1$};
	\draw (-0.85,-1) \wdot;
	\draw (0.85,-1) \wdot;
	\draw (-0.65,0) \wdot;
\end{tikzpicture}
};
\endxy+
\xy
(0,0)*{
\begin{tikzpicture}[color=\clr, scale=.3]
	\draw [ color=\clr, thick, directed=1] (0,.75) to (0,2);
	\draw [ color=\clr, thick, directed=.35] (0,-2.75) to [out=30,in=330] (0,.75);
	\draw [ color=\clr, thick, directed=.35] (0,-2.75) to [out=150,in=210] (0,.75);
	\draw [ color=\clr, thick, directed=.65] (0,-4) to (0,-2.75);
	\node at (0,-4.5) {\scriptsize $2$};
	\node at (0,2.5) {\scriptsize $2$};
	\node at (-1.5,-1) {\scriptsize $1$};
	\node at (1.5,-1) {\scriptsize $1$};
	\draw (-0.85,-1) \wdot;
	\draw (0.85,-1) \wdot;
	\draw (0.65,0) \wdot;
\end{tikzpicture}
};
\endxy \ .
\]
Using \cref{2-dots-zero} on the left, and superinterchange and \cref{dot-collision} on the right, this becomes
\[
0 \ =
\xy
(0,0)*{
\begin{tikzpicture}[color=\clr, scale=.3]
	\draw [ color=\clr, thick, directed=1] (0,.75) to (0,2);
	\draw [ color=\clr, thick, directed=.35] (0,-2.75) to [out=30,in=330] (0,.75);
	\draw [ color=\clr, thick, directed=.35] (0,-2.75) to [out=150,in=210] (0,.75);
	\draw [ color=\clr, thick, directed=.65] (0,-4) to (0,-2.75);
	\node at (0,-4.5) {\scriptsize $2$};
	\node at (0,2.5) {\scriptsize $2$};
	\node at (-1.5,-1) {\scriptsize $1$};
	\node at (1.5,-1) {\scriptsize $1$};
	\draw (0.875,-1) \wdot;
\end{tikzpicture}
};
\endxy-
\xy
(0,0)*{
\begin{tikzpicture}[color=\clr, scale=.3]
	\draw [ color=\clr, thick, directed=1] (0,.75) to (0,2);
	\draw [ color=\clr, thick, directed=.35] (0,-2.75) to [out=30,in=330] (0,.75);
	\draw [ color=\clr, thick, directed=.35] (0,-2.75) to [out=150,in=210] (0,.75);
	\draw [ color=\clr, thick, directed=.65] (0,-4) to (0,-2.75);
	\node at (0,-4.5) {\scriptsize $2$};
	\node at (0,2.5) {\scriptsize $2$};
	\node at (-1.5,-1) {\scriptsize $1$};
	\node at (1.5,-1) {\scriptsize $1$};
	\draw (-0.875,-1) \wdot;
\end{tikzpicture}
};
\endxy.
\]
Combining the above with \cref{dot-on-2-strand} and symmetry, we have \cref{dot-on-k-strand} in case $k=2.$ For general $k$, we use \cref{dots-past-merges} repeatedly to get
\beq\label{dot-past-k-explosion}
\xy
(0,0)*{
\begin{tikzpicture}[color=\clr, scale=.3]
	\draw [ color=\clr, thick, directed=1] (0,.75) to (0,2.5);
	\draw [ color=\clr, thick, directed=.35] (0,-2.75) to [out=30,in=330] (0,.75);
	\draw [ color=\clr, thick, directed=.35] (0,-2.75) to [out=150,in=210] (0,.75);
	\draw [ color=\clr, thick, directed=.65] (0,-4) to (0,-2.75);
	\node at (0,-4.5) {\scriptsize $k$};
	\node at (0,3) {\scriptsize $k$};
	\node at (-1.5,-1) {\scriptsize $1$};
	\node at (1.5,-1) {\scriptsize $1$};
	\node at (0,-1) {\,$\cdots$};
	\draw (0,1.5) \wdot;
\end{tikzpicture}
};
\endxy=
\xy
(0,0)*{
\begin{tikzpicture}[color=\clr, scale=.3]
	\draw [ color=\clr, thick, directed=1] (0,.75) to (0,2);
	\draw [ color=\clr, thick, directed=.35] (0,-2.75) to [out=30,in=330] (0,.75);
	\draw [ color=\clr, thick, directed=.35] (0,-2.75) to [out=150,in=210] (0,.75);
	\draw [ color=\clr, thick, directed=.65] (0,-4) to (0,-2.75);
	\node at (0,-4.5) {\scriptsize $k$};
	\node at (0,2.5) {\scriptsize $k$};
	\node at (-1.5,-1) {\scriptsize $1$};
	\node at (1.5,-1) {\scriptsize $1$};
	\node at (0,-1) {\,$\cdots$};
	\draw (-0.75,-0.25) \wdot;
\end{tikzpicture}
};
\endxy+\cdots+
\xy
(0,0)*{
\begin{tikzpicture}[color=\clr, scale=.3]
	\draw [ color=\clr, thick, directed=1] (0,.75) to (0,2);
	\draw [ color=\clr, thick, directed=.35] (0,-2.75) to [out=30,in=330] (0,.75);
	\draw [ color=\clr, thick, directed=.35] (0,-2.75) to [out=150,in=210] (0,.75);
	\draw [ color=\clr, thick, directed=.65] (0,-4) to (0,-2.75);
	\node at (0,-4.5) {\scriptsize $k$};
	\node at (0,2.5) {\scriptsize $k$};
	\node at (-1.5,-1) {\scriptsize $1$};
	\node at (1.5,-1) {\scriptsize $1$};
	\node at (0,-1) {\,$\cdots$};
	\draw (0.75,-0.25) \wdot;
\end{tikzpicture}
};
\endxy
\eeq
where the sum is over the $k$ different webs with a dot on a unique $1$-strand. By \cref{associativity} and the $k=2$ case, the summands are pairwise equal and we have, for example, 
\beq\label{dot-goes-left}
\xy
(0,0)*{
\begin{tikzpicture}[color=\clr, scale=.3]
	\draw [ color=\clr, thick, directed=1] (0,.75) to (0,2.5);
	\draw [ color=\clr, thick, directed=.35] (0,-2.75) to [out=30,in=330] (0,.75);
	\draw [ color=\clr, thick, directed=.35] (0,-2.75) to [out=150,in=210] (0,.75);
	\draw [ color=\clr, thick, directed=.65] (0,-4) to (0,-2.75);
	\node at (0,-4.5) {\scriptsize $k$};
	\node at (0,3) {\scriptsize $k$};
	\node at (-1.5,-1) {\scriptsize $1$};
	\node at (1.5,-1) {\scriptsize $1$};
	\node at (0,-1) {\,$\cdots$};
	\draw (0,1.5) \wdot;
\end{tikzpicture}
};
\endxy=(k)
\xy
(0,0)*{
\begin{tikzpicture}[color=\clr, scale=.3]
	\draw [ color=\clr, thick, directed=1] (0,.75) to (0,2);
	\draw [ color=\clr, thick, directed=.35] (0,-2.75) to [out=30,in=330] (0,.75);
	\draw [ color=\clr, thick, directed=.35] (0,-2.75) to [out=150,in=210] (0,.75);
	\draw [ color=\clr, thick, directed=.65] (0,-4) to (0,-2.75);
	\node at (0,-4.5) {\scriptsize $k$};
	\node at (0,2.5) {\scriptsize $k$};
	\node at (-1.5,-1) {\scriptsize $1$};
	\node at (1.5,-1) {\scriptsize $1$};
	\node at (0,-1) {\,$\cdots$};
	\draw (-0.75,-0.25) \wdot;
\end{tikzpicture}
};
\endxy
\eeq
where, on the right, only the leftmost 1-strand has a dot. We finish the proof by computing that
\[
\xy
(0,0)*{
\begin{tikzpicture}[color=\clr, scale=.3]
	\draw [ color=\clr, thick, directed=1] (0,-4) to (0,2);
	\node at (0,-4.5) {\scriptsize $k$};
	\node at (0,2.5) {\scriptsize $k$};
	\draw (0,1) \wdot;
\end{tikzpicture}
};
\endxy
 \ \stackrel{\eqref{digon-removal}}{=} \ 
\frac{1}{k!}
\xy
(0,0)*{
\begin{tikzpicture}[color=\clr, scale=.3]
	\draw [ color=\clr, thick, directed=1] (0,.75) to (0,2.5);
	\draw [ color=\clr, thick, directed=.35] (0,-2.75) to [out=30,in=330] (0,.75);
	\draw [ color=\clr, thick, directed=.35] (0,-2.75) to [out=150,in=210] (0,.75);
	\draw [ color=\clr, thick, directed=.65] (0,-4) to (0,-2.75);
	\node at (0,-4.5) {\scriptsize $k$};
	\node at (0,3) {\scriptsize $k$};
	\node at (-1.5,-1) {\scriptsize $1$};
	\node at (1.5,-1) {\scriptsize $1$};
	\node at (0,-1) {\,$\cdots$};
	\draw (0,1.5) \wdot;
\end{tikzpicture}
};
\endxy
 \ \stackrel{\eqref{dot-goes-left}}{=} \ 
\frac{1}{(k-1)!}
\xy
(0,0)*{
\begin{tikzpicture}[color=\clr, scale=.3]
	\draw [ color=\clr, thick, directed=1] (0,.75) to (0,2);
	\draw [ color=\clr, thick, directed=.35] (0,-2.75) to [out=30,in=330] (0,.75);
	\draw [ color=\clr, thick, directed=.35] (0,-2.75) to [out=150,in=210] (0,.75);
	\draw [ color=\clr, thick, directed=.65] (0,-4) to (0,-2.75);
	\node at (0,-4.5) {\scriptsize $k$};
	\node at (0,2.5) {\scriptsize $k$};
	\node at (-1.5,-1) {\scriptsize $1$};
	\node at (1.5,-1) {\scriptsize $1$};
	\node at (0,-1) {\,$\cdots$};
	\draw (-0.75,-0.25) \wdot;
\end{tikzpicture}
};
\endxy
 \ \stackrel{\eqref{digon-removal}}{=} \ 
\xy
(0,0)*{
\begin{tikzpicture}[color=\clr, scale=.3]
	\draw [ color=\clr, thick, directed=1] (0,.75) to (0,2);
	\draw [ color=\clr, thick, directed=.35] (0,-2.75) to [out=30,in=330] (0,.75);
	\draw [ color=\clr, thick, directed=.35] (0,-2.75) to [out=150,in=210] (0,.75);
	\draw [ color=\clr, thick, directed=.65] (0,-4) to (0,-2.75);
	\node at (0,-4.5) {\scriptsize $k$};
	\node at (0,2.5) {\scriptsize $k$};
	\node at (-1.5,-1) {\scriptsize $1$};
	\node at (1.75,-1) { \ \scriptsize $k\!-\!1$};
	\draw (-0.875,-1) \wdot;
\end{tikzpicture}
};
\endxy
\]
and noting that the other side of \cref{dot-on-k-strand} follows by symmetry.
\end{proof}

\begin{lemma}\label{Udot-relations}
The following local relations hold in $\W$ for $h,k,l,s,t \geq 0$: 
\begin{equation}\label{rung-collision}
\xy
(0,0)*{
\begin{tikzpicture}[color=\clr]
	\draw [color=\clr, thick, directed=.15, directed=1, directed=.55] (0,0) to (0,1.75);
	\node at (0,-0.15) {\scriptsize $k$};
	\node at (0,1.9) {\scriptsize $k\!+\!s\!+\!t$ \ };
	\draw [color=\clr, thick, directed=.15, directed=1, directed=.55] (1,0) to (1,1.75);
	\node at (1,-0.15) {\scriptsize $l$};
	\node at (1,1.9) {\scriptsize  \ $l\!-\!s\!-\!t$};
	\draw [color=\clr, thick, directed=.55] (1,0.5) to (0,0.5);
	\node at (0.5,0.25) {\scriptsize$r$};
	\node at (-0.5,0.875) {\scriptsize $k\!+\!s$};
	\draw [color=\clr, thick, directed=.55] (1,1.25) to (0,1.25);
	\node at (0.5,1.5) {\scriptsize$s$};
	\node at (1.5,0.875) {\scriptsize $l\!-\!s$};
\end{tikzpicture}
};
\endxy=\binom{r+s}{s}
\xy
(0,0)*{
\begin{tikzpicture}[color=\clr]
	\draw [color=\clr, thick, directed=.15, directed=1] (0,0) to (0,1.75);
	\node at (0,-0.15) {\scriptsize $k$};
	\node at (0,1.9) {\scriptsize $k\!+\!s\!+\!t$ \ };
	\draw [color=\clr, thick, directed=.15, directed=1] (1,0) to (1,1.75);
	\node at (1,-0.15) {\scriptsize $l$};
	\node at (1,1.9) {\scriptsize  \ $l\!-\!s\!-\!t$};
	\draw [color=\clr, thick, directed=.55] (1,0.75) to (0,0.75);
	\node at (0.5,1) {\scriptsize$s\!+\!t$};
\end{tikzpicture}
};
\endxy,
\end{equation}
\begin{equation}\label{square-switch-double-dots}
\xy
(0,0)*{
\begin{tikzpicture}[color=\clr]
	\draw [color=\clr, thick, directed=.15, directed=1, directed=.55] (0,0) to (0,1.75);
	\node at (0,-0.15) {\scriptsize $k$};
	\node at (0,1.9) {\scriptsize $k$};
	\draw [color=\clr, thick, directed=.15, directed=1, directed=.55] (1,0) to (1,1.75);
	\node at (1,-0.15) {\scriptsize $l$};
	\node at (1,1.9) {\scriptsize $l$};
	\draw [color=\clr, thick, directed=.55] (0,0.5) to (1,0.5);
	\node at (0.5,0.25) {\scriptsize$1$};
	\node at (-0.5,0.875) {\scriptsize $k\!-\!1$};
	\draw [color=\clr, thick, directed=.55] (1,1.25) to (0,1.25);
	\node at (0.5,1.5) {\scriptsize$1$};
	\node at (1.5,0.875) {\scriptsize $l\!+\!1$};
	\draw  (0.75,0.5) \wdot;
	\draw  (0.25,1.25) \wdot;
\end{tikzpicture}
};
\endxy+
\xy
(0,0)*{
\begin{tikzpicture}[color=\clr]
	\draw [color=\clr, thick, directed=.15, directed=1, directed=.55] (0,0) to (0,1.75);
	\node at (0,-0.15) {\scriptsize $k$};
	\node at (0,1.9) {\scriptsize $k$};
	\draw [color=\clr, thick, directed=.15, directed=1, directed=.55] (1,0) to (1,1.75);
	\node at (1,-0.15) {\scriptsize $l$};
	\node at (1,1.9) {\scriptsize $l$};
	\draw [color=\clr, thick, directed=.55] (1,0.5) to (0,0.5);
	\node at (0.5,0.25) {\scriptsize$1$};
	\node at (-0.5,0.875) {\scriptsize $k\!+\!1$};
	\draw [color=\clr, thick, directed=.55] (0,1.25) to (1,1.25);
	\node at (0.5,1.5) {\scriptsize$1$};
	\node at (1.5,0.875) {\scriptsize $l\!-\!1$};
	\draw  (0.75,0.5) \wdot;
	\draw  (0.25,1.25) \wdot;
\end{tikzpicture}
};
\endxy=(k+l)
\xy
(0,0)*{
\begin{tikzpicture}[color=\clr, scale=.3] 
	\draw [color=\clr, thick, directed=1] (1,-2.75) to (1,2.5);
	\draw [color=\clr, thick, directed=1] (-1,-2.75) to (-1,2.5);
	\node at (-1,3) {\scriptsize $k$};
	\node at (1,3) {\scriptsize $l$};
	\node at (-1,-3.3) {\scriptsize $k$};
	\node at (1,-3.3) {\scriptsize $l$};
\end{tikzpicture}
};
\endxy \ , 
\end{equation}
\begin{equation}\label{double-rungs-3}
\xy
(0,0)*{
\begin{tikzpicture}[color=\clr]
	\draw [color=\clr, thick, directed=.15, directed=1] (-1,0) to (-1,1.75);
	\node at (-1,-0.15) {\scriptsize $h$};
	\node at (-1,1.9) {\scriptsize $h\!+\!2$};
	\draw [color=\clr, thick, directed=.15, directed=1, directed=.55] (0,0) to (0,1.75);
	\node at (0,-0.15) {\scriptsize $k$};
	\node at (0,1.9) {\scriptsize $k\!-\!1$};
	\draw [color=\clr, thick, directed=.15, directed=1] (1,0) to (1,1.75);
	\node at (1,-0.15) {\scriptsize $l$};
	\node at (1,1.9) {\scriptsize $l\!-\!1$};
	\draw [color=\clr, thick, directed=.55] (1,0.5) to (0,0.5);
	\node at (0.5,0.25) {\scriptsize$1$};
	\draw [color=\clr, thick, directed=.55] (0,1.25) to (-1,1.25);
	\node at (-0.5,1.5) {\scriptsize$2$};
\end{tikzpicture}
};
\endxy-
\xy
(0,0)*{
\begin{tikzpicture}[color=\clr]
	\draw [color=\clr, thick, directed=.15, directed=1, directed=.55] (-1,0) to (-1,1.75);
	\node at (-1,-0.15) {\scriptsize $h$};
	\node at (-1,1.9) {\scriptsize $h\!+\!2$};
	\draw [color=\clr, thick, directed=.15, directed=1] (0,0) to (0,1.75);
	\node at (0,-0.15) {\scriptsize $k$};
	\node at (0,1.9) {\scriptsize $k\!-\!1$};
	\draw [color=\clr, thick, directed=.15, directed=1] (1,0) to (1,1.75);
	\node at (1,-0.15) {\scriptsize $l$};
	\node at (1,1.9) {\scriptsize $l\!-\!1$};
	\draw [color=\clr, thick, directed=.55] (0,0.5) to (-1,0.5);
	\node at (-0.5,0.25) {\scriptsize$1$};
	\draw [color=\clr, thick, directed=.55] (1,0.875) to (0,0.875);
	\node at (0.5,1.125) {\scriptsize$1$};
	\draw [color=\clr, thick, directed=.55] (0,1.25) to (-1,1.25);
	\node at (-0.5,1.5) {\scriptsize$1$};
\end{tikzpicture}
};
\endxy+
\xy
(0,0)*{
\begin{tikzpicture}[color=\clr]
	\draw [color=\clr, thick, directed=.15, directed=1] (-1,0) to (-1,1.75);
	\node at (-1,-0.15) {\scriptsize $h$};
	\node at (-1,1.9) {\scriptsize $h\!+\!2$};
	\draw [color=\clr, thick, directed=.15, directed=1, directed=.55] (0,0) to (0,1.75);
	\node at (0,-0.15) {\scriptsize $k$};
	\node at (0,1.9) {\scriptsize $k\!-\!1$};
	\draw [color=\clr, thick, directed=.15, directed=1] (1,0) to (1,1.75);
	\node at (1,-0.15) {\scriptsize $l$};
	\node at (1,1.9) {\scriptsize $l\!-\!1$};
	\draw [color=\clr, thick, directed=.55] (1,1.25) to (0,1.25);
	\node at (0.5,1.5) {\scriptsize$1$};
	\draw [color=\clr, thick, directed=.55] (0,0.5) to (-1,0.5);
	\node at (-0.5,0.25) {\scriptsize$2$};
\end{tikzpicture}
};
\endxy=0,
\end{equation}
\begin{equation}\label{double-rungs-4}
\xy
(0,0)*{
\begin{tikzpicture}[color=\clr]
	\draw [color=\clr, thick, directed=.15, directed=1] (-1,0) to (-1,1.75);
	\node at (-1,-0.15) {\scriptsize $h$};
	\node at (-1,1.9) {\scriptsize $h\!+\!2$};
	\draw [color=\clr, thick, directed=.15, directed=1] (0,0) to (0,1.75);
	\node at (0,-0.15) {\scriptsize $k$};
	\node at (0,1.9) {\scriptsize $k\!-\!1$};
	\draw [color=\clr, thick, directed=.15, directed=1] (1,0) to (1,1.75);
	\node at (1,-0.15) {\scriptsize $l$};
	\node at (1,1.9) {\scriptsize $l\!-\!1$};
	\draw [color=\clr, thick, directed=.55] (1,0.5) to (0,0.5);
	\node at (0.5,0.25) {\scriptsize$1$};
	\draw [color=\clr, thick, directed=.55] (0,0.875) to (-1,0.875);
	\node at (-0.5,0.65) {\scriptsize$1$};
	\draw [color=\clr, thick, directed=.55] (0,1.25) to (-1,1.25);
	\node at (-0.5,1.5) {\scriptsize$1$};
	\draw  (-0.75,1.25) \wdot;
\end{tikzpicture}
};
\endxy -
\xy
(0,0)*{
\begin{tikzpicture}[color=\clr]
	\draw [color=\clr, thick, directed=.15, directed=1, directed=.55] (-1,0) to (-1,1.75);
	\node at (-1,-0.15) {\scriptsize $h$};
	\node at (-1,1.9) {\scriptsize $h\!+\!2$};
	\draw [color=\clr, thick, directed=.15, directed=1] (0,0) to (0,1.75);
	\node at (0,-0.15) {\scriptsize $k$};
	\node at (0,1.9) {\scriptsize $k\!-\!1$};
	\draw [color=\clr, thick, directed=.15, directed=1] (1,0) to (1,1.75);
	\node at (1,-0.15) {\scriptsize $l$};
	\node at (1,1.9) {\scriptsize $l\!-\!1$};
	\draw [color=\clr, thick, directed=.55] (0,0.5) to (-1,0.5);
	\node at (-0.5,0.25) {\scriptsize$1$};
	\draw [color=\clr, thick, directed=.55] (1,0.875) to (0,0.875);
	\node at (0.5,1.125) {\scriptsize$1$};
	\draw [color=\clr, thick, directed=.55] (0,1.25) to (-1,1.25);
	\node at (-0.5,1.5) {\scriptsize$1$};
	\draw  (-0.75,1.25) \wdot;
\end{tikzpicture}
};
\endxy-
\xy
(0,0)*{
\begin{tikzpicture}[color=\clr]
	\draw [color=\clr, thick, directed=.15, directed=1, directed=.55] (-1,0) to (-1,1.75);
	\node at (-1,-0.15) {\scriptsize $h$};
	\node at (-1,1.9) {\scriptsize $h\!+\!2$};
	\draw [color=\clr, thick, directed=.15, directed=1] (0,0) to (0,1.75);
	\node at (0,-0.15) {\scriptsize $k$};
	\node at (0,1.9) {\scriptsize $k\!-\!1$};
	\draw [color=\clr, thick, directed=.15, directed=1] (1,0) to (1,1.75);
	\node at (1,-0.15) {\scriptsize $l$};
	\node at (1,1.9) {\scriptsize $l\!-\!1$};
	\draw [color=\clr, thick, directed=.55] (0,0.5) to (-1,0.5);
	\node at (-0.5,0.25) {\scriptsize$1$};
	\draw [color=\clr, thick, directed=.55] (1,0.875) to (0,0.875);
	\node at (0.5,1.125) {\scriptsize$1$};
	\draw [color=\clr, thick, directed=.55] (0,1.25) to (-1,1.25);
	\node at (-0.5,1.5) {\scriptsize$1$};
	\draw  (-0.25,0.5) \wdot;
\end{tikzpicture}
};
\endxy+
\xy
(0,0)*{
\begin{tikzpicture}[color=\clr]
	\draw [color=\clr, thick, directed=.15, directed=1] (-1,0) to (-1,1.75);
	\node at (-1,-0.15) {\scriptsize $h$};
	\node at (-1,1.9) {\scriptsize $h\!+\!2$};
	\draw [color=\clr, thick, directed=.15, directed=1] (0,0) to (0,1.75);
	\node at (0,-0.15) {\scriptsize $k$};
	\node at (0,1.9) {\scriptsize $k\!-\!1$};
	\draw [color=\clr, thick, directed=.15, directed=1] (1,0) to (1,1.75);
	\node at (1,-0.15) {\scriptsize $l$};
	\node at (1,1.9) {\scriptsize $l\!-\!1$};
	\draw [color=\clr, thick, directed=.55] (0,0.5) to (-1,0.5);
	\node at (-0.5,0.25) {\scriptsize$1$};
	\draw [color=\clr, thick, directed=.55] (0,0.875) to (-1,0.875);
	\node at (-0.5,1.125) {\scriptsize$1$};
	\draw [color=\clr, thick, directed=.55] (1,1.25) to (0,1.25);
	\node at (0.5,1.5) {\scriptsize$1$};
	\draw  (-0.25,0.5) \wdot;
\end{tikzpicture}
};
\endxy=0.
\end{equation}
In addition, we have the local relations obtained by reversing all rung orientations\footnote{Relabeling the tops of diagrams as needed to make a valid web.} of the ladders in \cref{rung-collision}.  We also have the local relations obtained from  \cref{double-rungs-3,double-rungs-4} by reversing all rung orientations, reflecting across the vertical line through the middle arrow, and both reversing and reflecting.  Hence \cref{double-rungs-3} and \cref{double-rungs-4} each represent four local relations. 
\end{lemma}

\begin{proof}
We sketch the proof and leave the details to the reader. The proof of \cref{rung-collision} follows from \cref{associativity} and \cref{digon-removal}. The proof of \cref{square-switch-double-dots} is similar to \cite[Lemma 2.10(b)]{TVW}. Its proof involves applications of \cref{square-switch} on the edges labeled $l+1$ and $k+1$ in the webs on the left, followed by two applications of \cref{digon-removal} and the \cref{dumbbell-relation}.

The proof of \cref{double-rungs-3} and \cref{double-rungs-4} are similar to \cite[Lemma 2.10(c)]{TVW}. The former involves using \cref{dumbbell-relation}  on the parallel ladder rungs in the middle web, followed by \cref{associativity},  \cref{additional-relations}\cref{2-dots-zero}, \cref{square-switch}, and two applications of \cref{digon-removal}. The proof of the latter is similar, but it also requires \cref{double-rungs-2}.
\end{proof}


\subsection{Clasp idempotents}\label{clasp-sec}


\begin{definition}\label{clasp-def}
For $k>0$ we locally define the \emph{$k^{\th}$ clasp} to be
\[
Cl_k
=\frac{1}{k!}
\xy
(0,0)*{
\begin{tikzpicture}[color=\clr, scale=.3]
	\draw [ thick, directed=.55] (0,-1) to (0,.75);
	\draw [ thick, directed=1] (0,.75) to [out=30,in=270] (1,2.5);
	\draw [ thick, directed=1] (0,.75) to [out=150,in=270] (-1,2.5); 
	\draw [ thick, directed=.65] (1,-2.75) to [out=90,in=330] (0,-1);
	\draw [ thick, directed=.65] (-1,-2.75) to [out=90,in=210] (0,-1);
	\node at (-1,3) {\scriptsize $1$};
	\node at (0.1,1.75) {$\cdots$};
	\node at (1,3) {\scriptsize $1$};
	\node at (-1,-3.3) {\scriptsize $1$};
	\node at (0.1,-2.35) {$\cdots$};
	\node at (1,-3.3) {\scriptsize $1$};
	\node at (0.75,-0.25) {\scriptsize $k$};
\end{tikzpicture}
};
\endxy
\]
where the dots indicate $k$ separate 1-strands. By \eqref{complete-explosion}, $Cl_k$ is idempotent.
\end{definition}

\begin{lemma}\label{clasp-recursion}
The following local relation holds in $\W$ for $k>1$:
\[
Cl_k
= \frac{k-1}{k}
\xy
(0,0)*{
\begin{tikzpicture}[color=\clr, scale=.3]
	\draw [ thick] (-2,-9) to (-2,-8);
	\draw [ thick] (0,-9) to (0,-8);
	\draw [ thick] (2,-9) to (2,-8);
	\draw [ thick] (4,-9) to (4,-5.85);
	\draw [ thick, directed=1] (-2,-0) to (-2,1);
	\draw [ thick, directed=1] (0,0) to (0,1);
	\draw [ thick, directed=1] (2,0) to (2,1);
	\draw [ thick, directed=1] (4,-2.15) to (4,1);
	\draw [ thick, directed=0.75] (-2,-5.85) to (-2,-2.15);
	\draw [ thick, directed=0.75] (0,-5.85) to (0,-2.15);
	\draw [ thick, directed=0.75] (3,-4.5) to (3,-3.5);
	\draw [ thick] (-2.3,-8) rectangle (2.3,-5.85);
	\draw [ thick] (-2.3,-2.15) rectangle (2.3,0);
	\draw [ thick] (3,-3.5) to [out=30,in=270] (4,-2.15);
	\draw [ thick] (3,-3.5) to [out=150,in=270] (2,-2.15); 
	\draw [ thick] (4,-5.85) to [out=90,in=330] (3,-4.5);
	\draw [ thick] (2,-5.85) to [out=90,in=210] (3,-4.5);
	\node at (-2,-9.55) {\scriptsize $1$};
	\node at (0,-9.55) {\scriptsize $1$};
	\node at (2,-9.55) {\scriptsize $1$};
	\node at (4,-9.55) {\scriptsize $1$};
	\node at (-2,1.55) {\scriptsize $1$};
	\node at (0,1.55) {\scriptsize $1$};
	\node at (2,1.55) {\scriptsize $1$};
	\node at (4,1.55) {\scriptsize $1$};
	\node at (3.65,-4) {\scriptsize $2$};
	\node at (-2.5,-4) {\scriptsize $1$};
	\node at (0.5,-4) {\scriptsize $1$};
	\node at (-1,0.25) { \ $\cdots$\,};
	\node at (-1,-4) { \ $\cdots$\,};
	\node at (-1,-8.5) { \ $\cdots$\,};
	\node at (0,-1.125) {\scriptsize $Cl_{k-1}$};
	\node at (0,-6.975) {\scriptsize $Cl_{k-1}$};
\end{tikzpicture}
};
\endxy
-\frac{k-2}{k}
\xy
(0,0)*{
\begin{tikzpicture}[color=\clr, scale=.3]
	\draw [ thick, ] (-2,-7) to (-2,-6);
	\draw [ thick, ] (0,-7) to (0,-6);
	\draw [ thick, ] (2,-7) to (2,-6);
	\draw [ thick, directed=1] (-2,-2) to (-2,-1);
	\draw [ thick, directed=1] (0,-2) to (0,-1);
	\draw [ thick, directed=1] (2,-2) to (2,-1);
	\draw [ thick, directed=1] (4,-7) to (4,-1);
	\draw [ thick] (-2.3,-6) rectangle (2.3,-2);
	\node at (-2,-7.55) {\scriptsize $1$};
	\node at (0,-7.55) {\scriptsize $1$};
	\node at (2,-7.55) {\scriptsize $1$};
	\node at (4,-7.55) {\scriptsize $1$};
	\node at (-2,-.45) {\scriptsize $1$};
	\node at (0,-.45) {\scriptsize $1$};
	\node at (2,-.45) {\scriptsize $1$};
	\node at (4,-.45) {\scriptsize $1$};
	\node at (-1,-1.75) { \ $\cdots$\,};
	\node at (-1,-6.5) { \ $\cdots$\,};
	\node at (0,-4) { $Cl_{k-1}$};
\end{tikzpicture}
};
\endxy \ .
\]
\end{lemma}

\begin{proof}
We start by computing that
\[
Cl_k=\frac{1}{k!}
\xy
(0,0)*{
\begin{tikzpicture}[color=\clr, scale=.3]
	\draw [ thick, directed=.55] (0,-1) to (0,.75);
	\draw [ thick, directed=1] (0,.75) to [out=30,in=270] (1,2.5);
	\draw [ thick, directed=1] (0,.75) to [out=150,in=270] (-1,2.5); 
	\draw [ thick, directed=.65] (1,-2.75) to [out=90,in=330] (0,-1);
	\draw [ thick, directed=.65] (-1,-2.75) to [out=90,in=210] (0,-1);
	\node at (-1,3) {\scriptsize $1$};
	\node at (0.1,1.75) {$\cdots$};
	\node at (1,3) {\scriptsize $1$};
	\node at (-1,-3.3) {\scriptsize $1$};
	\node at (0.1,-2.35) {$\cdots$};
	\node at (1,-3.3) {\scriptsize $1$};
	\node at (0.75,-0.25) {\scriptsize $k$};
\end{tikzpicture}
};
\endxy=\frac{1}{k!}
\xy
(0,0)*{
\begin{tikzpicture}[color=\clr, scale=.3]
	\draw [ thick, directed=.55] (0,-1) to (0,2.5);
	\draw [ thick, directed=1] (0,2.5) to [out=30,in=270] (1,4.25);
	\draw [ thick, directed=1] (0,2.5) to [out=150,in=270] (-1,4.25); 
	\draw [ thick, directed=.65] (1,-2.75) to [out=90,in=330] (0,-1);
	\draw [ thick, directed=.65] (-1,-2.75) to [out=90,in=210] (0,-1);
	\draw [ thick, directed=1] (0,1.5) to (2,2.5) to (2,4.25);
	\draw [ thick, directed=0.35] (2,-2.75) to (2,-1) to (0,0);
	\node at (-1,4.75) {\scriptsize $1$};
	\node at (0.1,3.5) {$\cdots$};
	\node at (1,4.75) {\scriptsize $1$};
	\node at (2,4.75) {\scriptsize $1$};
	\node at (-1,-3.3) {\scriptsize $1$};
	\node at (0.1,-2.35) {$\cdots$};
	\node at (1,-3.3) {\scriptsize $1$};
	\node at (2,-3.3) {\scriptsize $1$};
	\node at (0.75,0.75) {\scriptsize $k$};
\end{tikzpicture}
};
\endxy\stackrel{\eqref{square-switch}}{=}
\frac{1}{k!}
\xy
(0,0)*{
\begin{tikzpicture}[color=\clr, scale=.3]
	\draw [ thick, directed=.55] (0,-1) to (0,3.25);
	\draw [ thick, directed=1] (0,3.25) to [out=30,in=270] (1,5);
	\draw [ thick, directed=1] (0,3.25) to [out=150,in=270] (-1,5); 
	\draw [ thick, directed=.65] (1,-2.75) to [out=90,in=330] (0,-1);
	\draw [ thick, directed=.65] (-1,-2.75) to [out=90,in=210] (0,-1);
	\draw [ thick, directed=0.65] (3,1.75) to (0,2.75);
	\draw [ thick, directed=0.65] (0,-0.5) to (3,0.5);
	\draw [ thick, directed=0.2, directed=0.525, directed=1] (3,-2.75) to (3,5);
	\node at (-1,5.5) {\scriptsize $1$};
	\node at (0.1,4.25) {$\cdots$};
	\node at (1,5.5) {\scriptsize $1$};
	\node at (3,5.5) {\scriptsize $1$};
	\node at (-1,-3.3) {\scriptsize $1$};
	\node at (0.1,-2.35) {$\cdots$};
	\node at (1,-3.3) {\scriptsize $1$};
	\node at (3,-3.3) {\scriptsize $1$};
	\node at (1.5,-0.75) {\ss $1$};
	\node at (1.5,3.25) {\ss $1$};
	\node at (3.75,1.15) {\ss $2$};
	\node at (-1.5,1) {\scriptsize $k\!-\!1$};
\end{tikzpicture}
};
\endxy-\frac{k-2}{k!}
\xy
(0,0)*{
\begin{tikzpicture}[color=\clr, scale=.3]
	\draw [ thick, directed=.55] (0,-1) to (0,2.5);
	\draw [ thick, directed=1] (0,2.5) to [out=30,in=270] (1,4.25);
	\draw [ thick, directed=1] (0,2.5) to [out=150,in=270] (-1,4.25); 
	\draw [ thick, directed=.65] (1,-2.75) to [out=90,in=330] (0,-1);
	\draw [ thick, directed=.65] (-1,-2.75) to [out=90,in=210] (0,-1);
	\draw [ thick, directed=1] (2,-2.75) to (2,4.25);
	\node at (-1,4.75) {\scriptsize $1$};
	\node at (0.1,3.5) {$\cdots$};
	\node at (1,4.75) {\scriptsize $1$};
	\node at (2,4.75) {\scriptsize $1$};
	\node at (-1,-3.3) {\scriptsize $1$};
	\node at (0.1,-2.35) {$\cdots$};
	\node at (1,-3.3) {\scriptsize $1$};
	\node at (2,-3.3) {\scriptsize $1$};
	\node at (-1.5,0.75) {\scriptsize $k\!-\!1$};
\end{tikzpicture}
};
\endxy \ .
\]
Applying \eqref{complete-explosion} to the $(k-1)$-strand of the web on the left, and then the definition of $Cl_{k-1}$ to both webs yields
\[
\frac{((k-1)!)^2}{k!(k-2)!}
\xy
(0,0)*{
\begin{tikzpicture}[color=\clr, scale=.3]
	\draw [ thick] (-2,-9) to (-2,-8);
	\draw [ thick] (0,-9) to (0,-8);
	\draw [ thick] (2,-9) to (2,-8);
	\draw [ thick] (4,-9) to (4,-5.85);
	\draw [ thick, directed=1] (-2,-0) to (-2,1);
	\draw [ thick, directed=1] (0,0) to (0,1);
	\draw [ thick, directed=1] (2,0) to (2,1);
	\draw [ thick, directed=1] (4,-2.15) to (4,1);
	\draw [ thick, directed=0.75] (-2,-5.85) to (-2,-2.15);
	\draw [ thick, directed=0.75] (0,-5.85) to (0,-2.15);
	\draw [ thick, directed=0.75] (3,-4.5) to (3,-3.5);
	\draw [ thick] (-2.3,-8) rectangle (2.3,-5.85);
	\draw [ thick] (-2.3,-2.15) rectangle (2.3,0);
	\draw [ thick] (3,-3.5) to [out=30,in=270] (4,-2.15);
	\draw [ thick] (3,-3.5) to [out=150,in=270] (2,-2.15); 
	\draw [ thick] (4,-5.85) to [out=90,in=330] (3,-4.5);
	\draw [ thick] (2,-5.85) to [out=90,in=210] (3,-4.5);
	\node at (-2,-9.55) {\scriptsize $1$};
	\node at (0,-9.55) {\scriptsize $1$};
	\node at (2,-9.55) {\scriptsize $1$};
	\node at (4,-9.55) {\scriptsize $1$};
	\node at (-2,1.55) {\scriptsize $1$};
	\node at (0,1.55) {\scriptsize $1$};
	\node at (2,1.55) {\scriptsize $1$};
	\node at (4,1.55) {\scriptsize $1$};
	\node at (3.65,-4) {\scriptsize $2$};
	\node at (-2.5,-4) {\scriptsize $1$};
	\node at (0.5,-4) {\scriptsize $1$};
	\node at (-1,0.25) { \ $\cdots$\,};
	\node at (-1,-4) { \ $\cdots$\,};
	\node at (-1,-8.5) { \ $\cdots$\,};
	\node at (0,-1.125) {\scriptsize $Cl_{k-1}$};
	\node at (0,-6.975) {\scriptsize $Cl_{k-1}$};
\end{tikzpicture}
};
\endxy-\frac{(k-2)(k-1)!}{k!}
\xy
(0,0)*{
\begin{tikzpicture}[color=\clr, scale=.3]
	\draw [ thick, ] (-2,-7) to (-2,-6);
	\draw [ thick, ] (0,-7) to (0,-6);
	\draw [ thick, ] (2,-7) to (2,-6);
	\draw [ thick, directed=1] (-2,-2) to (-2,-1);
	\draw [ thick, directed=1] (0,-2) to (0,-1);
	\draw [ thick, directed=1] (2,-2) to (2,-1);
	\draw [ thick, directed=1] (4,-7) to (4,-1);
	\draw [ thick] (-2.3,-6) rectangle (2.3,-2);
	\node at (-2,-7.55) {\scriptsize $1$};
	\node at (0,-7.55) {\scriptsize $1$};
	\node at (2,-7.55) {\scriptsize $1$};
	\node at (4,-7.55) {\scriptsize $1$};
	\node at (-2,-.45) {\scriptsize $1$};
	\node at (0,-.45) {\scriptsize $1$};
	\node at (2,-.45) {\scriptsize $1$};
	\node at (4,-.45) {\scriptsize $1$};
	\node at (-1,-1.75) { \ $\cdots$\,};
	\node at (-1,-6.5) { \ $\cdots$\,};
	\node at (0,-4) { $Cl_{k-1}$};
\end{tikzpicture}
};
\endxy \ ,
\]
which simplifies to the right side of the equation in the statement of the lemma.
\end{proof}


\subsection{Crossings and Sergeev diagrams}\label{crossing-sec}


\begin{definition}\label{crossing-def}
We locally define the \emph{crossing} of 1-strands to be
\[
\xy
(0,0)*{
\bt[color=\clr, scale=1.25]
	\draw[thick, directed=1] (0,0) to (0.5,0.5);
	\draw[thick, directed=1] (0.5,0) to (0,0.5);
	\node at (0,-0.15) {\scriptsize $1$};
	\node at (0,0.65) {\scriptsize $1$};
	\node at (0.5,-0.15) {\scriptsize $1$};
	\node at (0.5,0.65) {\scriptsize $1$};
\et
};
\endxy
 \ := \ 
\xy
(0,0)*{
\begin{tikzpicture}[color=\clr, scale=.3]
	\draw [ thick, directed=.75] (0,0.25) to (0,1.25);
	\draw [ thick, directed=1] (0,1.25) to [out=30,in=270] (1,2.5);
	\draw [ thick, directed=1] (0,1.25) to [out=150,in=270] (-1,2.5); 
	\draw [ thick, directed=.65] (1,-1) to [out=90,in=330] (0,0.25);
	\draw [ thick, directed=.65] (-1,-1) to [out=90,in=210] (0,0.25);
	\node at (-1,3) {\scriptsize $1$};
	\node at (1,3) {\scriptsize $1$};
	\node at (-1,-1.5) {\scriptsize $1$};
	\node at (1,-1.5) {\scriptsize $1$};
	\node at (-0.75,0.75) {\scriptsize $2$};
\end{tikzpicture}
};
\endxy
 \ - \ 
 \xy
(0,0)*{
\begin{tikzpicture}[color=\clr, scale=.3]
	\draw [thick, directed=1] (-2,-1) to (-2,2.5);
	\draw [thick, directed=1] (-4,-1) to (-4,2.5);
	\node at (-2,-1.5) {\scriptsize $1$};
	\node at (-2,3) {\scriptsize $1$};
	\node at (-4,-1.5) {\scriptsize $1$};
	\node at (-4,3) {\scriptsize $1$};
\end{tikzpicture}
};
\endxy.
\]
Using this, we locally define for $k,l>0$ the \emph{crossing} of a $k$- and an $l$-strand to be
\[
\xy
(0,0)*{
\bt[scale=1.25, color=\clr]
	\draw[thick, directed=1] (0,0) to (0.5,0.5);
	\draw[thick, directed=1] (0.5,0) to (0,0.5);
	\node at (0,-0.15) {\scriptsize $k$};
	\node at (0,0.65) {\scriptsize $l$};
	\node at (0.5,-0.15) {\scriptsize $l$};
	\node at (0.5,0.65) {\scriptsize $k$};
\et
};
\endxy :=\frac{1}{k!\,l!}
\xy
(0,0)*{
\bt[scale=.35, color=\clr]
	\draw [ thick, directed=1] (0, .75) to (0,1.5);
	\draw [ thick, directed=0.75] (1,-1) to [out=90,in=330] (0,.75);
	\draw [ thick, directed=0.75] (-1,-1) to [out=90,in=210] (0,.75);
	\draw [ thick, directed=1] (4, .75) to (4,1.5);
	\draw [ thick, directed=0.75] (5,-1) to [out=90,in=330] (4,.75);
	\draw [ thick, directed=0.75] (3,-1) to [out=90,in=210] (4,.75);
	\draw [ thick, directed=0.75] (0,-6.5) to (0,-5.75);
	\draw [ thick, ] (0,-5.75) to [out=30,in=270] (1,-4);
	\draw [ thick, ] (0,-5.75) to [out=150,in=270] (-1,-4); 
	\draw [ thick, directed=0.75] (4,-6.5) to (4,-5.75);
	\draw [ thick, ] (4,-5.75) to [out=30,in=270] (5,-4);
	\draw [ thick, ] (4,-5.75) to [out=150,in=270] (3,-4); 
	\draw [ thick ] (5,-1) to (1,-4);
	\draw [ thick ] (3,-1) to (-1,-4);
	\draw [ thick ] (1,-1) to (5,-4);
	\draw [ thick ] (-1,-1) to (3,-4);
	\node at (2.6, -0.5) {\scriptsize $1$};
	\node at (5.4, -0.5) {\scriptsize $1$};
	\node at (1.4, -0.5) {\scriptsize $1$};
	\node at (-1.4, -0.5) {\scriptsize $1$};
	\node at (2.6, -4.5) {\scriptsize $1$};
	\node at (5.4, -4.5) {\scriptsize $1$};
	\node at (1.4, -4.5) {\scriptsize $1$};
	\node at (-1.4, -4.5) {\scriptsize $1$};
	\node at (0.1, -0.65) { $\cdots$};
	\node at (4.1, -0.65) { $\cdots$};
	\node at (0.1, -4.6) { $\cdots$};
	\node at (4.1, -4.6) { $\cdots$};
	\node at (0,2) {\scriptsize $l$};
	\node at (4,2) {\scriptsize $k$};
	\node at (0,-7) {\scriptsize $k$};
	\node at (4,-7) {\scriptsize $l$};
\et
};
\endxy.
\]
\end{definition}

In this paper, we will only need to consider crossings with at least one of $k$ or $l$ equal to 1. Nevertheless, the above definition is justified by the next result. Recall from Lemma \ref{beta-theta-maps} the element $\a_\omega:=(1,\dots,1)\in\A(r)$.

\begin{proposition}\label{xi-map}
There exists a surjective homomorphism $\xi\colon\H\onto\a_\omega\W\a_\omega$ given by
\[
s_i\mapsto
\xy
(0,0)*{
\bt[color=\clr, scale=1.25]
	\draw[thick, directed=1] (-0.75,0) to (-0.75,0.5);
	\draw[thick, directed=1] (-0.25,0) to (-0.25,0.5);
	\draw[thick, directed=1] (0.75,0) to (0.75,0.5);
	\draw[thick, directed=1] (1.25,0) to (1.25,0.5);
	\draw[thick, directed=1] (0,0) to (0.5,0.5);
	\draw[thick, directed=1] (0.5,0) to (0,0.5);
	\node at (0,-0.15) {\scriptsize $1$};
	\node at (0,0.65) {\scriptsize $1$};
	\node at (0.5,-0.15) {\scriptsize $1$};
	\node at (0.5,0.65) {\scriptsize $1$};
	\node at (-0.25,-0.15) {\scriptsize $1$};
	\node at (-0.25,0.65) {\scriptsize $1$};
	\node at (-0.75,-0.15) {\scriptsize $1$};
	\node at (-0.75,0.65) {\scriptsize $1$};
	\node at (0.75,-0.15) {\scriptsize $1$};
	\node at (0.75,0.65) {\scriptsize $1$};
	\node at (1.25,-0.15) {\scriptsize $1$};
	\node at (1.25,0.65) {\scriptsize $1$};
	\node at (-0.5,0.25) { \ $\cdots$\,};
	\node at (1,0.25) { \ $\cdots$\,};
\et
};
\endxy \ ,
\quad c_j\mapsto
\xy
(0,0)*{
\bt[color=\clr, scale=1.25]
	\draw[thick, directed=1] (-0.75,0) to (-0.75,0.5);
	\draw[thick, directed=1] (-0.25,0) to (-0.25,0.5);
	\draw[thick, directed=1] (0.25,0) to (0.25,0.5);
	\draw[thick, directed=1] (0.75,0) to (0.75,0.5);
	\draw[thick, directed=1] (0,0) to (0,0.5);
	\node at (0,-0.15) {\scriptsize $1$};
	\node at (0,0.65) {\scriptsize $1$};
	\node at (-0.25,-0.15) {\scriptsize $1$};
	\node at (-0.25,0.65) {\scriptsize $1$};
	\node at (-0.75,-0.15) {\scriptsize $1$};
	\node at (-0.75,0.65) {\scriptsize $1$};
	\node at (0.25,-0.15) {\scriptsize $1$};
	\node at (0.25,0.65) {\scriptsize $1$};
	\node at (.75,-0.15) {\scriptsize $1$};
	\node at (.75,0.65) {\scriptsize $1$};
	\node at (-0.5,0.25) { \ $\cdots$\,};
	\node at (0.5,0.25) { \ $\cdots$\,};
	\draw (0,0.25) \wdot;
\et
};
\endxy
\]
for $1\leq i\leq r-1$ and $1\leq j\leq r$, where $s_i$ crosses the $i^{\th}$ and $(i+1)^{\st}$ strands and $c_j$ has a dot on the $j^{\th}$ strand. In particular, the following local relations hold in $\W$:
\beq\label{dots-past-crossings}
\xy
(0,0)*{
\bt[color=\clr, scale=1.25]
	\draw[thick, directed=1] (0,0) to (0.5,1);
	\draw[thick, directed=1] (0.5,0) to (0,1);
	\node at (0,-0.15) {\scriptsize $1$};
	\node at (0,1.15) {\scriptsize $1$};
	\node at (0.5,-0.15) {\scriptsize $1$};
	\node at (0.5,1.15) {\scriptsize $1$};
	\draw (0.15,0.3) \wdot;
\et
};
\endxy=
\xy
(0,0)*{
\bt[color=\clr, scale=1.25]
	\draw[thick, directed=1] (0,0) to (0.5,1);
	\draw[thick, directed=1] (0.5,0) to (0,1);
	\node at (0,-0.15) {\scriptsize $1$};
	\node at (0,1.15) {\scriptsize $1$};
	\node at (0.5,-0.15) {\scriptsize $1$};
	\node at (0.5,1.15) {\scriptsize $1$};
	\draw (0.35,0.7) \wdot;
\et
};
\endxy \ ,\quad\quad
\xy
(0,0)*{
\bt[color=\clr, scale=1.25]
	\draw[thick, directed=1] (0,0) to (0.5,1);
	\draw[thick, directed=1] (0.5,0) to (0,1);
	\node at (0,-0.15) {\scriptsize $1$};
	\node at (0,1.15) {\scriptsize $1$};
	\node at (0.5,-0.15) {\scriptsize $1$};
	\node at (0.5,1.15) {\scriptsize $1$};
	\draw (0.15,0.7) \wdot;
\et
};
\endxy=
\xy
(0,0)*{
\bt[color=\clr, scale=1.25]
	\draw[thick, directed=1] (0,0) to (0.5,1);
	\draw[thick, directed=1] (0.5,0) to (0,1);
	\node at (0,-0.15) {\scriptsize $1$};
	\node at (0,1.15) {\scriptsize $1$};
	\node at (0.5,-0.15) {\scriptsize $1$};
	\node at (0.5,1.15) {\scriptsize $1$};
	\draw (0.35,0.3) \wdot;
\et
};
\endxy \ ,
\eeq
\beq\label{braid-relation}
\xy
(0,0)*{
\bt[color=\clr, scale=1.25]
	\draw[thick, directed=1] (0,0) to (0.5,0.5);
	\draw[thick, directed=1] (0.5,0) to (0,0.5);
	\draw[thick, ] (0,-0.5) to (0.5,0);
	\draw[thick, ] (0.5,-0.5) to (0,0);
	\node at (0,-0.65) {\scriptsize $1$};
	\node at (0,.65) {\scriptsize $1$};
	\node at (0.5,-.65) {\scriptsize $1$};
	\node at (0.5,.65) {\scriptsize $1$};
\et
};
\endxy=
\xy
(0,0)*{
\bt[color=\clr, scale=1.25]
	\draw[thick, directed=1] (0.5,-0.5) to (0.5,0.5);
	\draw[thick, rdirected=0.05] (0,0.5) to (0,-0.5);
	\node at (0,-.65) {\scriptsize $1$};
	\node at (0,.65) {\scriptsize $1$};
	\node at (0.5,-.65) {\scriptsize $1$};
	\node at (0.5,.65) {\scriptsize $1$};
\et
};
\endxy \ ,\quad\quad
\xy
(0,0)*{
\bt[color=\clr, scale=1.25]
	\draw[thick, directed=1] (0,0) to (1,1) to (1,1.5);
	\draw[thick, directed=1] (0.5,0) to (0,0.5) to (0,1) to (0.5,1.5);
	\draw[thick, directed=1] (1,0) to (1,0.5) to (0.5,1) to (0,1.5);
	\node at (0,-0.15) {\scriptsize $1$};
	\node at (0,1.65) {\scriptsize $1$};
	\node at (0.5,-0.15) {\scriptsize $1$};
	\node at (0.5,1.65) {\scriptsize $1$};
	\node at (1,-0.15) {\scriptsize $1$};
	\node at (1,1.65) {\scriptsize $1$};
\et
};
\endxy=
\xy
(0,0)*{
\bt[color=\clr, scale=1.25]
	\draw[thick, directed=1] (0,0) to (0,0.5) to (1,1.5);
	\draw[thick, directed=1] (0.5,0) to (1,0.5) to (1,1) to (0.5,1.5);
	\draw[thick, directed=1] (1,0) to (0,1) to (0,1.5);
	\node at (0,-0.15) {\scriptsize $1$};
	\node at (0,1.65) {\scriptsize $1$};
	\node at (0.5,-0.15) {\scriptsize $1$};
	\node at (0.5,1.65) {\scriptsize $1$};
	\node at (1,-0.15) {\scriptsize $1$};
	\node at (1,1.65) {\scriptsize $1$};
\et
};
\endxy \ .
\eeq
\end{proposition}

\begin{proof}
That the images of relations \eqref{hecke-relations} hold in $\W$ can be verified by direct calculations using the relations of $\W$, so we leave it to the reader. This implies $\xi$ is a well-defined homomorphism.

To show $\xi$ is surjective, we prove that every web $u\in\a_\omega\W\a_\omega$ can be expressed as a linear combination of products of $\xi(s_i)+\a_\omega$ and $\xi(c_j)$ for various $i,j$. Indeed, \eqref{dot-on-k-strand} ensures that every dot in $u$ lying on a strand with label greater than 1 can be moved onto a 1-strand. Next, for every merge in $u$ we apply \eqref{complete-explosion} to all three edges:
\[
\xy
(0,0)*{
\begin{tikzpicture}[color=\clr, scale=.5]
	\draw [ thick, directed=1] (-1,3) to (-1,6);
	\draw [ thick, rdirected=.55] (-1,3) to (-3,0);
	\draw [ thick, rdirected=.55] (-1,3) to (1,0);
	\node at (-1,6.35) {\scriptsize $k\! +\! l$};
	\node at (-3,-.35) {\scriptsize $k$};
	\node at (1,-.35) {\scriptsize $l$};
\end{tikzpicture}
};
\endxy
=\frac{1}{k! \, l! \, (k+l)!}
\xy
(0,0)*{
\begin{tikzpicture}[color=\clr, scale=.5]
	\draw [ thick, directed=1] (-1,5.5) to (-1,6);
	\draw [ thick, directed=.65] (-1,3) to (-1,3.5);
	\draw [ thick, directed=.35] (-1,3.5) to [out=150,in=210] (-1,5.5);
	\draw [ thick, directed=.35] (-1,3.5) to [out=30,in=330] (-1,5.5);
	\draw [ thick, rdirected=.65] (-1,3) to (-1.45,2.333);
	\draw [ thick, rdirected=.55] (-2.55,.67) to (-3,0);
	\draw [ thick, directed=.75] (-2.55,.67) to [out=116.3,in=176.3] (-1.45,2.333);
	\draw [ thick, directed=.75] (-2.55,.67) to [out=-3.7,in=296.3] (-1.45,2.333);
	\draw [ thick, rdirected=.65] (-1,3) to (-.55,2.333);
	\draw [ thick, rdirected=.55] (0.55,.67) to (1,0);
	\draw [ thick, directed=.75] (0.55,.67) to [out=187.7,in=243.7] (-.55,2.333);
	\draw [ thick, directed=.75] (0.55,.67) to [out=63.7,in=363.7] (-.55,2.333);
	\draw [dashed] (-3,0.85) rectangle (1,5);
	\node at (-1,6.35) {\scriptsize $k\! +\! l$};
	\node at (-3,-.35) {\scriptsize $k$};
	\node at (1,-.35) {\scriptsize $l$};
	\node at (-1.85,4.5) {\scriptsize $1$};
	\node at (-0.15,4.5) {\scriptsize $1$};
	\node at (-0.95,4.5) {\small $\cdots$};
	\node at (0.05,1.55) { \rotatebox{33.7}{\small $\cdots$}};
	\node at (-0.7,1.15) {\scriptsize $1$};
	\node at (0.7,1.85) {\scriptsize $1$};
	\node at (-2,1.45) { \rotatebox{-33.7}{\small $\cdots$}};
	\node at (-1.3,1.15) {\scriptsize $1$};
	\node at (-2.7,1.85) {\scriptsize $1$};
\end{tikzpicture}
};
\endxy.
\]
By associativity, the web enclosed by the dashed rectangle above is $(k+l)!\,Cl_{k+l}$, which, after finitely many iterations of the recursion in Lemma \ref{clasp-recursion}, can be written in the desired form. Finally, repeating this process for every split in $u$ finishes the proof.
\end{proof}

We will show later that $\xi$ is an isomorphism (Corollary \ref{xi-iso}). As such we will abuse notation and denote $\xi(w)$ simply as $w$ for $w\in\H$. The same goes for $w\in\H_c(k)$, $k<r$, which can be mapped into $\a_\omega\W\a_\omega$ by precomposing $\xi$ with the canonical embedding $\H_c(k)\into\H_c(r)$. Further, we will refer to the images under $\xi$ of elements of the standard basis of $\H$ as \emph{Sergeev diagrams}. An example of a Sergeev diagram for $r=6$ is
\[
c_1c_3c_5s_1s_4s_3s_5s_4s_2s_3=
\xy
(0,0)*{
\bt[scale=0.5, color=\clr]
	\draw [thick, directed=1] (0,0) to (1,1.5) to (1,2);
	\draw [thick,directed=1] (1,0) to (4,1.5) to (4,2);
	\draw [thick,directed=1] (2,0) to (5,1.5) to (5,2);
	\draw [thick,directed=1] (3,0) to (0,1.5) to (0,2);
	\draw [thick,directed=1] (4,0) to (2,1.5) to (2,2);
	\draw [thick,directed=1] (5,0) to (3,1.5) to (3,2);
	\draw (0,1.5) \wdot;
	\draw (2,1.5) \wdot;
	\draw (4,1.5) \wdot;
	\node at (0,-0.4) {\ss $1$};
	\node at (1,-0.4) {\ss $1$};
	\node at (2,-0.4) {\ss $1$};
	\node at (3,-0.4) {\ss $1$};
	\node at (4,-0.4) {\ss $1$};
	\node at (5,-0.4) {\ss $1$};
	\node at (0,2.4) {\ss $1$};
	\node at (1,2.4) {\ss $1$};
	\node at (2,2.4) {\ss $1$};
	\node at (3,2.4) {\ss $1$};
	\node at (4,2.4) {\ss $1$};
	\node at (5,2.4) {\ss $1$};
\et
};
\endxy.
\]

\begin{lemma}\label{clasp-sum}
The following local relations hold in $\W$.
\bi
\item[(a)] For $0<k\leq r$, 
\beq\label{clasp-sum}
Cl_k=\frac{1}{k!}\sum_{\sigma\in\frakS_k}\sigma.
\eeq
\item[(b)] For $0<k\leq r$ and $\sigma\in\frakS_k$,
\beq\label{untangle}
\xy
(0,0)*{
\begin{tikzpicture}[scale=.35, color=\clr]
	\draw [ thick, ] (0,-5) to (0,-4);
	\draw [ thick, ] (2,-5) to (2,-4);
	\draw [thick, directed=0.75] (0,-2) to [out=90, in=210] (1,-0.5);
	\draw [thick, directed=0.75] (2,-2) to [out=90, in=330] (1,-0.5);
	\draw [thick, directed=1] (1,-0.5) to (1,0.5);
	\draw [ thick] (-0.3,-4) rectangle (2.3,-2);
	\node at (1,-3) {$\sigma$};
	\node at (1,-1.65) { \,$\cdots$};
	\node at (1,-4.6) { \,$\cdots$};
	\node at (0,-5.5) {\ss $1$};
	\node at (2,-5.5) {\ss $1$};
	\node at (1,1) {\fs $k$};
\end{tikzpicture}
};
\endxy=
\xy
(0,0)*{
\begin{tikzpicture}[scale=.35, color=\clr]
	\draw [thick, directed=0.75] (0,-2) to [out=90, in=210] (1,-0.5);
	\draw [thick, directed=0.75] (2,-2) to [out=90, in=330] (1,-0.5);
	\draw [thick, directed=1] (1,-0.5) to (1,0.5);
	\node at (1,-1.65) { \,$\cdots$};
	\node at (0,-2.5) {\ss $1$};
	\node at (2,-2.5) {\ss $1$};
	\node at (1,1) {\fs $k$};
\end{tikzpicture}
};
\endxy,\quad\quad
\xy
(0,0)*{\rotatebox{180}{
\begin{tikzpicture}[scale=.35, color=\clr]
	\draw [ thick, directed=1] (0,-4) to (0,-5);
	\draw [ thick, directed=1] (2,-4) to (2,-5);
	\draw [thick, ] (0,-2) to [out=90, in=210] (1,-0.5);
	\draw [thick, ] (2,-2) to [out=90, in=330] (1,-0.5);
	\draw [thick, rdirected=0.55] (1,-0.5) to (1,0.5);
	\draw [ thick] (-0.3,-4) rectangle (2.3,-2);
	\node at (1,-3) {\rotatebox{180}{$\sigma$}};
	\node at (1,-1.65) { \,$\cdots$};
	\node at (1,-4.4) { \,$\cdots$};
	\node at (0,-5.5) {\rotatebox{180}{\ss $1$}};
	\node at (2,-5.5) {\rotatebox{180}{\ss $1$}};
	\node at (1,1) {\rotatebox{180}{\fs $k$}};
\end{tikzpicture}
}};
\endxy=
\xy
(0,0)*{\rotatebox{180}{
\begin{tikzpicture}[scale=.35, color=\clr]
	\draw [thick, directed=1] (1,-0.5) to [out=210, in=90] (0,-2);
	\draw [thick, directed=1] (1,-0.5) to [out=330, in=90] (2,-2);
	\draw [thick, rdirected=0.55] (1,-0.5) to (1,0.5);
	\node at (1,-1.5) { \,$\cdots$};
	\node at (0,-2.5) {\rotatebox{180}{\ss $1$}};
	\node at (2,-2.5) {\rotatebox{180}{\ss $1$}};
	\node at (1,1) {\rotatebox{180}{\fs $k$}};
\end{tikzpicture}
}};
\endxy.
\eeq
\item[(c)] For $k,l>0$,
\beq\label{merges-past-crossings}
\xy
(0,0)*{
\bt[scale=.35, color=\clr]
	\draw [ thick, directed=1] (0, .75) to (0,1.5) to (2,3.5);
	\draw [ thick, directed=.65] (1,-1) to [out=90,in=330] (0,.75);
	\draw [ thick, directed=.65] (-1,-1) to [out=90,in=210] (0,.75);
	\draw [ thick, directed=1] (2,-1) to (2,1.5) to (0,3.5);
	\node at (2, 4) {\scriptsize $k\! +\! l$};
	\node at (-1,-1.5) {\scriptsize $k$};
	\node at (1,-1.5) {\scriptsize $l$};
	\node at (2,-1.5) {\scriptsize $1$};
	\node at (0,4) {\scriptsize $1$};
\et
};
\endxy=
\xy
(0,0)*{
\bt[scale=.35, color=\clr]
	\draw [ thick, directed=1] (0, .75) to (0,1.5);
	\draw [ thick, directed=.65] (1,-1) to [out=90,in=330] (0,.75);
	\draw [ thick, directed=.65] (-1,-1) to [out=90,in=210] (0,.75);
	\draw [ thick, ] (-3,-3) to (-1,-1);
	\draw [ thick, ] (-1,-3) to (1,-1);
	\draw [ thick, directed=1] (1,-3) to (-2,-1) to (-2,1.5);
	\node at (0, 2) {\scriptsize $k\! +\! l$};
	\node at (-1,-3.5) {\scriptsize $l$};
	\node at (-3,-3.5) {\scriptsize $k$};
	\node at (-2,2) {\scriptsize $1$};
	\node at (1,-3.5) {\scriptsize $1$};
\et
};
\endxy,\quad\quad
\xy
(0,0)*{\rotatebox{180}{
\bt[scale=.35, color=\clr]
	\draw [ thick, rdirected=.6] (0, .75) to (0,1.5);
	\draw [ thick, ] (1,-1) to [out=90,in=330] (0,.75);
	\draw [ thick, ] (-1,-1) to [out=90,in=210] (0,.75);
	\draw [ thick, rdirected=0.1] (-3,-3) to (-1,-1);
	\draw [ thick, rdirected=0.1] (-1,-3) to (1,-1);
	\draw [ thick, directed=1] (-2,1.5) to (-2,-1) to (1,-3);
	\node at (0, 2) {\rotatebox{180}{\scriptsize $k\! +\! l$}};
	\node at (-1,-3.5) {\rotatebox{180}{\scriptsize $k$}};
	\node at (-3,-3.5) {\rotatebox{180}{\scriptsize $l$}};
	\node at (-2,2) {\rotatebox{180}{\scriptsize $1$}};
	\node at (1,-3.5) {\rotatebox{180}{\scriptsize $1$}};
\et
}};
\endxy=
\xy
(0,0)*{\rotatebox{180}{
\bt[scale=.35, color=\clr]
	\draw [ thick, rdirected=.15] (0, .75) to (0,1.5) to (2,3.5);
	\draw [ thick, rdirected=.1] (1,-1) to [out=90,in=330] (0,.75);
	\draw [ thick, rdirected=.1] (-1,-1) to [out=90,in=210] (0,.75);
	\draw [ thick, directed=1] (0,3.5) to (2,1.5) to (2,-1);
	\node at (2, 4) {\rotatebox{180}{\scriptsize $k\! +\! l$}};
	\node at (-1,-1.5) {\rotatebox{180}{\scriptsize $l$}};
	\node at (1,-1.5) {\rotatebox{180}{\scriptsize $k$}};
	\node at (2,-1.5) {\rotatebox{180}{\scriptsize $1$}};
	\node at (0,4) {\rotatebox{180}{\scriptsize $1$}};
\et
}};
\endxy,
\eeq
plus the local relations obtained by reflecting both of the above about a vertical axis.
\ei
\end{lemma}

\begin{proof}
The proof of (a) is by induction on $k$ and completely straightforward, so left to the reader.

To prove (b), we first compute that
\beq\label{untwist-crossing}
\xy
(0,0)*{
\bt[scale=.35, color=\clr]
	\draw [ thick, directed=1] (0,2.75) to (0,3.5);
	\draw [ thick, directed=0.75] (1,1.75) to [out=90,in=330] (0,2.75);
	\draw [ thick, directed=0.75] (-1,1.75) to [out=90,in=210] (0,2.75);
	\draw [thick] (-1,-1) to (1,1.75);
	\draw [thick] (1,-1) to (-1,1.75);
	\node at (0,3.95) {\scriptsize $2$};
	\node at (-1,-1.45) {\scriptsize $1$};
	\node at (1,-1.45) {\scriptsize $1$};
\et
};
\endxy=
\xy
(0,0)*{
\begin{tikzpicture}[scale=.35, color=\clr]
	\draw [ thick, ] (0,0) to (0,.75);
	\draw [ thick, ] (0,.75) to [out=30,in=270] (1,1.75);
	\draw [ thick, ] (0,.75) to [out=150,in=270] (-1,1.75); 
	\draw [ thick, directed=0.75] (1,-1) to [out=90,in=330] (0,0);
	\draw [ thick, directed=0.75] (-1,-1) to [out=90,in=210] (0,0);
	\draw [ thick, directed=1] (0,2.75) to (0,3.5);
	\draw [ thick, directed=0.75] (1,1.75) to [out=90,in=330] (0,2.75);
	\draw [ thick, directed=0.75] (-1,1.75) to [out=90,in=210] (0,2.75);
	\node at (-1.5,1.75) {\scriptsize $1$};
	\node at (1.5,1.75) {\scriptsize $1$};
	\node at (-1,-1.45) {\scriptsize $1$};
	\node at (1,-1.45) {\scriptsize $1$};
	\node at (0,3.95) {\scriptsize $2$};
\end{tikzpicture}
};
\endxy-
\xy
(0,0)*{
\bt[scale=.35, color=\clr]
	\draw [ thick, directed=1] (0,2.75) to (0,3.5);
	\draw [ thick, directed=0.75] (1,1.75) to [out=90,in=330] (0,2.75);
	\draw [ thick, directed=0.75] (-1,1.75) to [out=90,in=210] (0,2.75);
	\draw [thick] (-1,-1) to (-1,1.75);
	\draw [thick] (1,-1) to (1,1.75);
	\node at (0,3.95) {\scriptsize $2$};
	\node at (-1,-1.45) {\scriptsize $1$};
	\node at (1,-1.45) {\scriptsize $1$};
\et
};
\endxy\stackrel{\eqref{digon-removal}}{=}(2)
\xy
(0,0)*{
\bt[scale=.35, color=\clr]
	\draw [ thick, directed=1] (4,9) to (4,10);
	\draw [ thick, directed=0.75] (5,7.25) to [out=90,in=330] (4,9);
	\draw [ thick, directed=0.75] (3,7.25) to [out=90,in=210] (4,9);
	\node at (4,10.5) {\scriptsize $2$};
	\node at (3,6.75) {\scriptsize $1$};
	\node at (5,6.75) {\scriptsize $1$};
\et
};
\endxy-
\xy
(0,0)*{
\bt[scale=.35, color=\clr]
	\draw [ thick, directed=1] (4,9) to (4,10);
	\draw [ thick, directed=0.75] (5,7.25) to [out=90,in=330] (4,9);
	\draw [ thick, directed=0.75] (3,7.25) to [out=90,in=210] (4,9);
	\node at (4,10.5) {\scriptsize $2$};
	\node at (3,6.75) {\scriptsize $1$};
	\node at (5,6.75) {\scriptsize $1$};
\et
};
\endxy=
\xy
(0,0)*{
\bt[scale=.35, color=\clr]
	\draw [ thick, directed=1] (4,9) to (4,10);
	\draw [ thick, directed=0.75] (5,7.25) to [out=90,in=330] (4,9);
	\draw [ thick, directed=0.75] (3,7.25) to [out=90,in=210] (4,9);
	\node at (4,10.5) {\scriptsize $2$};
	\node at (3,6.75) {\scriptsize $1$};
	\node at (5,6.75) {\scriptsize $1$};
\et
};
\endxy \ .
\eeq
This combined with associativity obtains the relation on the left; for example, if $\sigma=s_2s_1\in\frakS_3$ then we have
\[
\xy
(0,0)*{
\bt[scale=.35, color=\clr]
	\draw [ thick, directed=1] (4,9) to (4,10);
	\draw [ thick, directed=0.45] (6,7.25) to [out=90,in=330] (4,9);
	\draw [ thick, directed=0.75] (2,7.25) to [out=90,in=210] (4,9);
	\draw [ thick, directed=0.75] (4,7.25) to [out=90, in=210] (5,8.5);
	\draw [thick] (4,5) to (2,7.25);
	\draw [thick] (2,5) to (6,7.25);
	\draw [thick] (6,5) to (4,7.25);
	\node at (4,10.5) {\scriptsize $3$};
	\node at (2,4.5) {\scriptsize $1$};
	\node at (4,4.5) {\scriptsize $1$};
	\node at (6,4.5) {\scriptsize $1$};
\et
};
\endxy\stackrel{\eqref{untwist-crossing}}{=}
\xy
(0,0)*{
\bt[scale=.35, color=\clr]
	\draw [ thick, directed=1] (4,9) to (4,10);
	\draw [ thick, directed=0.45] (6,7.25) to [out=90,in=330] (4,9);
	\draw [ thick, directed=0.75] (2,7.25) to [out=90,in=210] (4,9);
	\draw [ thick, directed=0.75] (4,7.25) to [out=90, in=210] (5,8.5);
	\draw [thick] (4,5) to (2,7.25);
	\draw [thick] (2,5) to (4,7.25);
	\draw [thick] (6,5) to (6,7.25);
	\node at (4,10.5) {\scriptsize $3$};
	\node at (2,4.5) {\scriptsize $1$};
	\node at (4,4.5) {\scriptsize $1$};
	\node at (6,4.5) {\scriptsize $1$};
\et
};
\endxy\stackrel{\eqref{associativity}}{=}
\xy
(0,0)*{
\bt[scale=.35, color=\clr]
	\draw [ thick, directed=1] (4,9) to (4,10);
	\draw [ thick, directed=0.75] (6,7.25) to [out=90,in=330] (4,9);
	\draw [ thick, directed=0.45] (2,7.25) to [out=90,in=210] (4,9);
	\draw [ thick, directed=0.75] (4,7.25) to [out=90, in=330] (3,8.5);
	\draw [thick] (4,5) to (2,7.25);
	\draw [thick] (2,5) to (4,7.25);
	\draw [thick] (6,5) to (6,7.25);
	\node at (4,10.5) {\scriptsize $3$};
	\node at (2,4.5) {\scriptsize $1$};
	\node at (4,4.5) {\scriptsize $1$};
	\node at (6,4.5) {\scriptsize $1$};
\et
};
\endxy\stackrel{\eqref{untwist-crossing}}{=}
\xy
(0,0)*{
\bt[scale=.35, color=\clr]
	\draw [ thick, directed=1] (4,9) to (4,10);
	\draw [ thick, directed=0.75] (6,6.75) to (6,7.25) to [out=90,in=330] (4,9);
	\draw [ thick, directed=0.5] (2,6.75) to (2,7.25) to [out=90,in=210] (4,9);
	\draw [ thick, directed=0.75] (4,6.75) to (4,7.25) to [out=90, in=330] (3,8.5);
	\node at (4,10.5) {\scriptsize $3$};
	\node at (2,6.25) {\scriptsize $1$};
	\node at (4,6.25) {\scriptsize $1$};
	\node at (6,6.25) {\scriptsize $1$};
\et
};
\endxy.
\]
The proof of the relation on the right is similar.

For (c), we compute that
\[
\xy
(0,0)*{
\bt[scale=.35, color=\clr]
	\draw [ thick, directed=1] (0, .75) to (0,1.5) to (2,3.5);
	\draw [ thick, directed=.65] (1,-1) to [out=90,in=330] (0,.75);
	\draw [ thick, directed=.65] (-1,-1) to [out=90,in=210] (0,.75);
	\draw [ thick, directed=1] (2,-1) to (2,1.5) to (0,3.5);
	\node at (2, 4) {\scriptsize $k\! +\! l$};
	\node at (-1,-1.5) {\scriptsize $k$};
	\node at (1,-1.5) {\scriptsize $l$};
	\node at (2,-1.5) {\scriptsize $1$};
	\node at (0,4) {\scriptsize $1$};
\et
};
\endxy\stackrel{\eqref{complete-explosion}}{=} \frac{1}{k!\,l!(k+l)!}
\xy
(0,0)*{
\begin{tikzpicture}[scale=.3, color=\clr]
	\draw [ thick, directed=.85] (0,-2.75) to [out=30,in=330] (0,.75);
	\draw [ thick, directed=.85] (0,-2.75) to [out=150,in=210] (0,.75);
	\draw [ thick, directed=.65] (0,-4) to (0,-2.75);
	\node at (0,-4.5) {\scriptsize $k$};
	\node at (-1.5,-1) {\scriptsize $1$};
	\node at (1.5,-1) {\scriptsize $1$};
	\node at (0,-1) { \ $\cdots$};
	\draw [ thick, directed=.85] (4.5,-2.75) to [out=30,in=330] (4.5,.75);
	\draw [ thick, directed=.85] (4.5,-2.75) to [out=150,in=210] (4.5,.75);
	\draw [ thick, directed=.65] (4.5,-4) to (4.5,-2.75);
	\node at (4.5,-4.5) {\scriptsize $l$};
	\node at (3,-1) {\scriptsize $1$};
	\node at (6,-1) {\scriptsize $1$};
	\node at (4.5,-1) { \ $\cdots$};
	\draw [ thick, directed=0.75, looseness=0.75] (0,0.75) to [out=90,in=210] (2.25,3);
	\draw [ thick, directed=0.75, looseness=0.75] (4.5,0.75) to [out=90,in=330] (2.25,3);
	\draw [ thick, directed=0.6] (2.25,3) to (2.25,5);
	\draw [ thick, looseness=0.75] (2.25,5) to [out=150,in=270] (0,7.25);
	\draw [ thick, looseness=0.75] (2.25,5) to [out=30,in=270] (4.5,7.25);
	\node at (2.25,6.5) {\,$\cdot \ \cdot \ \cdot$};
	\node at (-0.5,6.75) {\scriptsize $1$};
	\node at (5,6.75) {\scriptsize $1$};
	\node at (0,2.5) {\scriptsize $k$};
	\node at (4.5,2.5) {\scriptsize $l$};
	\node at (1,4) {\scriptsize $k\!+\!l$ \ };
	\draw [ thick, directed=1] (7,-4) to (7,7.25) to (2.25,11) to (2.25,14.5);
	\node at (2.25,15) {\scriptsize $1$};
	\node at (7,-4.5) {\scriptsize $1$};
	\draw [ thick, ] (0,7.25) to (4.5,11);
	\draw [ thick, ] (4.5,7.25) to (9,11);
	\draw [ thick, directed=0.75, looseness=0.75] (4.5,11) to [out=90,in=210] (6.75,13.25);
	\draw [ thick, directed=0.75, looseness=0.75] (9,11) to [out=90,in=330] (6.75,13.25);
	\node at (6.75,11.5) {$\cdot \ \cdot \ \cdot$};
	\node at (4,11.5) {\scriptsize $1$};
	\node at (9.5,11.5) {\scriptsize $1$};
	\draw [thick, directed=1] (6.75,13.25) to (6.75,14.5);
	\node at (6.75,15) {\scriptsize $k\!+\!l$};
\end{tikzpicture}
};
\endxy
\stackrel{\eqref{clasp-sum}}{=} \frac{1}{k!\,l!(k+l)!}\sum_{\sigma\in\frakS_{k+l}}
\xy
(0,0)*{
\begin{tikzpicture}[scale=.3, color=\clr]
	\draw [ thick, ] (0,-2.75) to [out=30,in=270] (1,-1);
	\draw [ thick, ] (0,-2.75) to [out=150,in=270] (-1,-1);
	\draw [ thick, directed=.65] (0,-4) to (0,-2.75);
	\node at (0,-4.5) {\scriptsize $k$};
	\node at (-1.5,-1.5) {\scriptsize $1$};
	\node at (1.5,-1.5) {\scriptsize $1$};
	\node at (0,-1.5) { \ $\cdots$};
	\draw [ thick, ] (4.5,-2.75) to [out=30,in=270] (5.5,-1);
	\draw [ thick, ] (4.5,-2.75) to [out=150,in=270] (3.5,-1);
	\draw [ thick, directed=.65] (4.5,-4) to (4.5,-2.75);
	\node at (4.5,-4.5) {\scriptsize $l$};
	\node at (3,-1.5) {\scriptsize $1$};
	\node at (6,-1.5) {\scriptsize $1$};
	\node at (4.5,-1.5) { \ $\cdots$};
	\draw [ thick ] (-1.5,-1) rectangle (6,6);
	\node at (2.25,2.5) {$\sigma$};
	\node at (2.25,6.5) {\,$\cdot \ \cdot \ \cdot$};
	\draw [ thick ] (0,6) to (0,7.25);
	\draw [ thick ] (4.5,6) to (4.5,7.25);
	\node at (-0.5,6.75) {\scriptsize $1$};
	\node at (5,6.75) {\scriptsize $1$};
	\draw [ thick, directed=1] (7,-4) to (7,7.25) to (2.25,11) to (2.25,14.5);
	\node at (2.25,15) {\scriptsize $1$};
	\node at (7,-4.5) {\scriptsize $1$};
	\draw [ thick, ] (0,7.25) to (4.5,11);
	\draw [ thick, ] (4.5,7.25) to (9,11);
	\draw [ thick, directed=0.75, looseness=0.75] (4.5,11) to [out=90,in=210] (6.75,13.25);
	\draw [ thick, directed=0.75, looseness=0.75] (9,11) to [out=90,in=330] (6.75,13.25);
	\node at (6.75,11.5) {$\cdot \ \cdot \ \cdot$};
	\node at (4,11.5) {\scriptsize $1$};
	\node at (9.5,11.5) {\scriptsize $1$};
	\draw [thick, directed=1] (6.75,13.25) to (6.75,14.5);
	\node at (6.75,15) {\scriptsize $k\!+\!l$};
\end{tikzpicture}
};
\endxy
\]
\bea
&\stackrel{\eqref{braid-relation}}{=}& \frac{1}{k!\,l!(k+l)!}\sum_{\sigma\in\frakS_{k+l}}
\xy
(0,0)*{
\begin{tikzpicture}[scale=.3, color=\clr]
	\draw [ thick, ] (0,-2.75) to [out=30,in=270] (1,-1);
	\draw [ thick, ] (0,-2.75) to [out=150,in=270] (-1,-1);
	\draw [ thick, directed=.65] (0,-4) to (0,-2.75);
	\node at (0,-4.5) {\scriptsize $k$};
	\node at (-1.5,-1.5) {\scriptsize $1$};
	\node at (1.5,-1.5) {\scriptsize $1$};
	\node at (0,-1.5) { \ $\cdots$};
	\draw [ thick, ] (4.5,-2.75) to [out=30,in=270] (5.5,-1);
	\draw [ thick, ] (4.5,-2.75) to [out=150,in=270] (3.5,-1);
	\draw [ thick, directed=.65] (4.5,-4) to (4.5,-2.75);
	\node at (4.5,-4.5) {\scriptsize $l$};
	\node at (3,-1.5) {\scriptsize $1$};
	\node at (6,-1.5) {\scriptsize $1$};
	\node at (4.5,-1.5) { \ $\cdots$};
	\draw [ thick ] (3,11) rectangle (10.5,4);
	\node at (7,7.5) {$\sigma$};
	\node at (4.75,3.25) {$\cdots$\,};
	\node at (9.25,3.25) {$\cdots$\,};
	\draw [ thick ] (3.5,2.75) to (3.5,4);
	\draw [ thick ] (10,2.75) to (10,4);
	\draw [ thick ] (5.5,2.75) to (5.5,4);
	\draw [ thick ] (8,2.75) to (8,4);
	\node at (3,3.5) {\scriptsize $1$};
	\node at (10.5,3.5) {\scriptsize $1$};
	\node at (6,3.5) {\scriptsize $1$};
	\node at (7.5,3.5) {\scriptsize $1$};
	\draw [ thick, directed=1] (7.25,-4) to (7.25,-1) to (2,2.75) to (2,14.5);
	\node at (2,15) {\scriptsize $1$};
	\node at (7.25,-4.5) {\scriptsize $1$};
	\draw [ thick, ] (-1,-1) to (3.5,2.75);
	\draw [ thick, ] (1,-1) to (5.5,2.75);
	\draw [ thick, ] (3.5,-1) to (8,2.75);
	\draw [ thick, ] (5.5,-1) to (10,2.75);
	\draw [ thick, directed=0.75, looseness=0.75] (4.5,11) to [out=90,in=210] (6.75,13.25);
	\draw [ thick, directed=0.75, looseness=0.75] (9,11) to [out=90,in=330] (6.75,13.25);
	\node at (6.75,11.5) {$\cdot \ \cdot \ \cdot$};
	\node at (4,11.5) {\scriptsize $1$};
	\node at (9.5,11.5) {\scriptsize $1$};
	\draw [thick, directed=1] (6.75,13.25) to (6.75,14.5);
	\node at (6.75,15) {\scriptsize $k\!+\!l$};
\end{tikzpicture}
};
\endxy
\stackrel{\eqref{untangle}}{=} \frac{1}{k!\,l!(k+l)!}\sum_{\sigma\in\frakS_{k+l}}
\xy
(0,0)*{
\begin{tikzpicture}[scale=.3, color=\clr]
	\draw [ thick, ] (0,-2.75) to [out=30,in=270] (1,-1);
	\draw [ thick, ] (0,-2.75) to [out=150,in=270] (-1,-1);
	\draw [ thick, directed=.65] (0,-4) to (0,-2.75);
	\node at (0,-4.5) {\scriptsize $k$};
	\node at (-1.5,-1.5) {\scriptsize $1$};
	\node at (1.5,-1.5) {\scriptsize $1$};
	\node at (0,-1.5) { \ $\cdots$};
	\draw [ thick, ] (4.5,-2.75) to [out=30,in=270] (5.5,-1);
	\draw [ thick, ] (4.5,-2.75) to [out=150,in=270] (3.5,-1);
	\draw [ thick, directed=.65] (4.5,-4) to (4.5,-2.75);
	\node at (4.5,-4.5) {\scriptsize $l$};
	\node at (3,-1.5) {\scriptsize $1$};
	\node at (6,-1.5) {\scriptsize $1$};
	\node at (4.5,-1.5) { \ $\cdots$};
	\node at (4.75,3.25) {$\cdots$\,};
	\node at (9.25,3.25) {$\cdots$\,};
	\draw [ thick ] (3.5,2.75) to (3.5,3.25);
	\draw [ thick ] (10,2.75) to (10,3.25);
	\draw [ thick ] (5.5,2.75) to (5.5,3.25);
	\draw [ thick ] (8,2.75) to (8,3.25);
	\node at (3,3.5) {\scriptsize $1$};
	\node at (10.5,3.5) {\scriptsize $1$};
	\node at (6,3.5) {\scriptsize $1$};
	\node at (7.5,3.5) {\scriptsize $1$};
	\draw [ thick, directed=1] (7.25,-4) to (7.25,-1) to (2,2.75) to (2,7.25);
	\node at (2,7.75) {\scriptsize $1$};
	\node at (7.25,-4.5) {\scriptsize $1$};
	\draw [ thick, directed=0.75, ] (5.5,3.25) to [out=90,in=330] (4.5,4.5);
	\draw [ thick, directed=0.75, ] (8,3.25) to [out=90,in=210] (9,4.5);
	\node at (5,5.75) {\scriptsize $k$};
	\node at (8.5,5.75) {\scriptsize $l$};
	\draw [ thick, ] (-1,-1) to (3.5,2.75);
	\draw [ thick, ] (1,-1) to (5.5,2.75);
	\draw [ thick, ] (3.5,-1) to (8,2.75);
	\draw [ thick, ] (5.5,-1) to (10,2.75);
	\draw [ thick, directed=0.8, directed=0.3, looseness=0.75] (3.5,3.25) to [out=90,in=210] (6.75,5.5);
	\draw [ thick, directed=0.8, directed=0.3, looseness=0.75] (10,3.25) to [out=90,in=330] (6.75,5.5);
	\draw [thick, directed=1] (6.75,5.5) to (6.75,7.25);
	\node at (6.75,7.75) {\scriptsize $k\!+\!l$};
\end{tikzpicture}
};
\endxy\\
&=&
\xy
(0,0)*{
\bt[scale=.35, color=\clr]
	\draw [ thick, directed=1] (0, .75) to (0,1.5);
	\draw [ thick, directed=.65] (1,-1) to [out=90,in=330] (0,.75);
	\draw [ thick, directed=.65] (-1,-1) to [out=90,in=210] (0,.75);
	\draw [ thick, ] (-3,-3) to (-1,-1);
	\draw [ thick, ] (-1,-3) to (1,-1);
	\draw [ thick, directed=1] (1,-3) to (-2,-1) to (-2,1.5);
	\node at (0, 2) {\scriptsize $k\! +\! l$};
	\node at (-1,-3.5) {\scriptsize $k$};
	\node at (-3,-3.5) {\scriptsize $l$};
	\node at (-2,2) {\scriptsize $1$};
	\node at (1,-3.5) {\scriptsize $1$};
\et
};
\endxy \ .
\eea
The proofs of the other relations are similar.
\end{proof}


\section{Main theorems}



\subsection{Equivalence of webs and $\H$-morphisms}


We turn now to the main results of the paper. First, a lemma. For $\a=(\a_1,\dots,\a_m)\in\A(r)$, $m\geq 1$, let $\widehat{\a}\in\Lambda'(r)$ be given by $\widehat{\a}_i=\a_i$ for $1\leq i\leq m$ and $\widehat{\a}_i=0$ for $m+1\leq i\leq r.$ 

\begin{lemma}\label{pi-map}
There exists a locally unital homomorphism $\pi^n\colon\Sdot_c(n,r)\to\W$ given by
\[
1_\lambda\mapsto
\xy
(0,0)*{
\begin{tikzpicture}[color=\clr, scale=.35]
	\draw [thick, directed=1] (-4,-1) to (-4,2.5);
	\draw [thick, directed=1] (-1.5,-1) to (-1.5,2.5);
	\node at (-1.5,3.1) {\fs $\lambda_{n}$};
	\node at (-1.5,-1.5) {\fs $\lambda_{n}$};
	\node at (-4,3.1) {\fs $\lambda_1$};
	\node at (-4,-1.5) {\fs $\lambda_1$};
	\node at (-2.8,0.5) { \ $\cdots$};
\end{tikzpicture}
};
\endxy,\quad\quad e_i1_\lambda\mapsto
\xy
(0,0)*{
\bt[color=\clr, scale=.4]
	\draw [ thick, color=\clr, directed=1] (-6,-1.5) to (-6,2);
	\draw [ thick, color=\clr, directed=1] (-4,-1.5) to (-4,2);
	\draw [ thick, color=\clr, directed=1] (3.5,-1.5) to (3.5,2);
	\draw [ thick, color=\clr, directed=1] (5.5,-1.5) to (5.5,2);
	\draw [ thick, color=\clr, directed=0.3, directed=1] (-2,-1.5) to (-2,2);
	\draw [ thick, color=\clr, directed=1, directed=0.3] (1,-1.5) to (1,2);
	\draw [ thick, color=\clr, directed=0.6] (1,0.25) to (-2,0.25);
	\node at (-6,-2) {\fs $\lambda_1$};
	\node at (-4,-2) {\fs $\lambda_{i-1}$};
	\node at (3.5,-2) {\fs $\lambda_{i+2}$};
	\node at (5.5,-2) {\fs $\lambda_n$};
	\node at (-6,2.5) {\fs $\lambda_1$};
	\node at (-4,2.5) {\fs $\lambda_{i-1}$};
	\node at (3.5,2.5) {\fs $\lambda_{i+2}$};
	\node at (5.5,2.5) {\fs $\lambda_n$};
	\node at (-2,-2) {\fs $\lambda_i$};
	\node at (1,-2) {\fs $\lambda_{i+1}$};
	\node at (-2,2.5) {\fs $\lambda_i\! +\! 1$};
	\node at (1,2.5) {\fs $\lambda_{i+1}\! -\! 1$};
	\node at (-0.5,1) {\ss $1$ \ };
	\node at (-5,0.5) {\,$\cdots$};
	\node at (4.5,0.5) {\,$\cdots$};
\et
};
\endxy,
\]
\[
f_i1_\lambda\mapsto
\xy
(0,0)*{
\bt[color=\clr, scale=.4]
	\draw [ thick, color=\clr, directed=1] (-6,-1.5) to (-6,2);
	\draw [ thick, color=\clr, directed=1] (-4,-1.5) to (-4,2);
	\draw [ thick, color=\clr, directed=1] (3.5,-1.5) to (3.5,2);
	\draw [ thick, color=\clr, directed=1] (5.5,-1.5) to (5.5,2);
	\draw [ thick, color=\clr, directed=0.3, directed=1] (-2,-1.5) to (-2,2);
	\draw [ thick, color=\clr, directed=1, directed=0.3] (1,-1.5) to (1,2);
	\draw [ thick, color=\clr, directed=0.55] (-2,0.25) to (1,0.25);
	\node at (-6,-2) {\fs $\lambda_1$};
	\node at (-4,-2) {\fs $\lambda_{i-1}$};
	\node at (3.5,-2) {\fs $\lambda_{i+2}$};
	\node at (5.5,-2) {\fs $\lambda_n$};
	\node at (-6,2.5) {\fs $\lambda_1$};
	\node at (-4,2.5) {\fs $\lambda_{i-1}$};
	\node at (3.5,2.5) {\fs $\lambda_{i+2}$};
	\node at (5.5,2.5) {\fs $\lambda_n$};
	\node at (-2,-2) {\fs $\lambda_i$};
	\node at (1,-2) {\fs $\lambda_{i+1}$};
	\node at (-2,2.5) {\fs $\lambda_i\! -\! 1$};
	\node at (1,2.5) {\fs $\lambda_{i+1}\! +\! 1$};
	\node at (-0.5,1) {\ss $1$ \ };
	\node at (-5,0.5) {\,$\cdots$};
	\node at (4.5,0.5) {\,$\cdots$};
\et
};
\endxy,\quad\quad e_{\ol{i}}1_\lambda\mapsto
\xy
(0,0)*{
\bt[color=\clr, scale=.4]
	\draw [ thick, color=\clr, directed=1] (-6,-1.5) to (-6,2);
	\draw [ thick, color=\clr, directed=1] (-4,-1.5) to (-4,2);
	\draw [ thick, color=\clr, directed=1] (3.5,-1.5) to (3.5,2);
	\draw [ thick, color=\clr, directed=1] (5.5,-1.5) to (5.5,2);
	\draw [ thick, color=\clr, directed=0.3, directed=1] (-2,-1.5) to (-2,2);
	\draw [ thick, color=\clr, directed=1, directed=0.3] (1,-1.5) to (1,2);
	\draw [ thick, color=\clr, directed=0.75] (1,0.25) to (-2,0.25);
	\node at (-6,-2) {\fs $\lambda_1$};
	\node at (-4,-2) {\fs $\lambda_{i-1}$};
	\node at (3.5,-2) {\fs $\lambda_{i+2}$};
	\node at (5.5,-2) {\fs $\lambda_n$};
	\node at (-6,2.5) {\fs $\lambda_1$};
	\node at (-4,2.5) {\fs $\lambda_{i-1}$};
	\node at (3.5,2.5) {\fs $\lambda_{i+2}$};
	\node at (5.5,2.5) {\fs $\lambda_n$};
	\node at (-2,-2) {\fs $\lambda_i$};
	\node at (1,-2) {\fs $\lambda_{i+1}$};
	\node at (-2,2.5) {\fs $\lambda_i\! +\! 1$};
	\node at (1,2.5) {\fs $\lambda_{i+1}\! -\! 1$};
	\node at (-0.5,1) {\ss $1$ \ };
	\node at (-5,0.5) {\,$\cdots$};
	\node at (4.5,0.5) {\,$\cdots$};
	\draw (0,0.25) \wdot;
\et
};
\endxy,
\]
\[
f_{\ol{i}}1_\lambda\mapsto
\xy
(0,0)*{
\bt[color=\clr, scale=.4]
	\draw [ thick, color=\clr, directed=1] (-6,-1.5) to (-6,2);
	\draw [ thick, color=\clr, directed=1] (-4,-1.5) to (-4,2);
	\draw [ thick, color=\clr, directed=1] (3.5,-1.5) to (3.5,2);
	\draw [ thick, color=\clr, directed=1] (5.5,-1.5) to (5.5,2);
	\draw [ thick, color=\clr, directed=0.3, directed=1] (-2,-1.5) to (-2,2);
	\draw [ thick, color=\clr, directed=1, directed=0.3] (1,-1.5) to (1,2);
	\draw [ thick, color=\clr, directed=0.75] (-2,0.25) to (1,0.25);
	\node at (-6,-2) {\fs $\lambda_1$};
	\node at (-4,-2) {\fs $\lambda_{i-1}$};
	\node at (3.5,-2) {\fs $\lambda_{i+2}$};
	\node at (5.5,-2) {\fs $\lambda_n$};
	\node at (-6,2.5) {\fs $\lambda_1$};
	\node at (-4,2.5) {\fs $\lambda_{i-1}$};
	\node at (3.5,2.5) {\fs $\lambda_{i+2}$};
	\node at (5.5,2.5) {\fs $\lambda_n$};
	\node at (-2,-2) {\fs $\lambda_i$};
	\node at (1,-2) {\fs $\lambda_{i+1}$};
	\node at (-2,2.5) {\fs $\lambda_i\! -\! 1$};
	\node at (1,2.5) {\fs $\lambda_{i+1}\! +\! 1$};
	\node at (-0.5,1) {\ss $1$ \ };
	\node at (-5,0.5) {\,$\cdots$};
	\node at (4.5,0.5) {\,$\cdots$};
	\draw (-1,0.25) \wdot;
\et
};
\endxy,\quad\quad h_{\ol{j}}1_\lambda\mapsto
\xy
(0,0)*{
\begin{tikzpicture}[color=\clr, scale=.35]
	\draw [thick, directed=1] (-4,-1) to (-4,2.5);
	\draw [thick, directed=1] (-1.5,-1) to (-1.5,2.5);
	\draw [thick, directed=1] (0.5,-1) to (0.5,2.5);
	\draw [thick, directed=1] (2.5,-1) to (2.5,2.5);
	\draw [thick, directed=1] (5,-1) to (5,2.5);
	\node at (0.5,3.1) {\fs $\lambda_j$};
	\node at (0.5,-1.5) {\fs $\lambda_j$};
	\node at (-1.5,3.1) {\fs $\lambda_{j-1}$};
	\node at (-1.5,-1.5) {\fs $\lambda_{j-1}$};
	\node at (-4,3.1) {\fs $\lambda_1$};
	\node at (-4,-1.5) {\fs $\lambda_1$};
	\node at (2.5,3.1) {\fs $\lambda_{j+1}$};
	\node at (2.5,-1.5) {\fs $\lambda_{j+1}$};
	\node at (5,3.1) {\fs $\lambda_n$};
	\node at (5,-1.5) {\fs $\lambda_n$};
	\node at (-2.8,0.5) { \ $\cdots$};
	\node at (3.7,0.5) { \ $\cdots$};
	\draw (0.5,0.5) \wdot;
\end{tikzpicture}
};
\endxy
\]
for $\lambda\in\Lambda(n,r)$, $1\leq i\leq n-1$, and $1\leq j\leq n$. Moreover $\pi^r\colon\Sdot_c(r)\to\W$ is surjective, and remains surjective when restricting to $\Sdot'_c(r).$
\end{lemma}

\begin{proof}
First we note that the six pictures above make sense as webs in $\W$ because of our chosen conventions (see Remark \ref{conventions}). Moreover these assignments define $\pi^n$ on all of $\Sdot_c(n,r)$ by \eqref{S-decomp}. 

That $\pi^n$ is well-defined amounts to showing that the images of \eqref{S1}-\eqref{S6} hold in $\W$. This is easily proved by direct calculations involving the defining relations of $\W$, Lemma \ref{additional-relations}, and Lemma \ref{Udot-relations}, so we leave it to the reader. 

To show $\pi^r$ is surjective when restricting to $\Sdot'_c(r)$, we prove that each generator of $\W$ is the image of an element in $\Sdot'_c(r)$. Indeed, one can check that
\[
\pi^r(1_{\widehat{\a}})=\a,\quad\pi^r\left(e_{r-1}^{(\,\widehat{\a}_r)}e_{r-2}^{(\,\widehat{\a}_{r-1})}\cdots e_i^{(\,\widehat{\a}_{i+1})}1_{\widehat{\a}}\right)=
\xy
(0,0)*{
\begin{tikzpicture}[color=\clr, scale=.35]
	\draw [thick, directed=1] (-3,-1) to (-3,2.5);
	\draw [thick, directed=1] (3,-1) to (3,2.5);
	\draw [thick, directed=1] (-5.5,-1) to (-5.5,2.5);
	\draw [thick, directed=1] (5.5,-1) to (5.5,2.5);
	\draw [ thick, directed=1] (0,1) to (0,2.5);
	\draw [ thick, directed=.65] (1,-1) to [out=90,in=330] (0,1);
	\draw [ thick, directed=.65] (-1,-1) to [out=90,in=210] (0,1);
	\node at (0,3) {\fs $\a_i\!+\!\a_{i+1}$};
	\node at (-1,-1.5) {\fs $\a_i$};
	\node at (1,-1.5) {\fs $\a_{i+1}$};
	\node at (-3,-1.5) {\fs $\a_{i-1}$};
	\node at (-3,3) {\fs $\a_{i-1}$};
	\node at (3,-1.5) {\fs $\a_{i+2}$};
	\node at (3,3) {\fs $\a_{i+2}$};
	\node at (-5.5,-1.5) {\fs $\a_1$};
	\node at (-5.5,3) {\fs $\a_1$};
	\node at (5.5,-1.5) {\fs $\a_m$};
	\node at (5.5,3) {\fs $\a_m$};
	\node at (-4.3,0.5) { \ $\cdots$};
	\node at (4.2,0.5) { \ $\cdots$};
\end{tikzpicture}
};
\endxy,
\]
\[
\pi^r\left(f_i^{(\,\widehat{\a}_{i+1})}f_{i+1}^{(\,\widehat{\a}_{i+2})}\cdots f_{r-1}^{(\,\widehat{\a}_r)}1_{\,\widehat{\a}}\right)=
\xy
(0,0)*{
\begin{tikzpicture}[color=\clr, scale=.35]
	\draw [thick, directed=1] (-3,-1) to (-3,2.5);
	\draw [thick, directed=1] (3,-1) to (3,2.5);
	\draw [thick, directed=1] (-5.5,-1) to (-5.5,2.5);
	\draw [thick, directed=1] (5.5,-1) to (5.5,2.5);
	\draw [ thick, directed=0.65] (0,-1) to (0,0.5);
	\draw [ thick, directed=1] (0,0.5) to [out=30,in=270] (1,2.5);
	\draw [ thick, directed=1] (0,0.5) to [out=150,in=270] (-1,2.5); 
	\node at (0,-1.5) {\fs $\a_i\!+\!\a_{i+1}$};
	\node at (-1,3) {\fs $\a_i$};
	\node at (1,3) {\fs $\a_{i+1}$};
	\node at (-3,-1.5) {\fs $\a_{i-1}$};
	\node at (-3,3) {\fs $\a_{i-1}$};
	\node at (3,-1.5) {\fs $\a_{i+2}$};
	\node at (3,3) {\fs $\a_{i+2}$};
	\node at (-5.5,-1.5) {\fs $\a_1$};
	\node at (-5.5,3) {\fs $\a_1$};
	\node at (5.5,-1.5) {\fs $\a_m$};
	\node at (5.5,3) {\fs $\a_m$};
	\node at (-4.3,0.5) { \ $\cdots$};
	\node at (4.2,0.5) { \ $\cdots$};
\end{tikzpicture}
};
\endxy,\quad\pi^r(h_{\ol{j}}1_{\,\widehat{\a}})=
\xy
(0,0)*{
\begin{tikzpicture}[color=\clr, scale=.35]
	\draw [thick, directed=1] (-4,-1) to (-4,2.5);
	\draw [thick, directed=1] (-1.5,-1) to (-1.5,2.5);
	\draw [thick, directed=1] (0.5,-1) to (0.5,2.5);
	\draw [thick, directed=1] (2.5,-1) to (2.5,2.5);
	\draw [thick, directed=1] (5,-1) to (5,2.5);
	\node at (0.5,3) {\fs $\a_j$};
	\node at (0.5,-1.5) {\fs $\a_j$};
	\node at (-1.5,3) {\fs $\a_{j-1}$};
	\node at (-1.5,-1.5) {\fs $\a_{j-1}$};
	\node at (-4,3) {\fs $\a_1$};
	\node at (-4,-1.5) {\fs $\a_1$};
	\node at (2.5,3) {\fs $\a_{j+1}$};
	\node at (2.5,-1.5) {\fs $\a_{j+1}$};
	\node at (5,3) {\fs $\a_m$};
	\node at (5,-1.5) {\fs $\a_m$};
	\node at (-2.8,0.5) { \ $\cdots$};
	\node at (3.7,0.5) { \ $\cdots$};
	\draw (0.5,0.5) \wdot;
\end{tikzpicture}
};
\endxy
\]
for $1\leq i\leq m-1$ and $1\leq j\leq m$, proving surjectivity and the lemma.
\end{proof}

For $\lambda\in\Lambda'(r)$ let $\a_\lambda:=\phi^r(1_\lambda)\in W.$

\begin{theorem}\label{psi-map}
There exists a locally unital isomorphism $\psi\colon\W\xrightarrow{\sim}\End_{\H}(\bigoplus_{\lambda\in\Lambda'(r)}M^\lambda)$ given by
\[\bg
\a\mapsto\phi^r(1_{\widehat{\a}}),\quad\quad
\xy
(0,0)*{
\begin{tikzpicture}[color=\clr, scale=.35]
	\draw [thick, directed=1] (-3,-1) to (-3,2.5);
	\draw [thick, directed=1] (3,-1) to (3,2.5);
	\draw [thick, directed=1] (-5.5,-1) to (-5.5,2.5);
	\draw [thick, directed=1] (5.5,-1) to (5.5,2.5);
	\draw [ thick, directed=1] (0,1) to (0,2.5);
	\draw [ thick, directed=.65] (1,-1) to [out=90,in=330] (0,1);
	\draw [ thick, directed=.65] (-1,-1) to [out=90,in=210] (0,1);
	\node at (0,3) {\fs $\a_i\!+\!\a_{i+1}$};
	\node at (-1,-1.5) {\fs $\a_i$};
	\node at (1,-1.5) {\fs $\a_{i+1}$};
	\node at (-3,-1.5) {\fs $\a_{i-1}$};
	\node at (-3,3) {\fs $\a_{i-1}$};
	\node at (3,-1.5) {\fs $\a_{i+2}$};
	\node at (3,3) {\fs $\a_{i+2}$};
	\node at (-5.5,-1.5) {\fs $\a_1$};
	\node at (-5.5,3) {\fs $\a_1$};
	\node at (5.5,-1.5) {\fs $\a_m$};
	\node at (5.5,3) {\fs $\a_m$};
	\node at (-4.3,0.5) { \ $\cdots$};
	\node at (4.2,0.5) { \ $\cdots$};
\end{tikzpicture}
};
\endxy\mapsto\phi^r\left(e_{r-1}^{(\,\widehat{\a}_r)}e_{r-2}^{(\,\widehat{\a}_{r-1})}\cdots e_i^{(\,\widehat{\a}_{i+1})}1_{\widehat{\a}}\right),\\
\xy
(0,0)*{
\begin{tikzpicture}[color=\clr, scale=.35]
	\draw [thick, directed=1] (-3,-1) to (-3,2.5);
	\draw [thick, directed=1] (3,-1) to (3,2.5);
	\draw [thick, directed=1] (-5.5,-1) to (-5.5,2.5);
	\draw [thick, directed=1] (5.5,-1) to (5.5,2.5);
	\draw [ thick, directed=0.65] (0,-1) to (0,0.5);
	\draw [ thick, directed=1] (0,0.5) to [out=30,in=270] (1,2.5);
	\draw [ thick, directed=1] (0,0.5) to [out=150,in=270] (-1,2.5); 
	\node at (0,-1.5) {\fs $\a_i\!+\!\a_{i+1}$};
	\node at (-1,3) {\fs $\a_i$};
	\node at (1,3) {\fs $\a_{i+1}$};
	\node at (-3,-1.5) {\fs $\a_{i-1}$};
	\node at (-3,3) {\fs $\a_{i-1}$};
	\node at (3,-1.5) {\fs $\a_{i+2}$};
	\node at (3,3) {\fs $\a_{i+2}$};
	\node at (-5.5,-1.5) {\fs $\a_1$};
	\node at (-5.5,3) {\fs $\a_1$};
	\node at (5.5,-1.5) {\fs $\a_m$};
	\node at (5.5,3) {\fs $\a_m$};
	\node at (-4.3,0.5) { \ $\cdots$};
	\node at (4.2,0.5) { \ $\cdots$};
\end{tikzpicture}
};
\endxy\mapsto\phi^r\left(f_i^{(\,\widehat{\a}_{i+1})}f_{i+1}^{(\,\widehat{\a}_{i+2})}\cdots f_{r-1}^{(\,\widehat{\a}_r)}1_{\,\widehat{\a}}\right),\\
\xy
(0,0)*{
\begin{tikzpicture}[color=\clr, scale=.35]
	\draw [thick, directed=1] (-4,-1) to (-4,2.5);
	\draw [thick, directed=1] (-1.5,-1) to (-1.5,2.5);
	\draw [thick, directed=1] (0.5,-1) to (0.5,2.5);
	\draw [thick, directed=1] (2.5,-1) to (2.5,2.5);
	\draw [thick, directed=1] (5,-1) to (5,2.5);
	\node at (0.5,3) {\fs $\a_j$};
	\node at (0.5,-1.5) {\fs $\a_j$};
	\node at (-1.5,3) {\fs $\a_{j-1}$};
	\node at (-1.5,-1.5) {\fs $\a_{j-1}$};
	\node at (-4,3) {\fs $\a_1$};
	\node at (-4,-1.5) {\fs $\a_1$};
	\node at (2.5,3) {\fs $\a_{j+1}$};
	\node at (2.5,-1.5) {\fs $\a_{j+1}$};
	\node at (5,3) {\fs $\a_m$};
	\node at (5,-1.5) {\fs $\a_m$};
	\node at (-2.8,0.5) { \ $\cdots$};
	\node at (3.7,0.5) { \ $\cdots$};
	\draw (0.5,0.5) \wdot;
\end{tikzpicture}
};
\endxy\mapsto\phi^r(h_{\ol{j}}1_{\,\widehat{\a}}),
\eg\]
for $\a\in\A(r)$, $m\geq 1$, $1\leq i\leq m-1$, and $1\leq j\leq m.$ In particular, we have superspace isomorphisms
\[
\psi_{\a,\b}\colon\b\W\a\xrightarrow{\sim}\Hom_{\H}(M^{\,\widehat{\a}},M^{\,\widehat{\b}})
\]
for $\a,\b\in\A(r).$
\end{theorem}

\begin{proof}
From the definitions, it is clear that we have a commutative diagram
\beq\label{triangle}
\xy
(0,0)*{
\bt
\node (S) at (0,0) {$\Sdot'_c(r)$};
\node (W) at (0,-2) {$\W$};
\node (E) at (4,0) {$\End_{\H}(\bigoplus_{\lambda\in\Lambda'(r)}M^\lambda)$};
\draw[->, ] (S) to [above] node {$\phi^r$} (E); 
\draw[->, ] (S) to [left] node {$\pi^r$} (W);
\draw[->, ] (W) to [below] node {$\psi$} (E);
\et
};
\endxy.
\eeq
%
 Next, one can check that the defining relations of $\W$ are the images under $\pi^r$ of certain cases of relations \eqref{S1}-\eqref{S6} of $\Sdot_c(r)$. By commutativity of \eqref{triangle}, this implies the images under $\psi$ of the defining relations of $\W$ hold in the endomorphism superalgebra, so $\psi$ is well-defined.

Now it is clear that $\psi$ is locally unital, and for $\lambda,\mu\in\Lambda'(r)$ the component $\Hom_{\H}(M^\lambda,M^\mu)$ is the target of the component map $\psi_{\a_\lambda,\a_\mu}$ and no other component maps. This implies $\psi$ is an isomorphism if and only if its component linear maps $\psi_{\b,\c}$ are bijective for all $\b,\c\in\A(r)$, and we claim this in turn is true if just $\psi_{\a_\omega,\a_\omega}$ alone is bijective. Indeed, let $s_\c\in\a_\omega\W\c$ and $m_\b\in\b\W\a_\omega$ be the webs such that $\beta_\b^\c(u)=s_\c\,u\,m_\b$ for all $u\in\c\W\b$ (i.e. $s_\c$ consists of the appropriate splits and $m_\b$ of the appropriate merges). Then since $\psi$ is a homomorphism, we have a commutative diagram
\[
\xy
(0,0)*{
\bt
\node (S) at (0,0) {$\c\W\b$};
\node (W) at (0,-1.5) {$\a_\omega\W\a_\omega$};
\node (E) at (3,0) {$\Hom_{\H}(M^{\,\widehat{\b}},M^{\,\widehat{\c}}\,)$};
\node (H) at (3,-1.5) {$\End_{\H}(M^{\omega})$};
\draw[->, ] (S) to [above] node {$\psi_{\b,\c}$} (E); 
\draw[->, ] (S) to [left] node {$\beta_\b^\c$} (W);
\draw[->, ] (W) to [below] node {$\psi_{\a_\omega,\a_\omega}$} (H);
\draw[->] (E) to [right] node {$\zeta_\b^\c$} (H);
\et
};
\endxy
\]
where $\zeta_\b^\c$ is the linear map given by $\zeta_\b^\c(y)=\psi(s_\c)\,y\,\psi(m_\b)$ for $y\in\Hom_\H(M^{\,\widehat{\b}},M^{\,\widehat{\c}}\,).$ It is injective because $\beta_\b^\c$ is by Lemma \ref{beta-theta-maps}. Using this, one sees by diagram-chasing that $\psi_{\b,\c}$ is injective (resp. surjective) if $\psi_{\a_\omega,\a_\omega}$ is injective (resp. surjective), proving the claim.

To show $\psi_{\a_\omega,\a_\omega}$ is bijective, we consider the homomorphism $\psi_{\a_\omega,\a_\omega}\circ\xi\colon\H\to\End_{\H}(M^{\omega})$. From the $\H$-isomorphism $\H\simeq M^{\omega}$ (see \S\ref{perm-mods}), we know elements of $\End_{\H}(M^{\omega})$ are completely determined by where they send $T^0$, the $\omega$-supertabloid with entries $1,2,\dots,r$ from top to bottom. One can check that
\[
(\psi_{\a_\omega,\a_\omega}\circ\xi)(s_i)(T^0)=T^0_{i\leftrightarrow i+1},\quad(\psi_{\a_\omega,\a_\omega}\circ\xi)(c_j)(T^0)=T^0_{j\to j'}
\]
for $1\leq i\leq r-1$ and $1\leq j\leq r$, from which it follows that $\psi_{\a_\omega,\a_\omega}\circ\xi$ is both injective and surjective. This implies $\xi$ is injective and $\psi_{\a_\omega,\a_\omega}$ is surjective. Since $\xi$ is surjective by Proposition \ref{xi-map} it is bijective, hence $\psi_{\a_\omega,\a_\omega}$ is bijective, completing the proof.
\end{proof}

The last paragraph of the above proof also established the following.

\begin{corollary}\label{xi-iso}
The homomorphism $\xi\colon\H\to\a_\omega\W\a_\omega$ is an isomorphism.
\end{corollary}

We now explicitly describe the $\H$-morphisms which are the images under $\psi$ of merges, splits, and dots. For $\lambda\in\Lambda'(r)$, $T\in\T(\lambda)$, and $1\leq i\leq l(\lambda)-1,$ we define $\merge_i(T)$ to be the supertabloid obtained by merging rows $i$ and $i+1$ of $T$. For $1\leq i\leq l(\lambda)$ and $k,l\in\Z_{>0}$ with $k+l=\lambda_i$, let $\split_i^{k,\,l}(T)$ be the sum of all supertabloids which can be obtained by splitting row $i$ of $T$ into (on top) a row of length $k$ and (on bottom) a row of length $l$. Recall also from Section \ref{supertabloids-section} the notation $T_{i\leftrightarrow i+1}$ and $T_{j\to j'}$ for $1\leq i\leq r-1$, $1\leq j\leq r$. Then one can check that
\beq\label{psi-dot}
\psi\left(
\xy
(0,0)*{
\begin{tikzpicture}[color=\clr, scale=.35]
	\draw [thick, directed=1] (-4,-1) to (-4,2.5);
	\draw [thick, directed=1] (-1.5,-1) to (-1.5,2.5);
	\draw [thick, directed=1] (0.5,-1) to (0.5,2.5);
	\draw [thick, directed=1] (2.5,-1) to (2.5,2.5);
	\draw [thick, directed=1] (5,-1) to (5,2.5);
	\node at (0.5,3) {\fs $\lambda_j$};
	\node at (0.5,-1.5) {\fs $\lambda_j$};
	\node at (-1.5,3) {\fs $\lambda_{j-1}$};
	\node at (-1.5,-1.5) {\fs $\lambda_{j-1}$};
	\node at (-4,3) {\fs $\lambda_1$};
	\node at (-4,-1.5) {\fs $\lambda_1$};
	\node at (2.5,3) {\fs $\lambda_{j+1}$};
	\node at (2.5,-1.5) {\fs $\lambda_{j+1}$};
	\node at (5,3) {\fs $\lambda_m$};
	\node at (5,-1.5) {\fs $\lambda_m$};
	\node at (-2.8,0.5) { \ $\cdots$};
	\node at (3.7,0.5) { \ $\cdots$};
	\draw (0.5,0.5) \wdot;
\end{tikzpicture}
};
\endxy\right)(T) \ = \ \sum_{\substack{1\leq h\leq r\\ T_h\text{ lies in row }j}}(-1)^{|T_1|+\cdots+|T_{h-1}|}T_{h\to h'},
\eeq
\beq\label{psi-merge}
\psi\left(
\xy
(0,0)*{
\begin{tikzpicture}[color=\clr, scale=.35]
	\draw [thick, directed=1] (-3,-1) to (-3,2.5);
	\draw [thick, directed=1] (3,-1) to (3,2.5);
	\draw [thick, directed=1] (-5.5,-1) to (-5.5,2.5);
	\draw [thick, directed=1] (5.5,-1) to (5.5,2.5);
	\draw [ thick, directed=1] (0,1) to (0,2.5);
	\draw [ thick, directed=.65] (1,-1) to [out=90,in=330] (0,1);
	\draw [ thick, directed=.65] (-1,-1) to [out=90,in=210] (0,1);
	\node at (0,3) {\fs $\lambda_i\!+\!\lambda_{i+1}$};
	\node at (-1,-1.5) {\fs $\lambda_i$};
	\node at (1,-1.5) {\fs $\lambda_{i+1}$};
	\node at (-3,-1.5) {\fs $\lambda_{i-1}$};
	\node at (-3,3) {\fs $\lambda_{i-1}$};
	\node at (3,-1.5) {\fs $\lambda_{i+2}$};
	\node at (3,3) {\fs $\lambda_{i+2}$};
	\node at (-5.5,-1.5) {\fs $\lambda_1$};
	\node at (-5.5,3) {\fs $\lambda_1$};
	\node at (5.5,-1.5) {\fs $\lambda_m$};
	\node at (5.5,3) {\fs $\lambda_m$};
	\node at (-4.3,0.5) { \ $\cdots$};
	\node at (4.2,0.5) { \ $\cdots$};
\end{tikzpicture}
};
\endxy\right)(T) \ = \ \merge_i(T),
\eeq
\beq\label{psi-split}
\psi\left(
\xy
(0,0)*{
\begin{tikzpicture}[color=\clr, scale=.35]
	\draw [thick, directed=1] (-3,-1) to (-3,2.5);
	\draw [thick, directed=1] (3,-1) to (3,2.5);
	\draw [thick, directed=1] (-5.5,-1) to (-5.5,2.5);
	\draw [thick, directed=1] (5.5,-1) to (5.5,2.5);
	\draw [ thick, directed=0.65] (0,-1) to (0,0.5);
	\draw [ thick, directed=1] (0,0.5) to [out=30,in=270] (1,2.5);
	\draw [ thick, directed=1] (0,0.5) to [out=150,in=270] (-1,2.5); 
	\node at (0,-1.5) {\fs $\lambda_i$};
	\node at (-1,3) {\fs $k$};
	\node at (1,3) {\fs $l$};
	\node at (-3,-1.5) {\fs $\lambda_{i-1}$};
	\node at (-3,3) {\fs $\lambda_{i-1}$};
	\node at (3,-1.5) {\fs $\lambda_{i+1}$};
	\node at (3,3) {\fs $\lambda_{i+1}$};
	\node at (-5.5,-1.5) {\fs $\lambda_1$};
	\node at (-5.5,3) {\fs $\lambda_1$};
	\node at (5.5,-1.5) {\fs $\lambda_m$};
	\node at (5.5,3) {\fs $\lambda_m$};
	\node at (-4.3,0.5) { \ $\cdots$};
	\node at (4.2,0.5) { \ $\cdots$};
\end{tikzpicture}
};
\endxy\right)(T) \ = \ \split_i^{k,\,l}(T)
\eeq
where $m:=l(\lambda).$ For example, 
\[
\psi\left(
\xy
(0,0)*{
\begin{tikzpicture}[scale=.3, color=\clr]
	\draw [thick, directed=1] (-1,-1) to (-1,2.5);
	\draw [thick, directed=1] (0.5,-1) to (0.5,2.5);
	\draw [thick, directed=1] (2,-1) to (2,2.5);
	\node at (0.5,3) {\scriptsize $1$};
	\node at (0.5,-1.5) {\scriptsize $1$};
	\node at (-1,3) {\scriptsize $2$};
	\node at (-1,-1.5) {\scriptsize $2$};
	\node at (2,3) {\scriptsize $3$};
	\node at (2,-1.5) {\scriptsize $3$};
	\draw (2,0.5) \wdot;
\end{tikzpicture}
};
\endxy\right)
\left(\,
\xy
(0,0)*{
\bt[scale=0.5]
	\draw (0,0) to (2,0);
	\draw (0,-1) to (2,-1);
	\draw (0,-2) to (3,-2);
	\draw (0,-3) to (3,-3);
	\node at (0.5,-0.5) { $3'$};
	\node at (1.5,-0.5) { $6'$};
	\node at (0.5,-1.5) { $2$};
	\node at (0.5,-2.5) { $1$};
	\node at (1.5,-2.5) { $4'$};
	\node at (2.5,-2.5) { $5$};
\et
};
\endxy\,\right) \ = \ 
\xy
(0,0)*{
\bt[scale=0.5]
	\draw (0,0) to (2,0);
	\draw (0,-1) to (2,-1);
	\draw (0,-2) to (3,-2);
	\draw (0,-3) to (3,-3);
	\node at (0.5,-0.5) { $3'$};
	\node at (1.5,-0.5) { $6'$};
	\node at (0.5,-1.5) { $2$};
	\node at (0.5,-2.5) { $1'$};
	\node at (1.5,-2.5) { $4'$};
	\node at (2.5,-2.5) { $5$};
\et
};
\endxy \ -\,
\xy
(0,0)*{
\bt[scale=0.5]
	\draw (0,0) to (2,0);
	\draw (0,-1) to (2,-1);
	\draw (0,-2) to (3,-2);
	\draw (0,-3) to (3,-3);
	\node at (0.5,-0.5) { $3'$};
	\node at (1.5,-0.5) { $6'$};
	\node at (0.5,-1.5) { $2$};
	\node at (0.5,-2.5) { $1$};
	\node at (1.5,-2.5) { $4$};
	\node at (2.5,-2.5) { $5$};
\et
};
\endxy \ + \ 
\xy
(0,0)*{
\bt[scale=0.5]
	\draw (0,0) to (2,0);
	\draw (0,-1) to (2,-1);
	\draw (0,-2) to (3,-2);
	\draw (0,-3) to (3,-3);
	\node at (0.5,-0.5) { $3'$};
	\node at (1.5,-0.5) { $6'$};
	\node at (0.5,-1.5) { $2$};
	\node at (0.5,-2.5) { $1$};
	\node at (1.5,-2.5) { $4'$};
	\node at (2.5,-2.5) { $5'$};
\et
};
\endxy \ ,
\]
\[
\psi\left(
\xy
(0,0)*{
\begin{tikzpicture}[scale=.3, color=\clr]
	\draw [thick, directed=1] (2.5,-1) to (2.5,2.5);
	\draw [ thick, directed=1] (0,1) to (0,2.5);
	\draw [ thick, directed=.65] (1,-1) to [out=90,in=330] (0,1);
	\draw [ thick, directed=.65] (-1,-1) to [out=90,in=210] (0,1);
	\node at (0,3) {\scriptsize $3$};
	\node at (-1,-1.5) {\scriptsize $2$};
	\node at (1,-1.5) {\scriptsize $1$};
	\node at (2.5,-1.5) {\scriptsize $3$};
	\node at (2.5,3) {\scriptsize $3$};
\end{tikzpicture}
};
\endxy\right)
\left(\,
\xy
(0,0)*{
\bt[scale=0.5]
	\draw (0,0) to (2,0);
	\draw (0,-1) to (2,-1);
	\draw (0,-2) to (3,-2);
	\draw (0,-3) to (3,-3);
	\node at (0.5,-0.5) { $3'$};
	\node at (1.5,-0.5) { $6'$};
	\node at (0.5,-1.5) { $2$};
	\node at (0.5,-2.5) { $1$};
	\node at (1.5,-2.5) { $4'$};
	\node at (2.5,-2.5) { $5$};
\et
};
\endxy\,\right) \ = \ 
\xy
(0,0)*{
\bt[scale=0.5]
	\draw (0,0) to (3,0);
	\draw (0,-1) to (3,-1);
	\draw (0,-2) to (3,-2);
	\node at (0.5,-0.5) { $2$};
	\node at (1.5,-0.5) { $3'$};
	\node at (2.5,-0.5) { $6'$};
	\node at (0.5,-1.5) { $1$};
	\node at (1.5,-1.5) { $4'$};
	\node at (2.5,-1.5) { $5$};
\et
};
\endxy \ ,
\]
\[
\psi\left(
\xy
(0,0)*{
\begin{tikzpicture}[scale=.3, color=\clr]
	\draw [thick, directed=1] (-3,-1) to (-3,2.5);
	\draw [thick, directed=1] (-5,-1) to (-5,2.5);
	\draw [ thick, directed=0.65] (0,-1) to (0,0.5);
	\draw [ thick, directed=1] (0,0.5) to [out=30,in=270] (1,2.5);
	\draw [ thick, directed=1] (0,0.5) to [out=150,in=270] (-1,2.5); 
	\node at (0,-1.5) {\scriptsize $3$};
	\node at (-1,3) {\scriptsize $2$};
	\node at (1,3) {\scriptsize $1$};
	\node at (-3,-1.5) {\scriptsize $1$};
	\node at (-3,3) {\scriptsize $1$};
	\node at (-5,-1.5) {\scriptsize $2$};
	\node at (-5,3) {\scriptsize $2$};
\end{tikzpicture}
};
\endxy\right)
\left(\,
\xy
(0,0)*{
\bt[scale=0.5]
	\draw (0,0) to (2,0);
	\draw (0,-1) to (2,-1);
	\draw (0,-2) to (3,-2);
	\draw (0,-3) to (3,-3);
	\node at (0.5,-0.5) { $3'$};
	\node at (1.5,-0.5) { $6'$};
	\node at (0.5,-1.5) { $2$};
	\node at (0.5,-2.5) { $1$};
	\node at (1.5,-2.5) { $4'$};
	\node at (2.5,-2.5) { $5$};
\et
};
\endxy\,\right) \ = \ 
\xy
(0,0)*{
\bt[scale=0.5]
	\draw (0,0) to (2,0);
	\draw (0,-1) to (2,-1);
	\draw (0,-2) to (2,-2);
	\draw (0,-3) to (2,-3);
	\draw (0,-4) to (1,-4);
	\node at (0.5,-0.5) { $3'$};
	\node at (1.5,-0.5) { $6'$};
	\node at (0.5,-1.5) { $2$};
	\node at (0.5,-2.5) { $1$};
	\node at (1.5,-2.5) { $4'$};
	\node at (0.5,-3.5) { $5$};
\et
};
\endxy \ + \ 
\xy
(0,0)*{
\bt[scale=0.5]
	\draw (0,0) to (2,0);
	\draw (0,-1) to (2,-1);
	\draw (0,-2) to (2,-2);
	\draw (0,-3) to (2,-3);
	\draw (0,-4) to (1,-4);
	\node at (0.5,-0.5) { $3'$};
	\node at (1.5,-0.5) { $6'$};
	\node at (0.5,-1.5) { $2$};
	\node at (0.5,-2.5) { $1$};
	\node at (1.5,-2.5) { $5$};
	\node at (0.5,-3.5) { $4'$};
\et
};
\endxy \ + \ 
\xy
(0,0)*{
\bt[scale=0.5]
	\draw (0,0) to (2,0);
	\draw (0,-1) to (2,-1);
	\draw (0,-2) to (2,-2);
	\draw (0,-3) to (2,-3);
	\draw (0,-4) to (1,-4);
	\node at (0.5,-0.5) { $3'$};
	\node at (1.5,-0.5) { $6'$};
	\node at (0.5,-1.5) { $2$};
	\node at (0.5,-2.5) { $4'$};
	\node at (1.5,-2.5) { $5$};
	\node at (0.5,-3.5) { $1$};
\et
};
\endxy \ .
\]
Note that by combining \eqref{psi-merge} and \eqref{psi-split} with Definition \ref{crossing-def} we have
\beq\label{psi-crossing}
\psi(s_i)(T)=T_{i\leftrightarrow i+1}
\eeq
for $1\leq i\leq r-1$ and $T\in\T(\omega)$ (with no sign in front of the $T_{i\leftrightarrow i+1}$).


\subsection{A diagrammatic basis of $\Hom_{\H}(M^\lambda,M^\mu)$}


We now prove the second main result of the paper by giving a diagrammatic basis of $\Hom_{\H}(M^\lambda,M^\mu)$ for $\lambda,\mu\in\Lambda'(r)$. We do so by giving a basis of $\a_\mu\W\a_\lambda$ and applying $\psi$. Elements of the latter basis will be of the form
\[
\theta_T:=\theta(w_T)=
\xy
(0,0)*{\rotatebox{180}{
\bt[scale=.35, color=\clr]
	\draw [ thick, rdirected=0.15] (0,1.5) to (0,2.5);
	\draw [ thick, rdirected=0.55] (0,2.5) to [out=30,in=270] (1,4.25);
	\draw [ thick, rdirected=0.55] (0,2.5) to [out=150,in=270] (-1,4.25); 
	\draw [ thick, rdirected=0.15] (5,1.5) to (5,2.5);
	\draw [ thick, rdirected=0.55] (5,2.5) to [out=30,in=270] (6,4.25);
	\draw [ thick, rdirected=0.55] (5,2.5) to [out=150,in=270] (4,4.25); 
	\node at (0,9) {\rotatebox{180}{\scriptsize $\lambda_m$}};
	\node at (5,9) {\rotatebox{180}{\scriptsize $\lambda_1$}};
	\node at (0,0.75) {\rotatebox{180}{\scriptsize $\mu_{m'}$}};
	\node at (5,0.75) {\rotatebox{180}{\scriptsize $\mu_1$}};
	\node at (3.6, 3.75) {\rotatebox{180}{\scriptsize $1$}};
	\node at (6.4, 3.75) {\rotatebox{180}{\scriptsize $1$}};
	\node at (1.4, 3.75) {\rotatebox{180}{\scriptsize $1$}};
	\node at (-1.4, 3.75) {\rotatebox{180}{\scriptsize $1$}};
	\node at (3.6, 6.25) {\rotatebox{180}{\scriptsize $1$}};
	\node at (6.4, 6.25) {\rotatebox{180}{\scriptsize $1$}};
	\node at (1.4, 6.25) {\rotatebox{180}{\scriptsize $1$}};
	\node at (-1.4,6.25) {\rotatebox{180}{\scriptsize $1$}};
	\draw [ thick, rdirected=0.65] (0, 7.5) to (0,8.5);
	\draw [ thick, rdirected=0.55] (1,5.75) to [out=90,in=330] (0,7.5);
	\draw [ thick, rdirected=0.55] (-1,5.75) to [out=90,in=210] (0,7.5);
	\draw [ thick, rdirected=0.65] (5,7.5) to (5,8.5);
	\draw [ thick, rdirected=0.55] (6,5.75) to [out=90,in=330] (5,7.5);
	\draw [ thick, rdirected=0.55] (4,5.75) to [out=90,in=210] (5,7.5);
	\draw [ thick ] (-1.5,5.75) rectangle (6.5,4.25);
	\node at (0.1, 6.25) { $\cdots$};
	\node at (5.1, 6.25) { $\cdots$};
	\node at (0.1, 3.75) { $\cdots$};
	\node at (5.1, 3.75) { $\cdots$};
	\node at (2.6, 7) { $\cdots$};
	\node at (2.6, 2.75) { $\cdots$};
	\node at (2.6, 5) {\rotatebox{180}{ $w_T$}};
\et
}};
\endxy
\]
where $m=l(\lambda)$, $m'=l(\mu)$, and $w_T$ is a certain Sergeev diagram corresponding to a supertabloid $T\in\T(\lambda,\mu)$. 
 
We introduce some useful terminology before describing $w_T$. For $1\leq i\leq m$ we refer to the iterated split
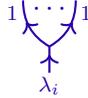
\[
\xy
(0,0)*{\rotatebox{180}{
\bt[scale=.35, color=\clr]
	\node at (5,10.5) {\rotatebox{180}{\scriptsize $\lambda_i$}};
	\node at (3.6, 7.75) {\rotatebox{180}{\scriptsize $1$}};
	\node at (6.4, 7.75) {\rotatebox{180}{\scriptsize $1$}};
	\draw [ thick, rdirected=0.65] (5,9) to (5,10);
	\draw [ thick, rdirected=0.55] (6,7.25) to [out=90,in=330] (5,9);
	\draw [ thick, rdirected=0.55] (4,7.25) to [out=90,in=210] (5,9);
	\node at (5.1, 7.5) { $\cdots$};
\et
}};
\endxy
\]
of $\theta_T$ as the \emph{$i^{\th}$ split}, and for $1\leq j\leq m'$ we refer to the iterated merge
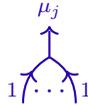
\[
\xy
(0,0)*{\rotatebox{180}{
\bt[scale=.35, color=\clr]
	\draw [ thick, rdirected=0.15] (5,1.5) to (5,2.5);
	\draw [ thick, rdirected=0.55] (5,2.5) to [out=30,in=270] (6,4.25);
	\draw [ thick, rdirected=0.55] (5,2.5) to [out=150,in=270] (4,4.25); 
	\node at (5,0.75) {\rotatebox{180}{\scriptsize $\mu_j$}};
	\node at (3.6, 3.75) {\rotatebox{180}{\scriptsize $1$}};
	\node at (6.4, 3.75) {\rotatebox{180}{\scriptsize $1$}};
	\node at (5.1, 3.75) { $\cdots$};
\et
}};
\endxy
\]
of $\theta_T$ as the \emph{$j^{\th}$ merge}. Taking $w_T$ by itself as a Sergeev diagram, we say a strand of $w_T$ \emph{starts in the $i^{\text{th}}$ grouping} (resp. \emph{ends in the $j^{\text{th}}$ grouping}) if, when $w_T$ is embedded in $\theta_T$, it starts in the $i^{\th}$ split (resp. ends in the $j^{\th}$ merge).

For $T\in\T(\lambda,\mu)$ we construct $w_T$ (and in turn $\theta_T$) as follows. First, for each row of $T$ we rearrange the entries into the increasing order of $D$. Next, we associate to each entry $d$ of $T$ a strand of $w_T$: if $d$ is the $i^{\th}$ entry from the left in the $j^{\th}$ row, then the corresponding strand is the $i^{\th}$ strand from the left among those starting in the $j^{\th}$ grouping. Reading the entries of $T$ from left to right, starting with the first row and working down, we connect the strand corresponding to $d$ to the leftmost available strand among those ending in the $k^{\text{th}}$ grouping if $d\in\{k,k'\}.$ Finally, we place a dot at the top of every strand corresponding to a primed entry. Below is an example for a $T\in\T((2,1,3),(1,3,2))$: 
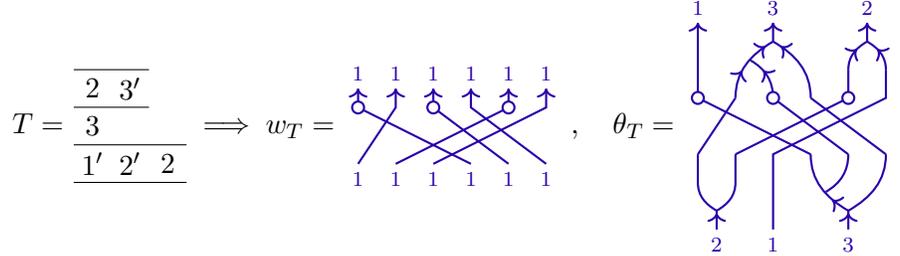
\[
T=
\xy
(0,0)*{
\bt[scale=0.5]
	\draw (0,0) to (2,0);
	\draw (0,-1) to (2,-1);
	\draw (0,-2) to (3,-2);
	\draw (0,-3) to (3,-3);
	\node at (0.5,-0.5) { $2$};
	\node at (1.5,-0.5) { $3'$};
	\node at (0.5,-1.5) { $3$};
	\node at (0.5,-2.5) { $1'$};
	\node at (1.5,-2.5) { $2'$};
	\node at (2.5,-2.5) { $2$};
\et
};
\endxy\implies w_T=
\xy
(0,0)*{
\bt[scale=0.5, color=\clr]
	\draw [thick, directed=1] (0,0) to (1,1.5) to (1,2);
	\draw [thick,directed=1] (1,0) to (4,1.5) to (4,2);
	\draw [thick,directed=1] (2,0) to (5,1.5) to (5,2);
	\draw [thick,directed=1] (3,0) to (0,1.5) to (0,2);
	\draw [thick,directed=1] (4,0) to (2,1.5) to (2,2);
	\draw [thick,directed=1] (5,0) to (3,1.5) to (3,2);
	\draw (0,1.5) \wdot;
	\draw (2,1.5) \wdot;
	\draw (4,1.5) \wdot;
	\node at (0,-0.4) {\ss $1$};
	\node at (1,-0.4) {\ss $1$};
	\node at (2,-0.4) {\ss $1$};
	\node at (3,-0.4) {\ss $1$};
	\node at (4,-0.4) {\ss $1$};
	\node at (5,-0.4) {\ss $1$};
	\node at (0,2.4) {\ss $1$};
	\node at (1,2.4) {\ss $1$};
	\node at (2,2.4) {\ss $1$};
	\node at (3,2.4) {\ss $1$};
	\node at (4,2.4) {\ss $1$};
	\node at (5,2.4) {\ss $1$};
\et
};
\endxy \ ,\quad \theta_T=
\xy
(0,0)*{
\bt[scale=0.5, color=\clr]
	\draw [thick, ] (0,0) to (1,1.5);
	\draw [thick, ] (1,0) to (4,1.5) to (4,2.25);
	\draw [thick, ] (2,-2) to (2,0) to (5,1.5) to (5,2.25);
	\draw [thick,directed=1] (3,0) to (0,1.5) to (0,3.5);
	\draw [thick, ] (4,0) to (2,1.5);
	\draw [thick, ] (5,0) to (3,1.5);
	\node at (0,3.9) {\ss $1$};
	\draw [thick, directed=0.4, directed=0.875] (1,1.5) to [out=90, in=210] (2,3);
	\draw [thick, directed=0.875] (3,1.5) to [out=90, in=330] (2,3);
	\draw [thick, directed=0.65] (2,1.5) to [out=90, in=330] (1.4,2.5);
	\draw [thick, directed=1] (2,3) to (2,3.5);
	\node at (2,3.9) {\ss $3$};
	\draw [thick, directed=0.65] (4,2.25) to [out=90, in=210] (4.5,3);
	\draw [thick, directed=0.65] (5,2.25) to [out=90, in=330] (4.5,3);
	\draw [thick, directed=1] (4.5,3) to (4.5,3.5);
	\node at (4.5,3.9) {\ss $2$};
	\draw [thick, ] (1,0) to (1,-0.75) to [out=270, in=30] (0.5,-1.5);
	\draw [thick, ] (0,0) to (0,-0.75) to [out=270, in=150] (0.5,-1.5);
	\draw [thick, directed=0.875] (0.5,-2) to (0.5,-1.5);
	\node at (0.5,-2.4) {\ss $2$};
	\node at (2,-2.4) {\ss $1$};
	\draw [thick, rdirected=0.85] (3,0) to [out=270, in=150] (4,-1.5);
	\draw [thick, ] (5,0) to [out=270, in=30] (4,-1.5);
	\draw [thick, ] (4,0) to [out=270, in=30] (3.4,-1);
	\draw [thick, directed=0.875] (4,-2) to (4,-1.5);
	\node at (4,-2.4) {\ss $3$};
	\draw (0,1.5) \wdot;
	\draw (2,1.5) \wdot;
	\draw (4,1.5) \wdot;
\et
};
\endxy \ .
\]
It is clear from these constructions that the assignment $T\squig(w_T,\theta_T)$ is well-defined. 

Let $\T_*(\lambda,\mu)$ be the set of $T\in\T(\lambda,\mu)$ with the property that for $1\leq i\leq l(\lambda)$ and $1\leq j\leq l(\mu)$, $j'$ occurs no more than once in the $i^{\th}$ row of $T$. For example, the $T$ pictured above is an element of $\T_*((2,1,3),(1,3,2)).$

\begin{theorem}\label{basis-theorem}
For $\lambda,\mu\in\Lambda'(r)$, the set $\{\theta_T\mid T\in\T_*(\lambda,\mu)\}$ constitutes a basis of $\a_\mu\W\a_\lambda.$
\end{theorem}

\begin{proof}
First we prove span. We claim that for any Sergeev diagram $w$, $\theta(w)$ is equal to either zero or $\theta_T$ for some $T\in\T_*(\lambda,\mu)$, which will imply we have a spanning set by Lemma \ref{beta-theta-maps}. Indeed, Lemma \ref{untangle} ensures that any crossings between strands starting in the same split of $\theta(w)$ can be untied, and ditto for any strands ending in the same merge (possibly in conjunction with \eqref{dots-past-crossings} to slide any intervening dots out of the way). Dots can then be moved to the leftmost 1-strand among those starting in a certain split and ending in a certain merge: one uses \eqref{dots-past-merges} to create a local web in which \eqref{dot-on-k-strand} with $k=1$ may be applied. In the same way, one can use \eqref{2-dots-zero} to show that $\theta(w)=0$  if there are two or more strands which start in the same split, end in the same merge, and carry a dot. This proves the claim.

Next we prove linear independence. For $T\in\T_*(\lambda,\mu)$, we make the following claims about the expression of $\beta(\theta_T)$ in the standard basis of $\a_\omega\W\a_\omega$:
\be
\item the coefficient of $w_T$ is a positive integer, and
\item the coefficient of $w_{T'}$ is zero for all $T'\in\T_*(\lambda,\mu)$ with $T\neq T'$.
\ee
This will imply that $\beta(\sum_{S\in\T_*(\lambda,\mu)}\alpha_S\,\theta_S)=0$ only if the coefficients $\alpha_S\in\C$ are all zero, which suffices for linear independence. By \eqref{clasp-sum} we have
\[
\beta(\theta_T)=\sum_{\substack{\sigma\in\frakS_\lambda\\ \rho\in\frakS_\mu}}\rho\, w_T\,\sigma
\]
where $\frakS_\lambda:=\frakS_{\lambda_1}\times\cdots\times\frakS_{l(\lambda)}$ is the Young subgroup of $\frakS_r$ corresponding to $\lambda$ and similarly for $\frakS_\mu$. Using this both (1) and (2) are easily verified, so we leave it to the reader. The proof is complete.
\end{proof}

For example, the set $\{\theta_T\mid T\in\T_*((2,1,2),(1,3,1))\}$ consists of the webs
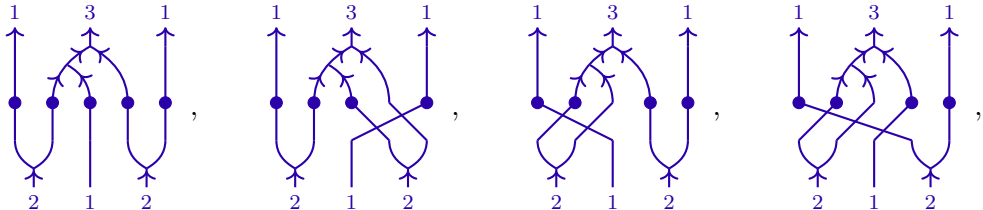
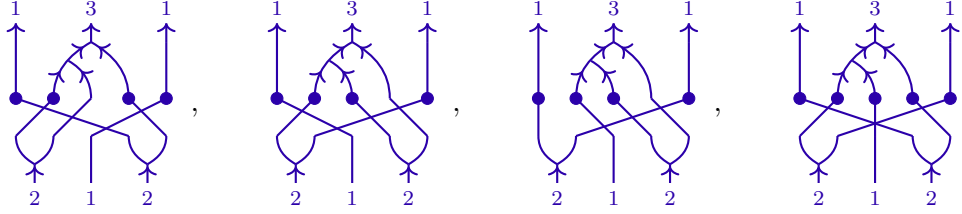
\[
\xy
(0,0)*{
\bt[scale=0.5, color=\clr]
	\draw [thick, ] (0,0.5) to (0,1.5);
	\draw [thick, ] (1,0.5) to (1,1.5);
	\draw [thick, ] (2,0.5) to (2,1.5);
	\draw [thick, ] (3,0.5) to (3,1.5);
	\draw [thick, ] (4,0.5) to (4,1.5);
	\draw [thick, ] (4,1.5) to (4,3);
	\draw [thick, ] (2,-0.75) to (2,0.5);
	\draw [thick,directed=1] (0,1.5) to (0,3.5);
	\node at (0,3.9) {\ss $1$};
	\draw [thick, directed=0.4, directed=0.875] (1,1.5) to [out=90, in=210] (2,3);
	\draw [thick, directed=0.875] (3,1.5) to [out=90, in=330] (2,3);
	\draw [thick, directed=0.65] (2,1.5) to [out=90, in=330] (1.4,2.5);
	\draw [thick, directed=1] (2,3) to (2,3.5);
	\node at (2,3.9) {\ss $3$};
	\draw [thick, directed=1] (4,3) to (4,3.5);
	\node at (4,3.9) {\ss $1$};
	\draw [thick, ] (1,0.5) to [out=270, in=30] (0.5,-0.25);
	\draw [thick, ] (0,0.5) to [out=270, in=150] (0.5,-0.25);
	\draw [thick, directed=0.875] (0.5,-0.75) to (0.5,-0.25);
	\node at (0.5,-1.15) {\ss $2$};
	\node at (2,-1.15) {\ss $1$};
	\draw [thick, ] (4,0.5) to [out=270, in=30] (3.5,-0.25);
	\draw [thick, ] (3,0.5) to [out=270, in=150] (3.5,-0.25);
	\draw [thick, directed=0.875] (3.5,-0.75) to (3.5,-0.25);
	\node at (3.5,-1.15) {\ss $2$};
	\filldraw
	(0,1.5) circle (4.5pt)
	(1,1.5) circle (4.5pt)
	(2,1.5) circle (4.5pt)
	(3,1.5) circle (4.5pt)
	(4,1.5) circle (4.5pt);
\et
};
\endxy \ ,\quad\quad
\xy
(0,0)*{
\bt[scale=0.5, color=\clr]
	\draw [thick, ] (0,0.5) to (0,1.5);
	\draw [thick, ] (1,0.5) to (1,1.5);
	\draw [thick, ] (2,0.5) to (4,1.5);
	\draw [thick, ] (3,0.5) to (2,1.5);
	\draw [thick, ] (4,0.5) to (3,1.5);
	\draw [thick, ] (4,1.5) to (4,3);
	\draw [thick, ] (2,-0.75) to (2,0.5);
	\draw [thick,directed=1] (0,1.5) to (0,3.5);
	\node at (0,3.9) {\ss $1$};
	\draw [thick, directed=0.4, directed=0.875] (1,1.5) to [out=90, in=210] (2,3);
	\draw [thick, directed=0.875] (3,1.5) to [out=90, in=330] (2,3);
	\draw [thick, directed=0.65] (2,1.5) to [out=90, in=330] (1.4,2.5);
	\draw [thick, directed=1] (2,3) to (2,3.5);
	\node at (2,3.9) {\ss $3$};
	\draw [thick, directed=1] (4,3) to (4,3.5);
	\node at (4,3.9) {\ss $1$};
	\draw [thick, ] (1,0.5) to [out=270, in=30] (0.5,-0.25);
	\draw [thick, ] (0,0.5) to [out=270, in=150] (0.5,-0.25);
	\draw [thick, directed=0.875] (0.5,-0.75) to (0.5,-0.25);
	\node at (0.5,-1.15) {\ss $2$};
	\node at (2,-1.15) {\ss $1$};
	\draw [thick, ] (4,0.5) to [out=270, in=30] (3.5,-0.25);
	\draw [thick, ] (3,0.5) to [out=270, in=150] (3.5,-0.25);
	\draw [thick, directed=0.875] (3.5,-0.75) to (3.5,-0.25);
	\node at (3.5,-1.15) {\ss $2$};
	\filldraw
	(0,1.5) circle (4.5pt)
	(1,1.5) circle (4.5pt)
	(2,1.5) circle (4.5pt)
	(4,1.5) circle (4.5pt);
\et
};
\endxy \ ,\quad\quad
\xy
(0,0)*{
\bt[scale=0.5, color=\clr]
	\draw [thick, ] (0,0.5) to (1,1.5);
	\draw [thick, ] (1,0.5) to (2,1.5);
	\draw [thick, ] (2,0.5) to (0,1.5);
	\draw [thick, ] (3,0.5) to (3,1.5);
	\draw [thick, ] (4,0.5) to (4,1.5);
	\draw [thick, ] (4,1.5) to (4,3);
	\draw [thick, ] (2,-0.75) to (2,0.5);
	\draw [thick,directed=1] (0,1.5) to (0,3.5);
	\node at (0,3.9) {\ss $1$};
	\draw [thick, directed=0.4, directed=0.875] (1,1.5) to [out=90, in=210] (2,3);
	\draw [thick, directed=0.875] (3,1.5) to [out=90, in=330] (2,3);
	\draw [thick, directed=0.65] (2,1.5) to [out=90, in=330] (1.4,2.5);
	\draw [thick, directed=1] (2,3) to (2,3.5);
	\node at (2,3.9) {\ss $3$};
	\draw [thick, directed=1] (4,3) to (4,3.5);
	\node at (4,3.9) {\ss $1$};
	\draw [thick, ] (1,0.5) to [out=270, in=30] (0.5,-0.25);
	\draw [thick, ] (0,0.5) to [out=270, in=150] (0.5,-0.25);
	\draw [thick, directed=0.875] (0.5,-0.75) to (0.5,-0.25);
	\node at (0.5,-1.15) {\ss $2$};
	\node at (2,-1.15) {\ss $1$};
	\draw [thick, ] (4,0.5) to [out=270, in=30] (3.5,-0.25);
	\draw [thick, ] (3,0.5) to [out=270, in=150] (3.5,-0.25);
	\draw [thick, directed=0.875] (3.5,-0.75) to (3.5,-0.25);
	\node at (3.5,-1.15) {\ss $2$};
	\filldraw
	(0,1.5) circle (4.5pt)
	(1,1.5) circle (4.5pt)
	(3,1.5) circle (4.5pt)
	(4,1.5) circle (4.5pt);
\et
};
\endxy \ ,\quad\quad
\xy
(0,0)*{
\bt[scale=0.5, color=\clr]
	\draw [thick, ] (0,0.5) to (1,1.5);
	\draw [thick, ] (1,0.5) to (2,1.5);
	\draw [thick, ] (2,0.5) to (3,1.5);
	\draw [thick, ] (3,0.5) to (0,1.5);
	\draw [thick, ] (4,0.5) to (4,1.5);
	\draw [thick, ] (4,1.5) to (4,3);
	\draw [thick, ] (2,-0.75) to (2,0.5);
	\draw [thick,directed=1] (0,1.5) to (0,3.5);
	\node at (0,3.9) {\ss $1$};
	\draw [thick, directed=0.4, directed=0.875] (1,1.5) to [out=90, in=210] (2,3);
	\draw [thick, directed=0.875] (3,1.5) to [out=90, in=330] (2,3);
	\draw [thick, directed=0.65] (2,1.5) to [out=90, in=330] (1.4,2.5);
	\draw [thick, directed=1] (2,3) to (2,3.5);
	\node at (2,3.9) {\ss $3$};
	\draw [thick, directed=1] (4,3) to (4,3.5);
	\node at (4,3.9) {\ss $1$};
	\draw [thick, ] (1,0.5) to [out=270, in=30] (0.5,-0.25);
	\draw [thick, ] (0,0.5) to [out=270, in=150] (0.5,-0.25);
	\draw [thick, directed=0.875] (0.5,-0.75) to (0.5,-0.25);
	\node at (0.5,-1.15) {\ss $2$};
	\node at (2,-1.15) {\ss $1$};
	\draw [thick, ] (4,0.5) to [out=270, in=30] (3.5,-0.25);
	\draw [thick, ] (3,0.5) to [out=270, in=150] (3.5,-0.25);
	\draw [thick, directed=0.875] (3.5,-0.75) to (3.5,-0.25);
	\node at (3.5,-1.15) {\ss $2$};
	\filldraw
	(0,1.5) circle (4.5pt)
	(1,1.5) circle (4.5pt)
	(3,1.5) circle (4.5pt)
	(4,1.5) circle (4.5pt);
\et
};
\endxy \ ,
\]
\[
\xy
(0,0)*{
\bt[scale=0.5, color=\clr]
	\draw [thick, ] (0,0.5) to (1,1.5);
	\draw [thick, ] (1,0.5) to (2,1.5);
	\draw [thick, ] (2,0.5) to (4,1.5);
	\draw [thick, ] (3,0.5) to (0,1.5);
	\draw [thick, ] (4,0.5) to (3,1.5);
	\draw [thick, ] (4,1.5) to (4,3);
	\draw [thick, ] (2,-0.75) to (2,0.5);
	\draw [thick,directed=1] (0,1.5) to (0,3.5);
	\node at (0,3.9) {\ss $1$};
	\draw [thick, directed=0.4, directed=0.875] (1,1.5) to [out=90, in=210] (2,3);
	\draw [thick, directed=0.875] (3,1.5) to [out=90, in=330] (2,3);
	\draw [thick, directed=0.65] (2,1.5) to [out=90, in=330] (1.4,2.5);
	\draw [thick, directed=1] (2,3) to (2,3.5);
	\node at (2,3.9) {\ss $3$};
	\draw [thick, directed=1] (4,3) to (4,3.5);
	\node at (4,3.9) {\ss $1$};
	\draw [thick, ] (1,0.5) to [out=270, in=30] (0.5,-0.25);
	\draw [thick, ] (0,0.5) to [out=270, in=150] (0.5,-0.25);
	\draw [thick, directed=0.875] (0.5,-0.75) to (0.5,-0.25);
	\node at (0.5,-1.15) {\ss $2$};
	\node at (2,-1.15) {\ss $1$};
	\draw [thick, ] (4,0.5) to [out=270, in=30] (3.5,-0.25);
	\draw [thick, ] (3,0.5) to [out=270, in=150] (3.5,-0.25);
	\draw [thick, directed=0.875] (3.5,-0.75) to (3.5,-0.25);
	\node at (3.5,-1.15) {\ss $2$};
	\filldraw
	(0,1.5) circle (4.5pt)
	(1,1.5) circle (4.5pt)
	(3,1.5) circle (4.5pt)
	(4,1.5) circle (4.5pt);
\et
};
\endxy \ ,\quad\quad
\xy
(0,0)*{
\bt[scale=0.5, color=\clr]
	\draw [thick, ] (0,0.5) to (1,1.5);
	\draw [thick, ] (1,0.5) to (4,1.5);
	\draw [thick, ] (2,0.5) to (0,1.5);
	\draw [thick, ] (3,0.5) to (2,1.5);
	\draw [thick, ] (4,0.5) to (3,1.5);
	\draw [thick, ] (4,1.5) to (4,3);
	\draw [thick, ] (2,-0.75) to (2,0.5);
	\draw [thick,directed=1] (0,1.5) to (0,3.5);
	\node at (0,3.9) {\ss $1$};
	\draw [thick, directed=0.4, directed=0.875] (1,1.5) to [out=90, in=210] (2,3);
	\draw [thick, directed=0.875] (3,1.5) to [out=90, in=330] (2,3);
	\draw [thick, directed=0.65] (2,1.5) to [out=90, in=330] (1.4,2.5);
	\draw [thick, directed=1] (2,3) to (2,3.5);
	\node at (2,3.9) {\ss $3$};
	\draw [thick, directed=1] (4,3) to (4,3.5);
	\node at (4,3.9) {\ss $1$};
	\draw [thick, ] (1,0.5) to [out=270, in=30] (0.5,-0.25);
	\draw [thick, ] (0,0.5) to [out=270, in=150] (0.5,-0.25);
	\draw [thick, directed=0.875] (0.5,-0.75) to (0.5,-0.25);
	\node at (0.5,-1.15) {\ss $2$};
	\node at (2,-1.15) {\ss $1$};
	\draw [thick, ] (4,0.5) to [out=270, in=30] (3.5,-0.25);
	\draw [thick, ] (3,0.5) to [out=270, in=150] (3.5,-0.25);
	\draw [thick, directed=0.875] (3.5,-0.75) to (3.5,-0.25);
	\node at (3.5,-1.15) {\ss $2$};
	\filldraw
	(0,1.5) circle (4.5pt)
	(1,1.5) circle (4.5pt)
	(2,1.5) circle (4.5pt)
	(4,1.5) circle (4.5pt);
\et
};
\endxy \ ,\quad\quad
\xy
(0,0)*{
\bt[scale=0.5, color=\clr]
	\draw [thick, ] (0,0.5) to (0,1.5);
	\draw [thick, ] (1,0.5) to (4,1.5);
	\draw [thick, ] (2,0.5) to (1,1.5);
	\draw [thick, ] (3,0.5) to (2,1.5);
	\draw [thick, ] (4,0.5) to (3,1.5);
	\draw [thick, ] (4,1.5) to (4,3);
	\draw [thick, ] (2,-0.75) to (2,0.5);
	\draw [thick,directed=1] (0,1.5) to (0,3.5);
	\node at (0,3.9) {\ss $1$};
	\draw [thick, directed=0.4, directed=0.875] (1,1.5) to [out=90, in=210] (2,3);
	\draw [thick, directed=0.875] (3,1.5) to [out=90, in=330] (2,3);
	\draw [thick, directed=0.65] (2,1.5) to [out=90, in=330] (1.4,2.5);
	\draw [thick, directed=1] (2,3) to (2,3.5);
	\node at (2,3.9) {\ss $3$};
	\draw [thick, directed=1] (4,3) to (4,3.5);
	\node at (4,3.9) {\ss $1$};
	\draw [thick, ] (1,0.5) to [out=270, in=30] (0.5,-0.25);
	\draw [thick, ] (0,0.5) to [out=270, in=150] (0.5,-0.25);
	\draw [thick, directed=0.875] (0.5,-0.75) to (0.5,-0.25);
	\node at (0.5,-1.15) {\ss $2$};
	\node at (2,-1.15) {\ss $1$};
	\draw [thick, ] (4,0.5) to [out=270, in=30] (3.5,-0.25);
	\draw [thick, ] (3,0.5) to [out=270, in=150] (3.5,-0.25);
	\draw [thick, directed=0.875] (3.5,-0.75) to (3.5,-0.25);
	\node at (3.5,-1.15) {\ss $2$};
	\filldraw
	(0,1.5) circle (4.5pt)
	(1,1.5) circle (4.5pt)
	(2,1.5) circle (4.5pt)
	(4,1.5) circle (4.5pt);
\et
};
\endxy \ ,\quad\quad
\xy
(0,0)*{
\bt[scale=0.5, color=\clr]
	\draw [thick, ] (0,0.5) to (1,1.5);
	\draw [thick, ] (1,0.5) to (4,1.5);
	\draw [thick, ] (2,0.5) to (2,1.5);
	\draw [thick, ] (3,0.5) to (0,1.5);
	\draw [thick, ] (4,0.5) to (3,1.5);
	\draw [thick, ] (4,1.5) to (4,3);
	\draw [thick, ] (2,-0.75) to (2,0.5);
	\draw [thick,directed=1] (0,1.5) to (0,3.5);
	\node at (0,3.9) {\ss $1$};
	\draw [thick, directed=0.4, directed=0.875] (1,1.5) to [out=90, in=210] (2,3);
	\draw [thick, directed=0.875] (3,1.5) to [out=90, in=330] (2,3);
	\draw [thick, directed=0.65] (2,1.5) to [out=90, in=330] (1.4,2.5);
	\draw [thick, directed=1] (2,3) to (2,3.5);
	\node at (2,3.9) {\ss $3$};
	\draw [thick, directed=1] (4,3) to (4,3.5);
	\node at (4,3.9) {\ss $1$};
	\draw [thick, ] (1,0.5) to [out=270, in=30] (0.5,-0.25);
	\draw [thick, ] (0,0.5) to [out=270, in=150] (0.5,-0.25);
	\draw [thick, directed=0.875] (0.5,-0.75) to (0.5,-0.25);
	\node at (0.5,-1.15) {\ss $2$};
	\node at (2,-1.15) {\ss $1$};
	\draw [thick, ] (4,0.5) to [out=270, in=30] (3.5,-0.25);
	\draw [thick, ] (3,0.5) to [out=270, in=150] (3.5,-0.25);
	\draw [thick, directed=0.875] (3.5,-0.75) to (3.5,-0.25);
	\node at (3.5,-1.15) {\ss $2$};
	\filldraw
	(0,1.5) circle (4.5pt)
	(1,1.5) circle (4.5pt)
	(2,1.5) circle (4.5pt)
	(3,1.5) circle (4.5pt)
	(4,1.5) circle (4.5pt);
\et
};
\endxy
\]
where the symbol
\(
\xy
(0,0)*{
\bt[scale=0.5, color=\clr]
	\filldraw
	(0,0) circle (4.5pt);
\et
};
\endxy
\)
denotes a location where a dot is permissible.\footnote{One can resolve the triple intersection in the last web using the second relation of \eqref{braid-relation}.} Since the existence of each dot is independent of the others, the set has cardinality $6(2^4)+2(2^5)=160$.

Combining Theorems \ref{psi-map} and \ref{basis-theorem}, we have:

\begin{corollary}\label{basis-corollary}
For $\lambda,\mu\in\Lambda'(r)$, the set $\{\psi(\theta_T)\mid T\in\T_*(\lambda,\mu)\}$ forms a basis of $\Hom_{\H}(M^\lambda,M^\mu)$.
\end{corollary}

The interested reader can use equations \eqref{psi-dot}-\eqref{psi-crossing} to deduce explicit formulas for the $\H$-morphisms $\psi(\theta_T)$ in terms of supertabloids. For example, if 
\[
T=\xy
(0,0)*{
\bt[scale=0.5]
	\draw (0,0) to (2,0);
	\draw (0,-1) to (2,-1);
	\draw (0,-2) to (2,-2);
	\draw (0,-3) to (2,-3);
	\node at (0.5,-0.5) { $2$};
	\node at (1.5,-0.5) { $2$};
	\node at (0.5,-1.5) { $3'$};
	\node at (0.5,-2.5) { $1'$};
	\node at (1.5,-2.5) { $2'$};
\et
};
\endxy\in\T_*((2,1,2),(1,3,1))
\]
then we have
\[
\theta_T=\xy
(0,0)*{
\bt[scale=0.5, color=\clr]
	\draw [thick, ] (0,0.5) to (1,1.5);
	\draw [thick, ] (1,0.5) to (2,1.5);
	\draw [thick, ] (2,0.5) to (4,1.5);
	\draw [thick, ] (3,0.5) to (0,1.5);
	\draw [thick, ] (4,0.5) to (3,1.5);
	\draw [thick, ] (4,1.5) to (4,3);
	\draw [thick, ] (2,-0.75) to (2,0.5);
	\draw [thick,directed=1] (0,1.5) to (0,3.5);
	\node at (0,3.9) {\ss $1$};
	\draw [thick, directed=0.4, directed=0.875] (1,1.5) to [out=90, in=210] (2,3);
	\draw [thick, directed=0.875] (3,1.5) to [out=90, in=330] (2,3);
	\draw [thick, directed=0.65] (2,1.5) to [out=90, in=330] (1.4,2.5);
	\draw [thick, directed=1] (2,3) to (2,3.5);
	\node at (2,3.9) {\ss $3$};
	\draw [thick, directed=1] (4,3) to (4,3.5);
	\node at (4,3.9) {\ss $1$};
	\draw [thick, ] (1,0.5) to [out=270, in=30] (0.5,-0.25);
	\draw [thick, ] (0,0.5) to [out=270, in=150] (0.5,-0.25);
	\draw [thick, directed=0.875] (0.5,-0.75) to (0.5,-0.25);
	\node at (0.5,-1.15) {\ss $2$};
	\node at (2,-1.15) {\ss $1$};
	\draw [thick, ] (4,0.5) to [out=270, in=30] (3.5,-0.25);
	\draw [thick, ] (3,0.5) to [out=270, in=150] (3.5,-0.25);
	\draw [thick, directed=0.875] (3.5,-0.75) to (3.5,-0.25);
	\node at (3.5,-1.15) {\ss $2$};
	\draw (0,1.5) \wdot;
	\draw (3,1.5) \wdot;
	\draw (4,1.5) \wdot;
\et
};
\endxy,
\]
\bea
\psi(\theta_T)\left(\,
\xy
(0,0)*{
\bt[scale=0.5]
	\draw (0,0) to (2,0);
	\draw (0,-1) to (2,-1);
	\draw (0,-2) to (2,-2);
	\draw (0,-3) to (2,-3);
	\node at (0.5,-0.5) { $2$};
	\node at (1.5,-0.5) { $4'$};
	\node at (0.5,-1.5) { $3'$};
	\node at (0.5,-2.5) { $1$};
	\node at (1.5,-2.5) { $5$};
\et
};
\endxy\,\right) & = & \psi\left(
\xy
(0,0)*{
\bt[scale=0.5, color=\clr]
	\draw [thick, ] (0,0.5) to (1,1.5);
	\draw [thick, ] (1,0.5) to (2,1.5);
	\draw [thick, ] (2,0.5) to (4,1.5);
	\draw [thick, ] (3,0.5) to (0,1.5);
	\draw [thick, ] (4,0.5) to (3,1.5);
	\draw [thick, ] (4,1.5) to (4,3);
	\draw [thick,directed=1] (0,1.5) to (0,3.5);
	\node at (0,3.9) {\ss $1$};
	\draw [thick, directed=0.4, directed=0.875] (1,1.5) to [out=90, in=210] (2,3);
	\draw [thick, directed=0.875] (3,1.5) to [out=90, in=330] (2,3);
	\draw [thick, directed=0.65] (2,1.5) to [out=90, in=330] (1.4,2.5);
	\draw [thick, directed=1] (2,3) to (2,3.5);
	\node at (2,3.9) {\ss $3$};
	\draw [thick, directed=1] (4,3) to (4,3.5);
	\node at (4,3.9) {\ss $1$};
	\draw (0,1.5) \wdot;
	\draw (3,1.5) \wdot;
	\draw (4,1.5) \wdot;
	\node at (0,0.1) {\ss $1$};
	\node at (1,0.1) {\ss $1$};
	\node at (2,0.1) {\ss $1$};
	\node at (3,0.1) {\ss $1$};
	\node at (4,0.1) {\ss $1$};
\et
};
\endxy\right)\left(\,
\xy
(0,0)*{
\bt[scale=0.5]
	\draw (0,0) to (1,0);
	\draw (0,-1) to (1,-1);
	\draw (0,-2) to (1,-2);
	\draw (0,-3) to (1,-3);
	\draw (0,-4) to (1,-4);
	\draw (0,-5) to (1,-5);
	\node at (0.5,-0.5) { $2$};
	\node at (0.5,-1.5) { $4'$};
	\node at (0.5,-2.5) { $3'$};
	\node at (0.5,-3.5) { $1$};
	\node at (0.5,-4.5) { $5$};
\et
};
\endxy \ + \ 
\xy
(0,0)*{
\bt[scale=0.5]
	\draw (0,0) to (1,0);
	\draw (0,-1) to (1,-1);
	\draw (0,-2) to (1,-2);
	\draw (0,-3) to (1,-3);
	\draw (0,-4) to (1,-4);
	\draw (0,-5) to (1,-5);
	\node at (0.5,-0.5) { $2$};
	\node at (0.5,-1.5) { $4'$};
	\node at (0.5,-2.5) { $3'$};
	\node at (0.5,-3.5) { $5$};
	\node at (0.5,-4.5) { $1$};
\et
};
\endxy \ + \ 
\xy
(0,0)*{
\bt[scale=0.5]
	\draw (0,0) to (1,0);
	\draw (0,-1) to (1,-1);
	\draw (0,-2) to (1,-2);
	\draw (0,-3) to (1,-3);
	\draw (0,-4) to (1,-4);
	\draw (0,-5) to (1,-5);
	\node at (0.5,-0.5) { $4'$};
	\node at (0.5,-1.5) { $2$};
	\node at (0.5,-2.5) { $3'$};
	\node at (0.5,-3.5) { $1$};
	\node at (0.5,-4.5) { $5$};
\et
};
\endxy \ + \ 
\xy
(0,0)*{
\bt[scale=0.5]
	\draw (0,0) to (1,0);
	\draw (0,-1) to (1,-1);
	\draw (0,-2) to (1,-2);
	\draw (0,-3) to (1,-3);
	\draw (0,-4) to (1,-4);
	\draw (0,-5) to (1,-5);
	\node at (0.5,-0.5) { $4'$};
	\node at (0.5,-1.5) { $2$};
	\node at (0.5,-2.5) { $3'$};
	\node at (0.5,-3.5) { $5$};
	\node at (0.5,-4.5) { $1$};
\et
};
\endxy\,\right)\\[2ex]
&=& \psi\left(
\xy
(0,0)*{
\bt[scale=0.5, color=\clr]
	\draw [thick, ] (0,1) to (0,1.5);
	\draw [thick, ] (1,1) to (1,1.5);
	\draw [thick, ] (2,1) to (2,1.5);
	\draw [thick, ] (3,1) to (3,1.5);
	\draw [thick, ] (4,1) to (4,1.5);
	\draw [thick, ] (4,1.5) to (4,3);
	\draw [thick,directed=1] (0,1.5) to (0,3.5);
	\node at (0,3.9) {\ss $1$};
	\draw [thick, directed=0.4, directed=0.875] (1,1.5) to [out=90, in=210] (2,3);
	\draw [thick, directed=0.875] (3,1.5) to [out=90, in=330] (2,3);
	\draw [thick, directed=0.65] (2,1.5) to [out=90, in=330] (1.4,2.5);
	\draw [thick, directed=1] (2,3) to (2,3.5);
	\node at (2,3.9) {\ss $3$};
	\draw [thick, directed=1] (4,3) to (4,3.5);
	\node at (4,3.9) {\ss $1$};
	\draw (0,1.5) \wdot;
	\draw (3,1.5) \wdot;
	\draw (4,1.5) \wdot;
	\node at (0,0.6) {\ss $1$};
	\node at (1,0.6) {\ss $1$};
	\node at (2,0.6) {\ss $1$};
	\node at (3,0.6) {\ss $1$};
	\node at (4,0.6) {\ss $1$};
\et
};
\endxy
\right)\left(\,
\xy
(0,0)*{
\bt[scale=0.5]
	\draw (0,0) to (1,0);
	\draw (0,-1) to (1,-1);
	\draw (0,-2) to (1,-2);
	\draw (0,-3) to (1,-3);
	\draw (0,-4) to (1,-4);
	\draw (0,-5) to (1,-5);
	\node at (0.5,-0.5) { $1$};
	\node at (0.5,-1.5) { $2$};
	\node at (0.5,-2.5) { $4'$};
	\node at (0.5,-3.5) { $5$};
	\node at (0.5,-4.5) { $3'$};
\et
};
\endxy \ + \ 
\xy
(0,0)*{
\bt[scale=0.5]
	\draw (0,0) to (1,0);
	\draw (0,-1) to (1,-1);
	\draw (0,-2) to (1,-2);
	\draw (0,-3) to (1,-3);
	\draw (0,-4) to (1,-4);
	\draw (0,-5) to (1,-5);
	\node at (0.5,-0.5) { $5$};
	\node at (0.5,-1.5) { $2$};
	\node at (0.5,-2.5) { $4'$};
	\node at (0.5,-3.5) { $1$};
	\node at (0.5,-4.5) { $3'$};
\et
};
\endxy \ + \ 
\xy
(0,0)*{
\bt[scale=0.5]
	\draw (0,0) to (1,0);
	\draw (0,-1) to (1,-1);
	\draw (0,-2) to (1,-2);
	\draw (0,-3) to (1,-3);
	\draw (0,-4) to (1,-4);
	\draw (0,-5) to (1,-5);
	\node at (0.5,-0.5) { $1$};
	\node at (0.5,-1.5) { $4'$};
	\node at (0.5,-2.5) { $2$};
	\node at (0.5,-3.5) { $5$};
	\node at (0.5,-4.5) { $3'$};
\et
};
\endxy \ + \ 
\xy
(0,0)*{
\bt[scale=0.5]
	\draw (0,0) to (1,0);
	\draw (0,-1) to (1,-1);
	\draw (0,-2) to (1,-2);
	\draw (0,-3) to (1,-3);
	\draw (0,-4) to (1,-4);
	\draw (0,-5) to (1,-5);
	\node at (0.5,-0.5) { $5$};
	\node at (0.5,-1.5) { $4'$};
	\node at (0.5,-2.5) { $2$};
	\node at (0.5,-3.5) { $1$};
	\node at (0.5,-4.5) { $3'$};
\et
};
\endxy
\,\right)\\[2ex]
&=&
\psi\left(
\xy
(0,0)*{
\bt[scale=0.5, color=\clr]
	\draw [thick, ] (4,1.5) to (4,3);
	\draw [thick,directed=1] (0,1.5) to (0,3.5);
	\node at (0,3.9) {\ss $1$};
	\draw [thick, directed=0.4, directed=0.875] (1,1.5) to [out=90, in=210] (2,3);
	\draw [thick, directed=0.875] (3,1.5) to [out=90, in=330] (2,3);
	\draw [thick, directed=0.65] (2,1.5) to [out=90, in=330] (1.4,2.5);
	\draw [thick, directed=1] (2,3) to (2,3.5);
	\node at (2,3.9) {\ss $3$};
	\draw [thick, directed=1] (4,3) to (4,3.5);
	\node at (4,3.9) {\ss $1$};
	\node at (0,1.1) {\ss $1$};
	\node at (1,1.1) {\ss $1$};
	\node at (2,1.1) {\ss $1$};
	\node at (3,1.1) {\ss $1$};
	\node at (4,1.1) {\ss $1$};
\et
};
\endxy
\right)\left(- \ 
\xy
(0,0)*{
\bt[scale=0.5]
	\draw (0,0) to (1,0);
	\draw (0,-1) to (1,-1);
	\draw (0,-2) to (1,-2);
	\draw (0,-3) to (1,-3);
	\draw (0,-4) to (1,-4);
	\draw (0,-5) to (1,-5);
	\node at (0.5,-0.5) { $1'$};
	\node at (0.5,-1.5) { $2$};
	\node at (0.5,-2.5) { $4'$};
	\node at (0.5,-3.5) { $5'$};
	\node at (0.5,-4.5) { $3$};
\et
};
\endxy \ + \ 
\xy
(0,0)*{
\bt[scale=0.5]
	\draw (0,0) to (1,0);
	\draw (0,-1) to (1,-1);
	\draw (0,-2) to (1,-2);
	\draw (0,-3) to (1,-3);
	\draw (0,-4) to (1,-4);
	\draw (0,-5) to (1,-5);
	\node at (0.5,-0.5) { $5'$};
	\node at (0.5,-1.5) { $2$};
	\node at (0.5,-2.5) { $4'$};
	\node at (0.5,-3.5) { $1'$};
	\node at (0.5,-4.5) { $3$};
\et
};
\endxy \ - \ 
\xy
(0,0)*{
\bt[scale=0.5]
	\draw (0,0) to (1,0);
	\draw (0,-1) to (1,-1);
	\draw (0,-2) to (1,-2);
	\draw (0,-3) to (1,-3);
	\draw (0,-4) to (1,-4);
	\draw (0,-5) to (1,-5);
	\node at (0.5,-0.5) { $1'$};
	\node at (0.5,-1.5) { $4'$};
	\node at (0.5,-2.5) { $2$};
	\node at (0.5,-3.5) { $5'$};
	\node at (0.5,-4.5) { $3$};
\et
};
\endxy \ + \ 
\xy
(0,0)*{
\bt[scale=0.5]
	\draw (0,0) to (1,0);
	\draw (0,-1) to (1,-1);
	\draw (0,-2) to (1,-2);
	\draw (0,-3) to (1,-3);
	\draw (0,-4) to (1,-4);
	\draw (0,-5) to (1,-5);
	\node at (0.5,-0.5) { $5'$};
	\node at (0.5,-1.5) { $4'$};
	\node at (0.5,-2.5) { $2$};
	\node at (0.5,-3.5) { $1'$};
	\node at (0.5,-4.5) { $3$};
\et
};
\endxy
\,\right)\\[2ex]
&=& - \ 
\xy
(0,0)*{
\bt[scale=0.5]
	\draw (0,0) to (1,0);
	\draw (0,-1) to (3,-1);
	\draw (0,-2) to (3,-2);
	\draw (0,-3) to (1,-3);
	\node at (0.5,-0.5) { $1'$};
	\node at (0.5,-1.5) { $2$};
	\node at (1.5,-1.5) { $4'$};
	\node at (2.5,-1.5) { $5'$};
	\node at (0.5,-2.5) { $3$};
\et
};
\endxy \ + \ 
\xy
(0,0)*{
\bt[scale=0.5]
	\draw (0,0) to (1,0);
	\draw (0,-1) to (3,-1);
	\draw (0,-2) to (3,-2);
	\draw (0,-3) to (1,-3);
	\node at (0.5,-0.5) { $5'$};
	\node at (0.5,-1.5) { $1'$};
	\node at (1.5,-1.5) { $2$};
	\node at (2.5,-1.5) { $4'$};
	\node at (0.5,-2.5) { $3$};
\et
};
\endxy \ - \ 
\xy
(0,0)*{
\bt[scale=0.5]
	\draw (0,0) to (1,0);
	\draw (0,-1) to (3,-1);
	\draw (0,-2) to (3,-2);
	\draw (0,-3) to (1,-3);
	\node at (0.5,-0.5) { $1'$};
	\node at (0.5,-1.5) { $2$};
	\node at (1.5,-1.5) { $4'$};
	\node at (2.5,-1.5) { $5'$};
	\node at (0.5,-2.5) { $3$};
\et
};
\endxy \ + \ 
\xy
(0,0)*{
\bt[scale=0.5]
	\draw (0,0) to (1,0);
	\draw (0,-1) to (3,-1);
	\draw (0,-2) to (3,-2);
	\draw (0,-3) to (1,-3);
	\node at (0.5,-0.5) { $5'$};
	\node at (0.5,-1.5) { $1'$};
	\node at (1.5,-1.5) { $2$};
	\node at (2.5,-1.5) { $4'$};
	\node at (0.5,-2.5) { $3$};
\et
};
\endxy \ .
\eea


\end{document}